\documentclass[a4paper,10pt]{amsart}

\usepackage{mathtools}
\usepackage{amssymb}
\usepackage{amsthm}
\usepackage{enumerate}
\usepackage{tikz-cd}
\usepackage{stmaryrd}
\usepackage{mathrsfs}
\usepackage[english]{babel}
\usepackage[T1]{fontenc}
\usepackage[utf8]{inputenc}
\usepackage{csquotes}
\usepackage[doi=false,isbn=false,url=false]{biblatex}
\usepackage{hyperref}

\tikzcdset{every label/.append style = {font = \small}}

\DeclareFieldFormat[article,periodical]{number}{\bibstring{number}~#1}

\renewbibmacro*{journal+issuetitle}{%
  \usebibmacro{journal}%
  \setunit*{\addspace}%
  \iffieldundef{series}
    {}
    {\newunit
     \printfield{series}%
     \setunit{\addspace}}%
  \printfield{volume}%
  \setunit{\addspace}%
  \usebibmacro{issue+date}%
  \setunit{\addcomma\space}%
  \printfield{number}%
  \setunit{\addcolon\space}%
  \usebibmacro{issue}%
  \setunit{\addcomma\space}%
  \printfield{eid}
  \newunit}

\renewbibmacro{in:}{}

\newcommand\piaut[1]{\mathop{\mathcal{E}}(#1)}
\newcommand\piautp[1]{\mathop{\mathcal{E}_*}(#1)}
\newcommand\piauto[1]{\mathop{\mathcal{E}^+}(#1)}
\newcommand\deckgrp[2]{\mathop{\mathcal{E}_{#1}}(#2)}

\newcommand\Baut[2]{B\operatorname{aut}_{#1}(#2)}
\newcommand\Bautfunc[2]{B_{#1}\aut(#2)}
\newcommand\BSaut[2]{B\operatorname{aut}_{#1}^+(#2)}
\newcommand{\BDiff}{B\operatorname{Diff}}

\newcommand{\cGL}[2]{\GL_{#1}(\Z,#2)}

\newcommand{\inv}[1]{{#1}^\times}

\newcommand{\CC}{\mathscr{C}}

\newcommand{\torelli}[1]{\operatorname{tor}(#1)}
\newcommand{\dquot}{/\hspace{-3pt}/}
\newcommand{\MM}{\mathfrak{M}}
\newcommand\hOrbits{/\mkern-6mu/}

\newcommand{\Xdimpower}[2]{S^{#1 \times #2}}

\newcommand{\rat}{\rho}

\newcommand{\monoid}{\mathcal{A}}
\newcommand{\lspace}{\mathcal{X}}
\newcommand{\rspace}{\mathcal{Y}}

\newcommand{\tensor}{\otimes}
\newcommand{\Sym}{\operatorname{Sym}}
\newcommand{\LL}{\mathbb{L}}
\newcommand{\lie}{\mathfrak}

\DeclareMathOperator\nil{nil}
\DeclareMathOperator\map{map}
\DeclareMathOperator\aut{aut}
\DeclareMathOperator\Diff{Diff}
\DeclareMathOperator\ad{ad}
\DeclareMathOperator\Hom{Hom}
\DeclareMathOperator\id{id}
\DeclareMathOperator\MC{\mathsf{MC}}
\DeclareMathOperator\Aut{Aut}
\DeclareMathOperator\AUT{\mathcal{A}ut}
\DeclareMathOperator\Lie{Lie}
\DeclareMathOperator\GL{\mathsf{GL}}
\DeclareMathOperator\SL{\mathsf{SL}}
\DeclareMathOperator\Sp{\mathsf{Sp}}
\DeclareMathOperator\orth{\mathsf{O}}
\DeclareMathOperator\Der{Der}
\DeclareMathOperator\gr{gr}

\newcommand{\AutNorm}[2]{\Aut_{#1}(#2)}
\newcommand{\AUTNorm}[2]{\AUT_{#1}(#2)}

\newcommand\Lbb{\mathbb{L}}
\newcommand\s{\mathrm{s}}

\newcommand\Q{\ensuremath{\mathbb{Q}}}
\newcommand\Z{\ensuremath{\mathbb{Z}}}
\newcommand\C{\ensuremath{\mathbb{C}}}

\newcommand\Rep{\mathrm{Rep}}

\newcommand\tor{\mathrm{tor}}

\newcommand\CE{\mathrm{CE}}

\newcommand\bU{\mathbb{U}}
\newcommand\tr[2]{#1 \langle #2 \rangle}
\newcommand\Derc{\Der^c}
\newcommand\nilrad{\nil\Derc(L)}

\newcommand\der[1]{\ensuremath{\frac{\partial}{\partial #1}}}
\newcommand{\RP}{\mathbb{R}\mathrm{P}}

\newcommand{\nerve}[1]{\langle {#1} \rangle}
\newcommand{\spatial}[1]{\langle {#1} \rangle}

\theoremstyle{plain}
\newtheorem{thm}{Theorem}[section]
\newtheorem{lemma}[thm]{Lemma}
\newtheorem{prop}[thm]{Proposition}
\newtheorem{cor}[thm]{Corollary}

\theoremstyle{definition}
\newtheorem{defn}[thm]{Definition}
\newtheorem{eg}[thm]{Example}

\theoremstyle{remark}
\newtheorem{rmk}[thm]{Remark}
\newtheorem*{rmk*}{Remark}

\newcommand{\shortenv}[2]{\newenvironment{#1}[1][]{\begin{#2}[{##1}]\pushQED{\qed}}{\popQED\end{#2}}}
\shortenv{sthm}{thm}
\shortenv{scor}{cor}
\shortenv{sprop}{prop}
\shortenv{slemma}{lemma}

\title{Algebraic models for classifying spaces of fibrations}
\author{Alexander Berglund \and Tom\'a\v{s} Zeman}
\address{\scriptsize Department of Mathematics\\
Stockholm University\\
SE-106 91 Stockholm\\
Sweden}

\email{alexb@math.su.se \\ tomas.zeman@math.su.se}

\begin{document}
\begin{abstract}
We prove new structural results for the rational homotopy type of the classifying space $\Baut{}{X}$ of fibrations with fiber a simply connected finite CW-complex $X$.

We first study nilpotent covers of $\Baut{}{X}$ and show that their rational cohomology groups are algebraic representations of the associated transformation groups. For the universal cover, this yields an extension of the Sullivan--Wilkerson theorem to higher homotopy and cohomology groups. For the cover corresponding to the kernel of the homology representation, this proves algebraicity of the cohomology of the homotopy Torelli space.

For the cover that classifies what we call {\it normal unipotent fibrations}, we then prove the stronger result that there exists a nilpotent dg Lie algebra $\lie g(X)$ in algebraic representations that models its equivariant rational homotopy type. This leads to an algebraic model for the space $\Baut{}{X}$ and to a description of its rational cohomology ring as the cohomology of a certain arithmetic group $\Gamma(X)$ with coefficients in the Chevalley--Eilenberg cohomology of $\lie g(X)$. This has strong structural consequences for the cohomology ring and, in certain cases, allows it to be completely determined using invariant theory and calculations with modular forms. We illustrate these points with concrete examples.

As another application, we significantly improve on certain results on self-homotopy equivalences of highly connected even-dimensional manifolds due to Berglund--Madsen, and we prove parallel new results in odd dimensions.
\end{abstract}

\maketitle


\section{Introduction}
In this paper we prove new structural results about the classifying space $\Baut{}{X}$ of the topological monoid of self-homotopy equivalences of a simply connected finite CW-complex $X$ from the point of view of rational homotopy theory.

The fundamental group of $\Baut{}{X}$ may be identified with the group $\piaut{X}$ of homotopy classes of self-homotopy equivalences of $X$. By a deep result due to Sullivan \cite{sul77} and Wilkerson \cite{wil76}, the group $\piaut{X}$ maps onto an arithmetic group with finite kernel. Our first result extends this in two directions: to more general deck transformation groups and to higher homotopy and (co)homology groups.




\begin{thm} \label{thm:A}
Let $\deckgrp{G}{X}$ be the deck transformation group of the cover $\Baut{G}{X}$ of $\Baut{}{X}$ associated to a subgroup $G$ of $\piaut{X}$. 
If $G$ acts nilpotently on $H_*(X;\Q)$, then $\deckgrp{G}{X}$ maps onto an arithmetic group with finite kernel. Moreover, the representations of $\deckgrp{G}{X}$ in the rational homology and simple homotopy groups\footnote{By the \emph{simple} homotopy groups of a connected space, we mean the homotopy groups modulo the action of the fundamental group, see Definition \ref{def:simple homotopy groups}.
} of $\Baut{G}{X}$ are algebraic representations of the ambient algebraic group.
\end{thm}


This appears as Theorem \ref{thm:general algebraicity} in the text. For $G$ the trivial group, this is precisely the result of Sullivan and Wilkerson, enhanced by the statement that the representations of $\piaut{X}$ in the higher homotopy and homology groups of the universal cover of $\Baut{}{X}$ are algebraic (see \S\ref{sec:higher arithmeticity}). For $G$ the kernel of the action of $\piaut{X}$ on $H_*(X;\Z)$, this yields a homotopical counterpart of Kupers--Randal-Williams' algebraicity result for the cohomology of Torelli groups of the manifolds $W_g$ \cite{krw20}---for \emph{arbitrary} simply connected finite complexes $X$ (see \S\ref{sec:torelli}).

Our main result extends Theorem \ref{thm:A} further to a space-level statement, for a particular choice of cover, and we use this to construct an algebraic model for $\Baut{}{X}$.

The problem of finding an algebraic model for $\Baut{}{X}$ has been raised in \cite[\S7.3]{gtht00} and \cite[\S7.2]{laz14}. This problem is highly non-trivial, because the space $\Baut{}{X}$ is not nilpotent in general and therefore not amenable to the classical methods of rational homotopy theory.
Algebraic models are known for the universal cover of \(\Baut{}{X}\) \cite{sul77,tan83} and for certain nilpotent covers \cite{ffm21}, but there are no general results that incorporate the action of the deck transformation groups into such models, let alone integrate such information to structural results about $\Baut{}{X}$. In this paper, we offer a solution. 
Instead of trying to incorporate the action of $\piaut{X}$ into models for the universal cover, which is the approach suggested in \cite{gtht00,laz14}, the key to our solution is pass to a cover with better properties.

\begin{thm} \label{thm:main}
There is a normal cover $\Baut{u}{X}$ of $\Baut{}{X}$, with deck transformation group $\Gamma(X)$, such that
\begin{enumerate}
\item $\Gamma(X)$ is an arithmetic subgroup of a reductive algebraic group $R(X)$,


\item $\Baut{u}{X}$ is $\Gamma(X)$-equivariantly rationally equivalent to the nerve $\nerve{\lie g(X)}$ of a nilpotent dg Lie algebra $\lie g(X)$ of algebraic representations of $R(X)$.

\end{enumerate}
\end{thm}

The proof is given in \S\ref{sec:algebraic lie model}. Here, we say that two connected spaces are \emph{rationally equivalent} if they can be connected by a zig-zag of maps that induce rational isomorphisms\footnote{We call a group homomorphism a rational isomorphism if its kernel is torsion and if for each element $x$ of the target, the $k$th power $x^k$ belongs to the image for some $k\ne 0$,
see Definition \ref{def:Q-iso}. For abelian groups, this is equivalent to inducing an isomorphism after tensoring with $\Q$.} on all homotopy and homology groups.
The \emph{nerve} is the functor from dg Lie algebras to topological spaces that effects the equivalence between the homotopy category of nilpotent dg Lie algebras and the homotopy category of rational nilpotent spaces, see \S\ref{sec:lie models}.
The cover $\Baut{u}{X}$ can be characterized as the classifying space for what we call \emph{normal unipotent fibrations}, see \S\ref{sec:unipotent fibrations}.

As a first application, Theorem \ref{thm:main} yields the following commutative dg algebra model for $\Baut{}{X}$.



\begin{cor} \label{cor:main}The space \(\Baut{}{X}\) is rationally equivalent to the homotopy orbit space \(\nerve{\lie g(X)}_{h \Gamma(X)}\) and there is a zig-zag of quasi-isomorphisms of commutative dg algebras
\[\Omega^*\big(\Baut{}{X}\big)\sim \Omega^*\big(\Gamma(X),C_{\CE}^*(\lie g(X))\big).\]
\end{cor}
This is proved as Theorem \ref{thm:cdga model} in the text.
Here $\Omega^*(B)$ stands for Sullivan's polynomial differential forms on $B$. It is a commutative dg algebra model for the singular cochains $C^*(B;\Q)$.
The right-hand side is a commutative dg algebra model for the cochains on the group $\Gamma(X)$ with coefficients in the Chevalley--Eilenberg cochains of $\lie g (X)$, see \S\ref{sec:cdga models}.



Theorem \ref{thm:main} can be interpreted as a `space-level' enhancement of Theorem \ref{thm:A} in the following sense: not only are the rational cohomology groups of $\Baut{u}{X}$ algebraic representations of the ambient algebraic group $R(X)$, but the entire \emph{rational homotopy type} of $\Baut{u}{X}$, represented by the dg Lie algebra $\lie g(X)$, is an algebraic representation of $R(X)$. This subsumes algebraicity of cohomology, because there is an isomorphism of $\Gamma(X)$-modules
$$H^*(\Baut{u}{X};\Q) \cong H_{CE}^*(\lie g(X))$$
and the right-hand side is obviously an algebraic representation of $R(X)$ if $\lie g(X)$ is. However, space-level algebraicity has much stronger consequences than algebraicity of cohomology. Together with semisimplicity of algebraic representations of reductive groups, it implies the following result, which reduces the computation of the rational cohomology ring of $\Baut{}{X}$ to the computation of Chevalley--Eilenberg cohomology and cohomology of arithmetic groups with coefficients in algebraic representations.

\begin{cor} \label{cor:cohomology ring}
There is an isomorphism of graded algebras
\begin{equation} \label{eq:alg iso}
H^*\big(\Baut{}{X};\Q\big) \cong H^*\big(\Gamma(X),H_{\CE}^*(\mathfrak{g}(X))\big).\qedhere
\end{equation}
\end{cor}
The proof is given in \S\ref{sec:cohomology}.
This result implies, but is stronger than, collapse of the rational Serre spectral sequence of the homotopy fiber sequence
\begin{equation} \label{eq:hfs}
\Baut{u}{X} \to \Baut{}{X} \to B\Gamma(X).
\end{equation}
Indeed, collapse of the spectral sequence only implies an isomorphism of algebras after passing to the associated graded algebra of some filtration, whereas Corollary \ref{cor:cohomology ring} says that $H^*(\Baut{}{X};\Q)$ is isomorphic to the $E_2$-page of this spectral sequence as a graded algebra. In particular, this algebra admits a bigrading by
\begin{equation} \label{eq:bigrading}
H^{p,q} = H^p\big(\Gamma(X),H_{\CE}^q(\mathfrak{g}(X))\big).
\end{equation}
Combined with the finiteness of the virtual cohomological dimension of arithmetic groups \cite{bs73}, this implies the following result, which reduces the computation of the ring $H^*(\Baut{}{X};\Q)$ modulo nilpotent elements to invariant theory.

\begin{cor} \label{cor:mod nilradical}
The ring homomorphism
\begin{equation} \label{eq:split surjection}
H^*(\Baut{}{X};\Q) \to H^*\big(\Baut{u}{X};\Q\big)^{\Gamma(X)}
\end{equation}
is split surjective and its kernel $I$ is a nilpotent ideal such that $I^n = 0$ for all $n>vcd(\Gamma(X))$. In particular, \eqref{eq:split surjection} induces an isomorphism modulo nilradicals. In other words, the algebraic variety of the graded commutative ring $H^*(\Baut{}{X};\Q)$ is isomorphic to that of the invariant ring $H_{\CE}^*(\lie g(X))^{\Gamma(X)}$.
\end{cor}

\begin{proof}
Under the isomorphism \eqref{eq:alg iso}, the homomorphism \eqref{eq:split surjection} corresponds to the projection $H^{*,*} \to H^{0,*}$, which is clearly split. The kernel is the ideal $H^{\geq 1,*}$, the $n$th power of which is contained in $H^{\geq n,*}$, which vanishes for $n> vcd(\Gamma(X))$.
\end{proof}

Let us point out another non-trivial facet of Corollary \ref{cor:cohomology ring}.

\begin{cor}
The ring homomorphism
\[ H^*(\Gamma(X),\Q) \to H^*(\Baut{}{X};\Q)\]
is split injective.
\end{cor}

\begin{proof}
The homomorphism in question corresponds to the inclusion of the subring $H^{*,0}$ into $H^{*,*}$, which is clearly split.
\end{proof}

This means in particular that all cohomology classes of the arithmetic group $\Gamma(X)$ are faithfully represented as characteristic classes of fibrations with fiber $X$. This is especially striking in view of the fact that many arithmetic groups can be realized as $\Gamma(X)$ for some $X$, \emph{cf.}~\cite[Theorem 10.3(iv)]{sul77}.

We have chosen to state our main result as an existence theorem in order to highlight the strong consequences of the mere existence of a dg Lie model of algebraic representations, but the ingredients $R(X)$, $\Gamma(X)$ and $\lie g(X)$ can be given concrete descriptions, see \S\ref{sec:concrete descriptions}. In \S\ref{sec:case studies} we use these concrete descriptions to make several explicit computations. Let us discuss one of these computations here.
For the $n$-fold product of a $d$-dimensional sphere, $\Xdimpower{d}{n} = S^d \times \cdots \times S^d$,
Corollary \ref{cor:cohomology ring} assumes the following form.

\begin{thm} \label{thm:X_n intro}
For $d$ odd, there is an isomorphism of graded algebras
\begin{equation} \label{eq:X_n intro}
H^*(B\aut(\Xdimpower{d}{n});\Q) \cong H^*\Big(\Gamma(\Xdimpower{d}{n}),\Sym^\bullet\big(V_n[d+1]\big) \Big),
\end{equation}
where $V_n[d+1]$ denotes the standard representation of $\GL_n(\Q)$ put in degree $d+1$,
\[\Gamma(\Xdimpower{d}{n}) =
\begin{cases}
\GL_n(\Z), & d=1,3,7, \\
\GL_n^\Sigma(\Z), & d\ne 1,3,7,
\end{cases}
\]
and $\GL_n^\Sigma(\Z)\leqslant \GL_n(\Z)$ is the congruence subgroup of matrices with exactly one odd entry in each row.
\end{thm}


For $n=2$ the right-hand side of \eqref{eq:X_n intro} can be computed in terms of modular forms via the Eichler--Shimura isomorphism.  
We also carry out computations for $n=3$, but for larger $n$ the cohomology of $\GL_n(\Z)$, or its congruence subgroups, with coefficients in algebraic representations is not fully known. See \S\ref{sec:products of spheres} for more details and a further discussion.

The above example illustrates that even in cases where $\Gamma(X)$ and $H_{\CE}^*(\lie g(X))$ can be described explicitly, a complete computation of the right-hand side of \eqref{eq:alg iso} is in general out of reach, due to the difficulty of calculating the cohomology of arithmetic groups. This suggests a different perspective on the isomorphism \eqref{eq:alg iso}. Rather than interpreting it as a computation of $H^*(\Baut{}{X};\Q)$, it tells us that cohomology classes of arithmetic groups in the right-hand side, say classes constructed using automorphic forms, can in principle be represented as characteristic classes of fibrations. Connections between characteristic classes and automorphic forms have been observed before in special cases (see \emph{e.g.}~\cite{fty88}). Our results show that this is not an isolated phenomenon. This suggests that there is a deep connection between characteristic classes of fibrations and cohomology of arithmetic groups.



Another important application of our results is that they lead to significant simplifications and improvements of certain key results of Berglund--Madsen \cite{bm20} about the cohomology of $\Baut{}{W_g}$, and related spaces, for the manifold
\[ W_g = \#^g S^d \times S^d.\]
In fact, this was the original motivation for this work. This is discussed in \S\ref{sec:highly connected manifolds}.
Our methods also yield parallel results for highly connected odd-dimensional manifolds that were unattainable by the methods of \cite{bm20}, see \S\ref{sec:odd}. This is used in the work of Stoll \cite{stoll}.


In the final section, \S\ref{sec:non-formal}, we study an example that, among other things, illuminates the advantage of working over $\Gamma(X)$ rather than $\piaut{X}$.



\subsection{Some comments on related work} \label{sec:relation to other works}


The first paragraph on p.314 in Sullivan's \cite{sul77} contains, without proof, the idea of modeling $\Baut{}{|\Lambda|}$, where $|\Lambda|$ is the realization of a minimal Sullivan algebra $\Lambda$, by taking the nerve of the maximal nilpotent ideal of $\Der \Lambda$ modulo the action of the reductive part of $\Aut \Lambda$. This idea seems to have been largely overlooked in the subsequent rational homotopy theory literature; we are not aware of any source where this idea and its consequences have been properly developed (and in fact we only became aware of this paragraph in the final stages of writing this paper). Theorem \ref{thm:sullivan} could be viewed as giving a precise formulation and proof. The key points of the present paper---the treatment of $\Baut{}{X}$ for non-rational $X$, the algebraicity of the cohomology of nilpotent covers of $\Baut{}{X}$, the existence of algebraic Lie models and its strong consequences for the structure of the cohomology ring of $\Baut{}{X}$---are to the best of our knowledge new. Our results could be regarded as a strong vindication of Sullivan's idea.

We were inspired by Oprea \cite{oprea84} (via Burghelea \cite{bur84}) for the idea of passing to the maximal reductive quotient of \(\piaut{X_\Q}\) to rectify homotopy actions on the algebraic models. The idea of studying the fiberwise rationalization of $\Baut{}{X} \to B\Gamma(X)$ as we do here is similar in spirit to studying relative Malcev completions of mapping class groups as done by Hain \cite{hain97}. The algebraicity result for the cohomology of the Torelli group of $W_{g}$ of Kupers and Randal-Williams \cite{krw20} inspired us to study similar questions for $\Baut{}{X}$.


Lazarev \cite[Theorem 5.1]{laz14} constructs Lie models for the universal cover of $\Baut{}{X_\Q}$ and shows that the action of the Lie algebra of $\piaut{X_\Q}$ on the higher homotopy groups of $\Baut{}{X_\Q}$ can be computed in terms of Chevalley--Eilenberg and Harrison cohomology. However, he does not address algebraicity of the representations and he does not construct group actions on the Lie models; in fact he raises this as a problem \cite[\S7.2]{laz14}.

F\'elix--Fuentes--Murillo \cite{ffm21} construct a Lie model for the space $B\aut_G(X)$ when $G\leqslant \piaut{X}$ is a subgroup that acts nilpotently on $H_*(X; \Z)$. We recover this Lie model (see Corollary \ref{cor:nonequivariant lie model} and Remark \ref{remark:recover ffm}). A crucial advantage of our approach is that it lets us incorporate the action of the deck transformation group. This aspect is not addressed in \cite{ffm21}. This is what allows to construct an algebraic model for the full space $\Baut{}{X}$ and not only for nilpotent covers of it.

\subsection*{Acknowledgements}
We thank Jonas Bergstr\"om, Diarmuid Crowley, Robin Stoll, and Torbj\"orn Tambour for useful discussions and we thank Richard Hain for helpful comments. We also thank the referee, whose comments led to an improvement of the structure of the paper.

This research was supported by the Knut and Alice Wallenberg foundation through grant no.~2017.0416 and by the Swedish Research Council through grant no.~2021-03946.

\section{Background and preliminaries}

\subsection{Localization of nilpotent groups and spaces}
We begin by introducing some terminology and recalling some facts about localizations of nilpotent groups and spaces, mainly following \cite{hmr75}.


\begin{defn} \label{def:Q-iso}
We call a group $G$ uniquely divisible, or $\Q$-local, if the equation $x^k = a$ has a unique solution $x\in G$ for every non-zero integer $k$ and every $a\in G$.
For a group homomorphism $f\colon G\to H$, we use the following terminology:
\begin{enumerate}
\item $f$ is $\Q$-injective if every element of $\ker(f)$ has finite order.
\item $f$ is $\Q$-surjective if for all $x\in H$, we have $x^k\in \operatorname{im}(f)$ for some $k\ne 0$.
\item $f$ is a $\Q$-isomorphism if it is both $\Q$-injective and $\Q$-surjective.
\end{enumerate}
\end{defn}
Every nilpotent group $G$ admits a $\Q$-localization $r\colon G\to G_\Q$, characterized up to isomorphism by the properties that $G_\Q$ is a $\Q$-local nilpotent group and that $r$ is a $\Q$-isomorphism (see \cite[p.7]{hmr75}). A homomorphism $G\to H$ between nilpotent groups is a $\Q$-isomorphism if and only if the induced homomorphism $G_\Q\to H_\Q$ is an isomorphism. For abelian groups $G$, one has $G_\Q \cong G\tensor \Q$.



Recall that a connected topological space $X$ is called nilpotent if the group $\pi_1(X)$ is nilpotent and if its action on $\pi_n(X)$ is nilpotent for every $n\geq 2$, in the sense that there is a filtration of $\pi_n(X)$ by $\pi_1(X)$-submodules such that the action on the filtration quotients is trivial.

A nilpotent space $X$ is called $\Q$-local if the group $\pi_n(X)$ is $\Q$-local for each $n\geq 1$. Every nilpotent space $X$ admits a $\Q$-localization, or rationalization, $r\colon X\to X_\Q$, characterized up to homotopy by the properties that $X_\Q$ is a $\Q$-local nilpotent space and that $\pi_n(r)\colon \pi_n(X) \to \pi_n(X_\Q)$ is a $\Q$-localization for every $n$.


\begin{defn} \label{def:rational equivalence}
We call a map $f\colon X\to Y$ between connected topological spaces a \emph{rational homotopy equivalence} if $\pi_n(f)\colon \pi_n(X)\to \pi_n(Y)$ is a $\Q$-isomorphism for all $n$, and a \emph{rational homology equivalence} if $H_n(f)\colon H_n(X;\Q) \to H_n(Y;\Q)$ is an isomorphism for all $n$.

We say that $f$ is a \emph{rational equivalence} if it is both a rational homotopy equivalence and a rational homology equivalence.
\end{defn}


It is well-known that a map between nilpotent spaces is a rational homotopy equivalence if and only if it is a rational homology equivalence. We will need an extension of this fact to virtually nilpotent spaces. Recall \cite{ddk77,ddk81} that a connected space $X$ is called \emph{virtually nilpotent} if $\pi_1(X)$ has a nilpotent subgroup of finite index and for each $n\geq 2$, there is a finite-index subgroup of $\pi_1(X)$ that acts nilpotently on $\pi_n(X)$. Equivalently, $X$ is virtually nilpotent if each Postnikov section $P_n X$ of $X$ admits a finite cover $E\to P_n X$ such that $E$ is nilpotent.



\begin{lemma} \label{lemma:virtual nilpotence lemma}
Let $f\colon X\to Y$ be a map from a virtually nilpotent space $X$ to a nilpotent space $Y$. If $f$ is a rational homotopy equivalence, then $f$ is a rational homology equivalence.
\end{lemma}

\begin{proof}
It suffices to show that $f$ induces a rational homology equivalence on each Postnikov section, so we may without loss of generality assume that $X$ admits a finite cover $p\colon E\to X$ such that $E$ is nilpotent. Clearly, $p$ is a rational homotopy equivalence. The composite $fp\colon E\to Y$ is then a rational homotopy equivalence between nilpotent spaces, so it is a rational homology equivalence. Since $p$ is a finite cover, a transfer argument shows that $H_*(p;\Q)\colon H_*(E;\Q)\to H_*(X;\Q)$ is surjective. We just saw that $H_*(f;\Q)\circ H_*(p;\Q) = H_*(fp;\Q)$ is an isomorphism, so $H_*(p;\Q)$ must be injective as well. It follows that $H_*(f;\Q)$ is an isomorphism.
\end{proof}

\begin{rmk}
The converse is true if $X$ is nilpotent, but false in general: the map $\RP^2 \to *$ is a rational homology equivalence from a virtually nilpotent space to a nilpotent space, but the induced map on $\pi_2$ is not a $\Q$-isomorphism.

Also, one can not relax nilpotence of $Y$ to virtual nilpotence. The universal cover $S^2 \to \RP^2$ provides an example of a map from a nilpotent space to a virtually nilpotent space which is a rational homotopy equivalence but not a rational homology equivalence.
\end{rmk}

\subsection{Affine algebraic groups and arithmetic groups}
In this section we will collect the basic facts about affine algebraic groups over $\Q$ and arithmetic groups that we will need, following mainly \cite{hoch71,mil17,serre79}. 
All algebras, vector spaces, undecorated tensor products, affine schemes, \emph{etc.}, should be taken to be over \(\Q\) unless explicitly specified otherwise.

Recall that an affine algebraic group is a group object in the category of affine schemes.
A linear algebraic group is an algebraic subgroup of \(\GL_n\) for some \(n\).
It is a well-known fact that every affine algebraic group admits a faithful algebraic representation, \emph{cf.}~\cite[Corollary~4.10]{mil17}, so the notions of affine algebraic group and linear algebraic group essentially coincide.

\subsubsection{Unipotent and reductive groups}\label{sec: unipotent and reductive gps}
Recall that a representation $V$ is called unipotent (or nilpotent) if there is a sequence of subrepresentations,
$$0 = V_0 \subseteq V_1 \subseteq \cdots \subseteq V_r = V,$$
such that $V_i/V_{i-1}$ is a trivial representation for every $i$.

An affine algebraic group \(U\) is called unipotent if every algebraic representation $V$ of \(U\) is unipotent.
Equivalently, \(U\) is unipotent if and only if it is an algebraic subgroup of the group \(\bU_n\) of \(n \times n\) upper triangular matrices with \(1\)s along the diagonal for some \(n\), \emph{cf.}~\cite[Theorem~14.5]{mil17}. The $\Q$-points $U(\Q)$ of a unipotent group $U$ is in particular a nilpotent and uniquely divisible group.

Every affine algebraic group \(G\) admits a largest normal unipotent subgroup, called the unipotent radical \(G_u\) of \(G\) (\emph{cf.}~\cite[Theorem~10.5]{hoch71}).
An affine algebraic group \(G\) is called reductive if \(G_u\) is trivial.
The representation theory of reductive groups in characteristic zero is particularly well-behaved:
\begin{sthm}[{\cite[Corollary~22.43]{mil17}, see also \cite[p.78]{hoch71}}] \label{thm:linearly reductive}
Every finite-dimensional algebraic representation of a reductive group is semisimple.
\end{sthm}
In other words, the category \(\Rep_\Q(G)\) of finite-dimensional algebraic representations of \(G\) is a semisimple abelian category whenever \(G\) is reductive.

Another feature of affine algebraic groups in characterstic zero is the existence of Levi decompositions. There is a canonical extension
\begin{equation} \label{eq:Levi}
    1 \longrightarrow G_u \longrightarrow G \longrightarrow G/G_u \longrightarrow 1
\end{equation}
of the maximal reductive quotient of \(G\) by the unipotent radical.

\begin{sthm}[{\cite[Theorem~14.2]{hoch71}}] \label{thm:Levi}
The extension \ref{eq:Levi} is (non-canonically) split and moreover any two splittings are conjugate in the action of \(G_u\).
\end{sthm}

The following lemma gives a concrete description of the unipotent radical and the maximal reductive quotient. It is presumably well-known, but we include a proof as we have not found the precise statement we give here in the literature.

Recall that a composition series of a representation $V$ is a filtration
\[
    0 = V_0 \subseteq \cdots \subseteq V_n = V
\]
by subrepresentations such that each \(V_i/V_{i-1}\) is a simple representation.
The `semisimplification' of $V$ is the associated graded representation
\[V^{ss} = \bigoplus_{i=1}^n V_i/V_{i-1}.\]
It is semisimple by construction and its isomorphism type is independent of the choice of composition series by the Jordan--H\"older theorem. Note that $V$ is isomorphic to $V^{ss}$ if and only if $V$ is semisimple.


\begin{lemma} \label{lemma:gr V}
Let $G$ be an affine algebraic group defined over $\Q$. If $V$ is a representation of $G$ with unipotent kernel, then the unipotent radical $G_u$ and the maximal reductive quotient $G/G_u$ may be identified with the kernel and the image of the homomorphism $G\to \GL(V^{ss})$, respectively.
\end{lemma}

\begin{proof}
Let $K$ and $N$ denote the kernel of the action of $G$ on $V$ and $V^{ss}$, respectively.  Clearly, both $K$ and $N$ are normal in $G$ and $K\leqslant N$. Furthermore, $K$ is unipotent by hypothesis, and $N/K$ is unipotent because $V$ is a faithful unipotent representation of it. Since unipotent groups are closed under extensions (\emph{cf.}~\cite[Corollary 14.7]{mil17}), it follows that $N$ is unipotent.

To show that $N$ is the unipotent radical, we need to show that every normal unipotent subgroup $U$ of $G$ acts trivially on $V^{ss}$. For this, it suffices to show that $U$ acts trivially on every simple $G$-representation $W$. Since $U$ is unipotent, there is a non-zero $w\in W$ that is fixed by $U$ (\emph{cf.}~\cite[Proposition 14.1]{mil17}). Since $U$ is normal in $G$, it also fixes $gw$ for every $g\in G$. Indeed, for every $u\in U$ we have $g^{-1}ug\in U$, whence $(g^{-1}ug)w =w$ so that $ugw=gw$. Since $W$ is simple, $w$ generates $W$ as a $G$-module, so $U$ acts trivially on $W$.
\end{proof}

This has the following consequence:
\begin{lemma}\label{lemma: uniradical}
Let $G$ be an affine algebraic group defined over $\Q$ with Lie algebra $\lie g$. If $V$ is a representation of $G$ with unipotent kernel, then the Lie algebra of the unipotent radical, \(\Lie G_u\), may be identified with the maximal ideal
\[\nil_V \lie g \subseteq \lie g,\]
consisting of elements which act nilpotently on \(V\).
\end{lemma}
\begin{proof}
By \cite[\S 10.14]{mil17}, the functor \(G \mapsto \Lie G\) from affine algebraic groups to Lie algebras commutes with pullbacks, so in particular it preserves kernels of morphisms.
Thus by the preceding lemma,
\begin{equation}\label{lie ker is ker lie}
    \Lie G_u = \ker \left( \lie g \longrightarrow \lie{gl}(V^{ss})\right).
\end{equation}
Thus we have that \(\Lie G_u \subseteq \nil_V \lie g\).

For the reverse inclusion, it suffices to show that \(\nil_V \lie g\) is the Lie algebra of a normal unipotent subgroup \(H\) of \(G\), as then \(H \leqslant G_u\) and \(\nil_V \lie g \subseteq \Lie G_u\).
Let \(K\) be the kernel of the action of \(G\) on \(V\), and let \(\lie h'\) be the image of \(\nil_V \lie g\) under the action morphism \(\lie g \to \lie{gl}(V)\).
Then \(\lie h'\) consists of nilpotent endomorphisms of \(V\), so by Engel's theorem, \(\lie h'\) is a nilpotent subalgebra of \(\Lie(G/K) \subseteq \lie{gl}(V)\).
Thus by \cite[Theorem~14.37]{mil17} there is a unipotent subgroup \(H' \leqslant G / K \leqslant \GL(V)\) such that \(\lie h' = \Lie H'\).
Let \(H\) be the preimage of \(H'\) in \(G\).
Then \(H\) is an extension of the unipotent group \(H'\) by the unipotent group \(K\), hence unipotent, and \(\Lie H\) is therefore nilpotent.
Moreover, since the functor \(\Lie\) commutes with pullbacks, we get that \(\nil_V \lie g = \Lie H\).
Finally, since \(\nil_V \lie g\) is an ideal of \(\lie g\), the subgroup \(H\) of \(G\) is normal.
\end{proof}

The quotient $G/H$ of an affine algebraic group $G$ by a normal algebraic subgroup $H$ always exists \cite[Theorem~5.14]{mil17}, but the rational points of the quotient $(G/H)(\Q)$ need not agree with $G(\Q)/H(\Q)$ in general. They do agree, however, if $H$ is unipotent.

\begin{lemma} \label{lemma:quotient}
Let $G$ be an affine algebraic group.
If $H$ is a normal unipotent algebraic subgroup of $G$, then the natural homomorphism
\[G(\Q)/H(\Q) \to (G/H)(\Q)\]
is an isomorphism.
\end{lemma}

\begin{proof}
This is a consequence of the vanishing of Galois cohomology, $$H^1(\operatorname{Gal}(\overline{\Q}/\Q),H),$$
for unipotent groups $H$, see \cite[Theorem 9.5]{wil76}.
\end{proof}

Similarly, the normalizer $N_G(H)$ of an algebraic subgroup $H$ of an affine algebraic group $G$ always exists and its rational points, $N_G(H)(\Q)$, are contained in $N_{G(\Q)}(H(\Q))$, the normalizer of $H(\Q)$ in $G(\Q)$ in the ordinary sense \cite[Proposition 1.83]{mil17}. In general, equality need not hold, but it does if $H$ is unipotent.

\begin{lemma} \label{lemma:normalizer}
If $H$ is a unipotent algebraic subgroup of an affine algebraic group $G$, then
\(N_G(H)(\Q) = N_{G(\Q)}(H(\Q))\).
\end{lemma}

\begin{proof}
This follows from \cite[Proposition 1.84]{mil17} upon noting that every affine algebraic group over $\Q$ is smooth \cite[Theorem 3.23]{mil17} and that $H(\Q)$ is dense in $H$ if $H$ is unipotent (by, e.g., \cite[Theorem 17.93]{mil17}).
\end{proof}

\subsubsection{Arithmetic subgroups}
Given a linear algebraic group \(G \leqslant \GL_n\), we write
\[
    G(\Z) = G(\Q) \cap \GL_n(\Z)
\]
for the integer matrices inside the rational points of \(G\). By definition, an arithmetic subgroup \(\Gamma\) of \(G\) is a subgroup of \(G(\Q)\) which is commensurable with \(G(\Z)\), \emph{i.e.} the intersection \(\Gamma \cap G(\Z)\) has finite index in both \(G(\Z)\) and \(\Gamma\).

Arithmeticity can be thought of as a strong finiteness property. In particular, arithmetic groups are finitely presented \cite{serre79}.


\begin{slemma}[{\cite[\S1.1]{serre79}}]\label{lemma:arithmetic preimage}
Let $G$ be an affine algebraic group over $\Q$ and let $H\leqslant G$ be an algebraic subgroup. If $\Gamma$ is arithmetic in $G(\Q)$, then $\Gamma \cap H(\Q)$ is arithmetic in $H(\Q)$.
\end{slemma}

\begin{slemma}[{\cite[Theorem 6]{borel66}}]\label{lemma:image arithmetic}
Let $\varphi \colon G\to G'$ be a surjective morphism of affine algebraic groups defined over $\Q$. If $\Gamma$ is an arithmetic subgroup of $G(\Q)$, then $\varphi(\Gamma)$ is an arithmetic subgroup of $G'(\Q)$.
\end{slemma}

We record the following, presumably well-known, characterization of arithmetic subgroups of unipotent groups for later use. We include a proof for completeness.


\begin{lemma} \label{lemma: integer points of unipotent group}
Let $U$ be a unipotent algebraic group defined over $\Q$ and let $\Gamma$ be a subgroup of $U(\Q)$.
\begin{enumerate}
\item $\Gamma$ is Zariski dense in $U$ if and only if the inclusion of $\Gamma$ into $U(\Q)$ is $\Q$-surjective.

\item $\Gamma$ is an arithmetic subgroup of $U(\Q)$ if and only if it is finitely generated and Zariski dense.
\end{enumerate}
\end{lemma}

\begin{proof}
We may assume that $U$ is a subgroup of $\bU_n$ for some $n$. Following \cite[p.104]{segal}, consider the Lie subalgebra $\mathscr{L}(\Gamma)$ of the Lie algebra $\lie u$ of $U$ given by the linear span of the image of $\Gamma$ under the bijection $\log \colon U(\Q) \to \lie u$. The associated unipotent algebraic subgroup $\exp \mathscr{L}(\Gamma) \leqslant U$ contains $\Gamma$, so it must be equal to $U$ if $\Gamma$ is Zariski dense. The inclusion $\Gamma \to U(\Q)$ is then $\Q$-surjective by \cite[Theorem 2(iv), p.104]{segal}. Conversely, if the inclusion of $\Gamma$ into $U(\Q)$ is $\Q$-surjective, then so is the inclusion of the Zariski closure $\overline{\Gamma}$. This is then a $\Q$-isomorphism between nilpotent uniquely divisible groups, so it must be an equality. This proves the first statement.

It follows from \cite[Exercise 13, p.123]{segal} that every finitely generated dense subgroup of $U(\Q)$ is arithmetic. Conversely, if $\Gamma$ arithmetic, then it is finitely generated, and we will now argue that the inclusion $\Gamma \to U(\Q)$ is $\Q$-surjective. Inclusions of finite-index subgroups are clearly \(\Q\)-isomorphisms, so it suffices to show that the inclusion of $U(\Z)$ into $U(\Q)$ is \(\Q\)-surjective, \emph{i.e.}, that for every \(A \in U(\Q)\), there is a positive integer \(k\) such that \(A^k \in U(\Z)\). Since \(A\) is a unipotent $n\times n$-matrix, the matrix \(N = A - I\) satisfies \(N^n = 0\).
Pick a positive integer \(d\) such that \(d N^i\) has integer entries for all $i\geq 1$, and let \(k = 1! 2! \cdots n! d\).
Then the matrix
\[
    A^k = \sum_{i = 0}^{n-1} \binom{k}{i} N^i
\]
has integer entries, because each coefficient \(\binom{k}{i}\) is divisible by \(d\).
\end{proof}

Since $U(\Q)$ is nilpotent and $\Q$-local when $U$ is unipotent, the group $U(\Q)$ will be a $\Q$-localization of the nilpotent group $\Gamma\leqslant U(\Q)$ whenever $\Gamma$ is dense in $U$.

\begin{rmk}
One can not relax unipotence to nilpotence in Lemma \ref{lemma: integer points of unipotent group}. The multiplicative group is abelian and in particular nilpotent, but the inclusion of $\Z^\times$ into $\Q^\times$ is not $\Q$-surjective.
\end{rmk}

\subsection{Nilpotent radicals of dg Lie algebras} \label{sec:nilradical}
In this section we discuss the notion of nilradical in the setting of differential graded Lie algebras.
Let \((\lie{g}, \delta)\) be a dg Lie algebra, possibly unbounded as a chain complex. We write \(\Gamma^k\lie{g}\) for the lower central series of \(\lie{g}\), so \(\Gamma^1\lie{g} = \lie{g}\) and \(\Gamma^{k+1}\lie{g} = \left[ \Gamma^k\lie{g}, \lie{g} \right]\).
Given an integer \(k\), we write \(\tr{\lie{g}}{k}\) for the truncation of \(\lie{g}\) given by
\[
    \tr{\lie{g}}{k}_n = \begin{cases} 0 & \text{if } n < k\\ \ker (\delta: \lie{g}_n \to \lie{g}_{n-1}) & \text{if } n = k\\ \lie{g}_n & \text{if } n > k \end{cases}
\]
We call $\lie g$ \emph{connected} if $\lie g = \lie g\langle 0 \rangle$ and we call $\lie g$ \emph{simply connected} if $\lie g = \lie g\langle 1 \rangle$. We say that $\lie g$ is of \emph{finite type} if $\lie g_i$ is finite-dimensional for every $i$.

\begin{defn}
A connected dg Lie algebra \(\lie{g}\) is \emph{nilpotent} if the following equivalent conditions are satisfied:
\begin{enumerate}[(i)]
    \item For every $n$, we have \((\Gamma^k \lie{g})_n = 0\) for \(k \gg n\);
    \item \(\lie{g}_0\) is a nilpotent Lie algebra which acts nilpotently on \(\lie{g}_n\) for every \(n>0\).
\end{enumerate}
\end{defn}


\begin{defn} \label{def:nilradical}
The \emph{nilradical} of \(\lie{g}\), denoted \(\nil \lie{g}\), is the maximal nilpotent ideal of \(\lie{g}\langle 0 \rangle\), provided such an ideal exists.
\end{defn}

\begin{rmk}
Ordinary Lie algebras may be identified with dg Lie algebras concentrated in degree \(0\). For these, the above definitions specialize to the usual definitions of nilpotence and nilradicals. In particular, the nilradical does not necessarily exist unless certain finiteness conditions are imposed. If the nilradical exists, however, then it is unique: if \(I, J\) are nilpotent ideals of \(\lie{g}\), then so is \(I+J\).
Thus if \(I\) is maximal, then \(J \subseteq I+J = I\).
\end{rmk}

\begin{lemma}\label{existence of nil}
Suppose that \(Z_0(\lie{g})\) is finite-dimensional.
Then \(\nil \lie{g}\) exists.
\end{lemma}
\begin{proof}
We may assume that $\lie g$ is connected.
Note that the positively graded truncation \(\tr{\lie{g}}{1}\) is always a nilpotent ideal of \(\lie{g}\), so the poset of all nilpotent ideals contains a maximal element if and only if the poset of nilpotent ideals above \(\tr{\lie{g}}{1}\) does.
But the latter poset injects into the poset of nilpotent (or indeed all) ideals of the quotient dg Lie algebra \(\lie{g}' = \lie{g} / \tr{\lie{g}}{1}\).
Clearly \(\lie{g}'_i = 0\) for \(i > 1\), the differential \(\lie{g}'_1 \to \lie{g}'_0\) is injective, and \(\lie{g}'_0 = \lie{g}_0\) is finite-dimensional, so every chain of ideals of \(\lie{g}'\) is finite.
\end{proof}

\subsection{Geometric realizations of nilpotent dg Lie algebras}
To each nilpotent Lie algebra $\lie g$ over $\Q$, one can associate a group $\exp(\lie g)$ with underlying set $\lie g$ and with multiplication given by the Baker--Campbell--Hausdorff formula. The association $\lie g \mapsto \exp(\lie g)$ is part of an equivalence of categories between nilpotent Lie algebras over $\Q$ and nilpotent uniquely divisible groups (\emph{cf.}~\cite[Appendix A]{quillen69}). When $\lie g$ is finite-dimensional, $\exp(\lie g)$ can be given the structure of an affine algebraic group, and the association $\lie g \mapsto \exp(\lie g)$ is part of an equivalence between the category of finite-dimensional nilpotent Lie algebras over $\Q$ and the category of unipotent algebraic groups over $\Q$ (\emph{cf.}~\cite[Theorem 14.37]{mil17}).

Now let $\lie g$ be a nilpotent dg Lie algebra and consider the simplicial group
\[
    \exp_\bullet \lie{g} = \exp Z_0(\lie{g} \otimes \Omega_\bullet),
\]
where $\Omega_\bullet$ denotes the simplicial commutative differential graded algebra over $\Q$ of polynomial differential forms on the standard simplices. 

\begin{sprop}[{\cite[Theorem 6.2]{be23}}] \label{prop:homotopy of exp}
For every nilpotent dg Lie algebra $\lie g$, there is a natural isomorphism of groups,
\begin{equation} \label{eq:groups}
\exp H_0(\lie g) \to \pi_0 \exp_\bullet(\lie g),
\end{equation}
and a natural isomorphism of abelian groups
\[H_k(\lie g) \to \pi_k\exp_\bullet(\lie g),\]
for every $k>0$, compatible with the actions of the groups in \eqref{eq:groups}.
\end{sprop}
By \cite[Theorem 5.2(2)]{be23}, a natural model for the classifying space $B\exp_\bullet(\lie g)$ is given by the nerve, or Maurer--Cartan space.
This is the simplicial set defined by
\[
    \MC_\bullet(\lie{g}) = \MC(\lie{g} \otimes \Omega_\bullet),
\]
where $\MC$ denotes the set of Maurer--Cartan elements in a dg Lie algebra, i.e., the solutions to the equation
\[\delta(\tau) + \frac{1}{2}\big[\tau,\tau\big] = 0.\]
Note that Proposition \ref{prop:homotopy of exp} implies that $|\MC_\bullet(\lie g)|$ is a $\Q$-local nilpotent space if $\lie g$ is a nilpotent dg Lie algebra.

\begin{defn} \label{def:lie model}
A \emph{Lie model} for a space $B$ is by definition a dg Lie algebra $\lie g$ such that $B$ is rationally homology equivalent to $|\MC_\bullet(\lie g)|$.
\end{defn}

If $X$ is a simply connected space, then Quillen's dg Lie algebra $\lambda(X)$ \cite{quillen69} is a Lie model for $X$ in the sense of the above definition (this follows from \cite[Theorem 8.1]{be23}). Recall (see e.g.~\cite{bl77}) that every simply connected space $X$ admits a unique, up to non-canonical isomorphism, Lie model of the form $L_X= (\LL(V),d)$, where $\LL(V)$ denotes the free graded Lie algebra on the graded vector space $V= s^{-1}\widetilde{H}_*(X;\Q)$ and $d$ is decomposable in the sense that the induced differential on $L_X/[L_X,L_X]$ is trival. Note that $L_X$ is finitely generated if and only if $\widetilde{H}_*(X;\Q)$ is finite-dimensional. We will refer to $L_X$ as the minimal Quillen model of $X$.

Recall from e.g.~\cite[\S22]{fht01} that the Chevalley--Eilenberg complex of a dg Lie algebra $\lie g$ is the dg coalgebra
\[
    C_*(\lie g) = (\Lambda \s \lie g, d = d_0 + d_1)
\]
where \(d_0\) and \(d_1\) are the coderivations characterized by
\begin{align*}
    d_0 (\s x) & = - \s (dx)\\
    d_1 (\s x_1 \wedge \s x_2) & = (-1)^{\lvert x_1 \rvert} \s [x_1, x_2].
\end{align*}
The Chevalley--Eilenberg cochain algebra is the dual dg algebra $C^*(\lie g) = C_*(\lie g)^\vee$. We denote its cohomology algebra by $H_{CE}^*(\lie g)$.

Next, recall that the polynomial differential forms on a simplicial set $K$ is the dg algebra $\Omega^*(K) = \Hom_{sSet}(K,\Omega_\bullet)$. It is a commutative dg algebra model for the cochains on $K$, see \S\ref{sec:local systems} below for a further discussion. Also, recall that the spatial realization of a commutative dg algebra $\Lambda$ is the simplicial set
\(
    \spatial{\Lambda} = \Hom_{\mathrm{cdga}}(\Lambda, \Omega_\bullet).
\)

\begin{prop} \label{prop:ce cochains}
Let $\lie g$ be a nilpotent dg Lie algebra of finite type. There is a natural quasi-isomorphism of commutative dg algebras
\[ C^*(\lie g) \to \Omega^*( \MC_\bullet(\lie g)).\]
In particular, if $\lie g$ is a Lie model for the space $B$, then $H^*(B;\Q) \cong H_{CE}^*(\lie g)$ and $|\MC_\bullet(\lie g)|$ is a $\Q$-localization of $B$.
\end{prop}

\begin{proof}
There is a natural isomorphism of simplicial sets
\[\MC_\bullet(\lie g) \xrightarrow{\cong} \spatial{C^*(\lie g)},\]
see \emph{e.g.}~\cite[Corollary 3.6]{be17}. 
The Chevalley--Eilenberg cochains $C^*(\lie g)$ is a Sullivan algebra and hence cofibrant (see e.g.~\cite[Theorem 2.3]{be15}), so the adjoint map $C^*(\lie g) \to \Omega^*(\MC_\bullet(\lie g))$ is a quasi-isomorphism as a consequence of \cite[Theorem 9.4]{bg76}.
\end{proof}

\subsection{The dg Lie algebra of curved derivations} \label{sec:curved}
The main result of Quillen's theory \cite{quillen69} is that the functor $X\mapsto \lambda(X)$ induces an equivalence between the rational homotopy category of simply connected pointed spaces and the homotopy category of simply connected dg Lie algebras. If one wants to model unpointed spaces, one has to enlarge the set of morphisms of dg Lie algebras. A possible solution is to work with so called curved morphisms of dg Lie algebras (\emph{cf.}~\cite{maunder18}).

Let \(L\) be a finitely generated, positively graded dg Lie algebra and let \(L_+ = (L * \Lbb(\tau), d^\tau)\) be the dg Lie algebra obtained by freely adjoining a Maurer--Cartan element \(\tau\) to \(L\) and twisting the differential by \(\tau\), so \(d^\tau(x) = d(x) + [\tau, x]\) for \(x \in L\). It is straightforward to check that morphisms from $L_+$ to a dg Lie algebra $L'$ correspond to curved morphisms from $L$ to $L'$. This is analogous to the fact that the space of free maps from a pointed space $X$ to another pointed space $Y$ can be recovered as the space of pointed maps from $X_+$ to $Y$, where $X_+$ is the space obtained from $X$ by adding a disjoint basepoint.

The projection $p\colon L_+ \to L$ that restricts to the identity on $L$ and sends $\tau$ to zero is a morphism of dg Lie algebras. Let $\Der^c(L)$ denote the chain complex of $p$-derivations from $L_+$ to $L$. Its elements are maps $\theta \colon L_+ \to L$ that satisfy
$$\theta[x,y] = [\theta(x),p(y)] + (-1)^{|\theta| |x|}[p(x),\theta(y)],$$
for all $x,y\in L_+$. 

As is well known, the mapping cone of the chain map $\ad \colon L \to \Der L$, denoted $\Der L \dquot \ad L$ or $\Der L \ltimes_{\ad} sL$, admits a dg Lie algebra structure, see \emph{e.g.}~\cite{tan83} or \cite[p.252]{be17}. We now make the observation that this mapping cone may be identified with the chain complex of curved derivations.

\begin{prop}\label{prop: curved derivations}
The map $\varphi\colon \Der L \dquot \ad L \to \Der^c(L)$, defined by
$$\varphi(\theta,s\xi) = \theta \circ p + (-1)^{|\xi|} \xi \frac{\partial}{\partial \tau},$$
is an isomorphism of chain complexes, with inverse
$$\varphi^{-1}(\nu) = \big( \nu|_L, (-1)^{|\nu|+1}s\nu(\tau) \big).$$
\end{prop}

\begin{proof}
Straightforward calculation.
\end{proof}

In particular, this implies that $\Der^c(L)$ admits a dg Lie algebra structure and that it acts on $L$ by outer derivations in the sense of \cite[\S3.5]{be17}. We do not recall the full definition here, but we point out that the outer action of \(\Derc(L)\) on \(L\) gives rise to an (ordinary) action by coderivations on the Chevalley--Eilenberg chains $C_*(L) = (\Lambda sL,d)$. This action may be constructed by noting that \(\s L_+\) contains \(\Q[0] \oplus \s L = \Lambda^{\leqslant 1} \s L\) as a graded subspace, where the copy of \(\Q\) in degree 0 is generated by \(\s \tau\).
Hence every curved derivation \(\phi \in \Derc(L)\) determines a map \(\Lambda^{\leqslant 1} \s L \to \s L\) by suspension and restriction, and thus it determines a unique coderivation of the cofree coalgebra \(\Lambda \s L\).
Explicitly,
\begin{equation}\label{eqn: coderivation action}
\begin{aligned}
    \phi (\s x_1 \wedge \cdots \wedge \s x_n) & = (-1)^{\lvert \phi \rvert} \s \phi(\tau) \wedge \s x_1 \wedge \cdots \wedge \s x_n\\
    & + \sum_i \pm \s x_1 \wedge \cdots \wedge \s \phi(x_i) \wedge \cdots \wedge \s x_n
\end{aligned}
\end{equation}
Dually, \(\Derc(L)\) then acts by derivations on the cdga \(C^*(L) = C_*(L)^\vee\) of Chevalley--Eilenberg cochains on \(L\).

\begin{rmk} \label{remark:comparison to laz14}
The constructions $L_+$ and $\Derc(L)$ essentially agree with the constructions $L\langle \tau \rangle$ and $\Der_\tau(L\langle \tau \rangle)$ considered in \cite[pp.44--45]{laz14}.
\end{rmk}

\subsection{Algebraic groups of automorphisms} \label{sec:algebraic groups of automorphisms}
The following goes back to Sullivan \cite[\S6]{sul77}. We outline the proof for the reader's convenience.

\begin{thm}\label{thm:sul-wil}
Let $X$ be a simply connected finite CW-complex with minimal Quillen model $L$.

\begin{enumerate}[(i)]
\item\label{thm:sul-wil:aut L}
The automorphisms of $L$ form an affine algebraic group $\AUT(L)$ and the automorphisms homotopic to the identity form a unipotent subgroup $\AUT_h(L)$.

\item\label{thm:sul-wil:Q-points} The group $\piaut{X_\Q}$ may be identified with the $\Q$-points of the quotient affine algebraic group $\AUT^h(L) = \AUT(L)/\AUT_h(L)$.

\item \label{thm:sul-wil:lie algebra} The Lie algebra of the algebraic group $\AUT^h(L)$ is isomorphic to $H_0(\Derc L)$.

\item\label{thm:sul-wil:unipotent ker} The representation of $\piaut{X_\Q}$ in $H_*(X;\Q)$ extends to an algebraic representation of $\AUT^h(L)$ with unipotent kernel.

\end{enumerate}
\end{thm}

\begin{proof}[Proof outline]
Since \(X\) is a finite CW-complex, the minimal Quillen model $L$ is finitely generated. We can then define a linear algebraic group $\AUT L$ with functor of points
\[
    R \longmapsto \Aut_{\mathrm{dgl}(R)}(L \otimes R),
\]
sending a \(\Q\)-algebra \(R\) to the group of \(R\)-linear automorphisms of the dg Lie algebra \(L \otimes R\). To see that this indeed is a linear algebraic group, note that a faithful finite-dimensional representation is given by $L_{\leq N}$, where $N$ is the maximal degree of a generator.

The subgroup $\AUT_h(L)$ of automorphisms homotopic to the identity may be identified with the unipotent algebraic group associated to the nilpotent Lie algebra $B_0 \Der L$ of derivations of $L$ of the form $[d,\theta]$, for some derivation $\theta\colon L\to L$ of degree $1$, see \cite[\S 6]{sul77} or \cite[Theorem~3.4]{blla05}, or the proof of \cite[Lemma~8.1]{ffm21}.

Since $X_\Q$ is simply connected, there is an isomorphism of groups $\piautp{X_\Q} \cong \piaut{X_\Q}$.
Quillen's equivalence \cite{qui67} between the category of simply connected pointed spaces, localized at the rational homotopy equivalences, and the category of positively graded dg Lie algebras, localized at the quasi-isomorphisms, coupled with the fact that a quasi-isomorphism between minimal dg Lie algebras is an isomorphism, leads to an isomorphism of groups
\[\piautp{X_\Q} \cong \Aut(L)/\Aut_h(L),\]
where $\Aut(L)$ and $\Aut_h(L)$ denote the $\Q$-points of $\AUT L$ and $\AUT_h L$, respectively. It follows from Lemma \ref{lemma:quotient} that the right-hand side may be identified with the $\Q$-points of the quotient algebraic group $\AUT^h L = \AUT L /\AUT_h L$.

The Lie algebra of $\AUT L$ may be identified with $Z_0(\Der L)$, so the Lie algebra of $\AUT^h(L)$ may be identified with $Z_0(\Der L)/B_0(\Der L) = H_0(\Der L)$. Since $L$ is positively graded, $H_0(\Der L) = H_0(\Der^c L)$.

For the last statement, one identifies
$$H_*(X;\Q)  = \Q\oplus sL/[L,L],$$
and notes that the right-hand side is an algebraic representation of $\AUT L$ on which $\AUT_h L$ acts trivially. The associated graded $\gr L$ of $L$ with respect to the lower central series is isomorphic to the free Lie algebra on $L/[L,L]$ as a representation of $\AUT L$. It follows that the kernel of the representation $L/[L,L]$ acts trivially on $\gr L$, which implies that it acts unipotently on $L$ in each degree.
\end{proof}

There is a parallel statement for finite Postnikov stages. The proof is essentially the same so we omit it.

\begin{sthm} \label{thm:sul-wil II}
Let $X$ be a simply connected Postnikov stage of finite type with minimal Sullivan model $\Lambda$.
\begin{enumerate}[(i)]

\item\label{thm:sul-wil:aut L II} The automorphisms of $\Lambda$ form an affine algebraic group $\AUT(\Lambda)$ and the automorphisms homotopic to the identity form a unipotent subgroup $\AUT_h(\Lambda)$.

\item\label{thm:sul-wil:Q-points II}The group $\piaut{X_\Q}$ may be identified with the $\Q$-points of the algebraic group $\AUT^h \Lambda = \AUT \Lambda / \AUT_h\Lambda$.

\item\label{thm:sul-wil:lie algebra II} The Lie algebra of the algebraic group $\AUT^h \Lambda$ is isomorphic to $H_0(\Der \Lambda)$.

\item\label{thm:sul-wil:unipotent ker II} The representation of $\piaut{X_\Q}$ in $H_*(X;\Q)$ extends to an algebraic representation of $\AUT^h(\Lambda)$ with unipotent kernel.
\end{enumerate}
\end{sthm}


For a simply connected finite complex $X$ of dimension $n$, there is an isomorphism $\piaut{X_\Q} \cong \piaut{P_n X_\Q}$, where $P_n X$ denotes the $n$th Postnikov stage of $X$, cf.~\cite[Proposition 3.1(2)]{blla05}. We therefore have two, a priori different, ways to realize $\piaut{X_\Q}$ as the $\Q$-points of an algebraic group, one using Quillen models and one using Sullivan models. Saleh \cite{sa23} has recently shown that these algebraic group structures agree. In the rest of the paper we will, somewhat imprecisely, refer to $\piaut{X_\Q}$ as an algebraic group, with the understanding that we refer to either of the isomorphic algebraic groups above.

The following goes back to Sullivan \cite[Theorem 10.3]{sul77} and Wilkerson \cite[Theorem B]{wil76}. An elaboration of Sullivan's argument can be found in \emph{e.g.}~\cite{triantafillou}.

\begin{sthm} \label{thm:arithmeticity}
The homomorphism $\rat\colon \piaut{X} \to \piaut{X_\Q}$ induced by rationalization has finite kernel and image an arithmetic subgroup.
\end{sthm}

\section{Proofs of the main results}
Throughout this section, we fix a simply connected finite CW-complex $X$ and we let $\aut(X)$ denote the topological monoid of self-homotopy equivalences of $X$. The group $\piaut{X}$ will interchangeably be thought of as the group $\pi_0\aut(X)$ of components of the monoid $\aut(X)$ or as the fundamental group $\pi_1 \Baut{}{X}$ of the classfying space.

\begin{defn}\label{def: autU}
For a subgroup $G\leqslant \piaut{X}$, let $\aut_G(X)$ denote the union of the components of $\aut(X)$ that belong to $G$, so that there is a pullback square
\[
\begin{tikzcd}
\aut_G(X) \ar[d] \ar[r] & G \ar[d] \\
\aut(X) \ar[r] & \piaut{X}.
\end{tikzcd}
\]
\end{defn}
The cover of $\Baut{}{X}$ associated to $G\leqslant \piaut{X}$ is weakly equivalent to the classifying space $\Baut{G}{X}$ of the monoid $\aut_G(X)$ (defined e.g.~using the geometric bar construction \cite[\S7]{may75})) and we will tacitly identify these two spaces.

Here is an outline of the proof of the main results:


We fix a subgroup $G\leqslant \piaut{X}$ that acts nilpotently on the rational homology of $X$. In \S\ref{sec:unipotent groups} we show that $G$ uniquely determines a unipotent algebraic subgroup $U\leqslant\piaut{X_\Q}$ such that the homomorphism $\rat\colon \piaut{X} \to \piaut{X_\Q}$ restricts to a $\Q$-isomorphism $G\to U$, and we give several equivalent characterizations of this $U$. In \S\ref{sec:virtually nilpotent covers} we show that the space $\Baut{G}{X}$ is virtually nilpotent. In \S\ref{sec:rationalization of covers}, we start incorporating the action of the deck transformation group $\deckgrp{G}{X}$ on $\Baut{G}{X}$; in particular, we show that $\Baut{G}{X}$ is $\deckgrp{G}{X}$-equivariantly rationally equivalent to $\Baut{U}{X_\Q}$.
In \S\ref{sec:lie models} we construct a Lie model for the space $\Baut{U}{X_\Q}$. In general, there is no action of the deck transformation group on this Lie model, but there is an action of a larger group defined in terms of the minimal Quillen model of $X$. A crucial step in the proof is to relate this algebraically defined action to the action of the deck transformation group. The key is a lemma about conjugation actions on bar constructions, which we prove in \S\ref{sec: bar lemma}.
In \S\ref{sec:algebraicity}, we use this to prove Theorem \ref{thm:A}. 
In \S\ref{sec:algebraic lie model}, we observe that if $U$ is the unipotent radical of $\piaut{X_\Q}$, then the deck transformation group can be made to act on the Lie model and this leads to the proof of Theorem \ref{thm:main}.

\subsection{Unipotent groups of self-homotopy equivalences} \label{sec:unipotent groups}
Recall that $X$ is assumed to be a simply connected finite CW-complex. We begin by discussing how the homomorphism induced by rationalization,
\[\rat\colon \piaut{X} \to \piaut{X_\Q},\]
can be used to characterize subgroups of $\piaut{X}$ that act nilpotently on $H_*(X;\Q)$.


Since $X$ is simply connected, $\piaut{X}$ is isomorphic to the group $\piautp{X}$ of pointed homotopy classes of pointed self-homotopy equivalences. This group acts on the homotopy and homology groups of $X$. Moreover, the Hurewicz homomorphism $\pi_n(X)\to H_n(X)$ is $\piaut{X}$-equivariant. The \emph{spherical homology} $SH_n(X)$ is by definition the image of the Hurewicz homomorphism. It is a $\piaut{X}$-submodule of $H_n(X)$ and a quotient $\piaut{X}$-module of $\pi_n(X)$.

\begin{prop} \label{prop:nilpotent action}
Let $G\leqslant \piaut{X}$ be a subgroup and let $R$ be a subring of $\Q$. The following are equivalent:
\begin{enumerate}[(i)]
\item $G$ acts nilpotently on $H_n(X;R)$ for all $n$.
\item $G$ acts nilpotently on $\pi_n(X)\tensor R$ for all $n$.
\item $G$ acts nilpotently on $SH_n(X;R)$ for all $n$. 
\end{enumerate}
For $R=\Q$ these conditions are equivalent to the following:
\begin{enumerate}[(i)]
\setcounter{enumi}{3}
\item $\rat(G)$ is contained in a unipotent algebraic subgroup of $\piaut{X_\Q}$.
\end{enumerate}
\end{prop}

\begin{proof}
It is well-known that the first two conditions are equivalent (e.g.~the argument in \cite[Theorem 2.1]{hilton} goes through) and they clearly imply the third. For the converse, assume inductively that $G$ acts nilpotently on the $R$-homotopy groups of the Postnikov section $P_{n-1}(X)$ and use that the Hurewicz homomorphism sits in an exact sequence
\[H_{n+1}(P_{n-1}(X)) \to \pi_n(X) \to H_n(X).\]
This gives rise to an exact sequence
\[H_{n+1}(P_{n-1}(X);R) \to \pi_n(X)\tensor R  \to SH_n(X;R) \to 0,\]
where the left and right terms are nilpotent $G$-modules. It follows from \cite[Proposition I.4.3]{hmr75} that $\pi_n(X)\tensor R$ is a nilpotent $G$-module.

Finally, we prove the equivalence between the first condition and the fourth when $R=\Q$. If \(G\) acts nilpotently on \(H_*(X;\Q)\), then the image of \(G\) in \(\GL(H_*(X;\Q))\) lies in a unipotent algebraic subgroup \(U''\). The preimage \(U'\) of \(U''\) in \(\piaut{X_\Q}\) contains $\rat(G)$ and it is a unipotent algebraic subgroup since it is an extension of \(U''\) by the kernel of the \(\piaut{X_\Q}\)-representation \(H_*(X;\Q)\), which is unipotent by Theorem \ref{thm:sul-wil}\eqref{thm:sul-wil:unipotent ker}. Conversely, since $H_*(X;\Q)$ is an algebraic representation of $\piaut{X_\Q}$, any unipotent algebraic subgroup of the latter acts nilpotently on it.
\end{proof}

\begin{prop} \label{prop:unipotent groups of self-equivalences}
If $G\leqslant \piaut{X}$ is a subgroup that acts nilpotently on $SH_*(X;\Q)$, then there is a unique unipotent algebraic subgroup $U\leqslant \piaut{X_\Q}$ that satisfies the following equivalent conditions:
\begin{enumerate}[(i)]
\item $U$ is minimal among the unipotent algebraic subgroups that contain $\rat(G)$.

\item $\rat(G)$ is a Zariski dense subgroup of $U$.

\item $\rat(G)$ is an arithmetic subgroup of $U$.

\item $G$ is a finite-index subgroup of $\rat^{-1}(U)$.

\item $\rat(G)\leqslant U$ and the induced homomorphism $\rat\colon G\to U$ is a $\Q$-isomorphism.
\end{enumerate}
Moreover, $G$ is finitely generated in this situation.

Conversely, if $U\leqslant \piaut{X_\Q}$ is a unipotent algebraic subgroup, then there is a unique commensurability class of subgroups $G\leqslant \piaut{X}$ such that the above conditions are satisfied.
\end{prop}


\begin{proof}


If $G\leqslant \piaut{X}$ acts nilpotently on $SH_*(X;\Q)$, then $\rat(G)$ is contained in a unipotent algebraic subgroup of $\piaut{X_\Q}$ by Proposition \ref{prop:nilpotent action}. If we let $U$ be the intersection of all unipotent algebraic subgroups of $\piaut{X_\Q}$ that contain $\rat(G)$, then $U$ is clearly the unique minimal unipotent algebraic subgroup that contains $\rat(G)$.

If $U$ is minimal among the unipotent algebraic subgroups that contain $\rat(G)$, then the Zariski closure $\overline{\rat(G)}$ is contained in $U$ since the latter is Zariski closed. But then $\overline{\rat(G)}$ must be unipotent and therefore equal to $U$ by minimality of $U$.


Now assume that $\rat(G)$ is Zariski dense in $U$. Clearly, $\rat(G)$ is contained in $U\cap \rat(\piaut{X})$ and the latter is an arithmetic subgroup of $U$ by Lemma \ref{lemma:arithmetic preimage}. Arithmetic subgroups of unipotent groups are nilpotent and finitely generated (cf.~Lemma \ref{lemma: integer points of unipotent group}), and subgroups of finitely generated nilpotent groups are always finitely generated, so it follows that $\rat(G)$ is finitely generated. Since $\rat(G)$ is Zariski dense in $U$, Lemma \ref{lemma: integer points of unipotent group} shows that $\rat(G)$ is an arithmetic subgroup of $U$. We note in passing that we can use the exact sequence
$$
1 \to \ker(\rat)\cap G \to G \to \rat(G) \to 1
$$
to deduce that $G$ is finitely generated as well. Indeed, the kernel is finite by Theorem \ref{thm:arithmeticity} and we have just seen that $\rat(G)$ is finitely generated.

Next, if $\rat(G)$ is an arithmetic subgroup of $U$, then it must have finite index in $\rat(\rat^{-1}(U)) = U \cap \rat(\piaut{X})$, because the latter is also an arithmetic subgroup of $U$. As $\rat$ has finite kernel, this implies that $G$ has finite index in $\rat^{-1}(U)$.

The homomorphism $\rat\colon G\to U$ is the composite of the inclusion $G \to \rat^{-1}(U)$ followed by the homomorphism $\rat\colon \rat^{-1}(U) \to U$. If $G$ has finite index in $\rat^{-1}(U)$, then the inclusion is in particular a $\Q$-isomorphism. The homomorphism $\rat\colon \rat^{-1}(U) \to U$ has finite kernel and image an arithmetic subgroup of $U$, so it is a $\Q$-isomorphism by Lemma \ref{lemma: integer points of unipotent group}. This shows that $\rat\colon G\to U$ is a $\Q$-isomorphism if $G$ has finite index in $\rat^{-1}(U)$.

Finally, suppose $\rat$ restricts to a $\Q$-isomorphism $G\to U$. If $U'\leqslant \piaut{X_\Q}$ is a unipotent algebraic subgroup such that $\rat(G)\leqslant U'$, then $G\to U$ factors as $G\to U\cap U'$ followed by the inclusion $U\cap U' \to U$, implying the latter is $\Q$-surjective. But this is then a $\Q$-surjective inclusion of uniquely divisible groups, so it must be an equality. This shows that $U$ is minimal among the unipotent algebraic subgroups that contain $\rat(G)$.
By that we have gone full circle, showing the five conditions are equivalent.

Conversely, given $U$ and two subgroups $G$ and $G'$ that satisfy the equivalent conditions, the fourth condition shows that $G$ and $G'$ are finite-index subgroups of $\rat^{-1}(U)$. This implies that $G\cap G'$ has finite index in both $G$ and $G'$. So the commensurability class determined by $U$ is precisely the set of finite-index subgroups of $\rat^{-1}(U)$.

 
\end{proof}




\subsection{Virtual nilpotence of covers} \label{sec:virtually nilpotent covers}
As before, $X$ is a simply connected finite CW-complex and $G$ is a subgroup of $\piaut{X}$. The purpose of this section is to show that the space $\Baut{G}{X}$ is virtually nilpotent if $G$ acts nilpotently on the rational spherical homology of $X$. This extends a result of Dror and Zabrodsky \cite[Theorem D]{drza79}, which says that \(\Baut{G}{X}\) is nilpotent if $G$ acts nilpotently on the integral homology of $X$.


\begin{lemma} \label{lemma:vn}
Let $G$ be a group acting on a finitely generated abelian group $A$. If $G$ acts nilpotently on $A\tensor \Q$, then there is a finite-index subgroup $K\leqslant G$ that acts nilpotently on $A$.
\end{lemma}

\begin{proof}
Let $K$ be the kernel of the action of $G$ on the torsion subgroup $A_{\tor}$. Since the group $A_{\tor}$ is finite, so is its automorphism group, so the exact sequence
\[
1 \to K \to G \to \Aut\big(A_{\tor}\big)
\]
shows that $K$ has finite index in $G$. If $G$ acts nilpotently on $A\tensor \Q$, then there exists a filtration of $G$-modules
\[
0 = V_0 \subseteq V_1 \subseteq \ldots \subseteq V_n = A\tensor \Q
\]
such that $G$ acts trivially on  $V_i/V_{i-1}$ for each $i$. Let \(W_i \subseteq A\) be the preimage of \(V_i\) under the homomorphism \(A \to A\tensor \Q\).
This yields a filtration of $G$-modules
\[
W_0 \subseteq W_1 \subseteq \ldots \subseteq W_n = A,
\]
and it follows that $W_i/W_{i-1}$ has trivial \(G\)-action for $1\leq i \leq n$. Note that \(W_0 = A_\tor\), so if $G$ acts nilpotently on $A_\tor$ it follows that it acts nilpotently on \(A\) as well. In general, the $G$-action on \(A_\tor\) need not be nilpotent, but the action of \(K\) on it is trivial by definition. Thus, extending the filtration by \(W_{-1} = 0\) yields a filtration witnessing that \(K\) acts nilpotently on $A$.
\end{proof}

\begin{prop}\label{prop:virtually nilpotent}
If $G\leqslant \piaut{X}$ is a subgroup that acts nilpotently on $SH_*(X;\Q)$, then the space \(\Baut{G}{X}\) is virtually nilpotent.
It is nilpotent if $G$ acts nilpotently on $SH_*(X;\Z)$.
\end{prop}


\begin{proof}
We will in fact prove the slightly stronger statement that $\Baut{G}{X}$ admits a finite cover which is nilpotent. By Lemma \ref{lemma:vn}, there is a finite-index subgroup $K\leqslant G$ that acts nilpotently on $SH_*(X;\Z)$. By Proposition \ref{prop:nilpotent action}, the group \(K\) acts nilpotently on \(H_*(X; \Z)\). It follows from \cite[Theorem D]{drza79} that the space $\Baut{K}{X}$ is nilpotent. The space $\Baut{K}{X}$ is weakly equivalent to the finite cover of $\Baut{G}{X}$ that corresponds to the finite-index subgroup $K\leqslant G$.
\end{proof}

\begin{eg}
The space $\Baut{G}{X}$ is not virtually nilpotent in general. For example, if we take $X=S^2 \vee S^2$ and $G = \piaut{X}$, then $G \cong \GL_2(\Z)$, which is not virtually nilpotent.
\end{eg}

\begin{eg}
The space $\Baut{G}{X}$ need not be nilpotent even if $G$ acts nilpotently on $SH_*(X;\Q)$. To see this, consider the Moore space \(X = M(\Z/3\Z, 2)\), i.e., the homotopy cofiber of a degree $3$ self-map of $S^2$. This space is rationally equivalent to point, so $G=\piaut{X}$ acts nilpotently on $SH_*(X;\Q) = 0$ for trivial reasons.
The group of self-equivalences is not difficult to compute (see \emph{e.g.}~\cite[Theorem 2]{sieradski72}); there is an isomorphism
\[
 \piaut{X} \cong \left(\Z / 3\Z\right)^\times \ltimes \Z/3\Z \cong \Sigma_3,
\]
showing $\pi_1(\Baut{}{X}) = \piaut{X}$ is not nilpotent. The criterion for nilpotency of $\Baut{G}{X}$ in Proposition \ref{prop:virtually nilpotent}
is not satisfied because the action on $SH_2(X;\Z) = \Z/3\Z$ is through the projection onto $(\Z/3\Z)^\times$ and this action is not nilpotent.
\end{eg}

\subsection{Equivariant rationalization of covers} \label{sec:rationalization of covers}
In this section, we will show that $\Baut{G}{X}$ is rationally equivalent to $\Baut{U}{X_\Q}$ when $G$ and $U$ are as in Proposition \ref{prop:unipotent groups of self-equivalences}. To keep track of the action of the deck transformation group, we need a model for $\Baut{G}{X}$ that is appropriately functorial in $G\leqslant \piaut{X}$. Recall that the subgroups of $\piaut{X}$ form the objects the \emph{orbit category}, where a morphism $G\to H$ is given by a morphism of left $\piaut{X}$-sets $\piaut{X}/G\to \piaut{X}/H$.
\begin{defn} \label{def:bautfunc}
We define a functor from the orbit category of $\piaut{X}$ to the category of spaces over $\Baut{}{X}$ by sending $G\leqslant \piaut{X}$ to the space
$$\Bautfunc{G}{X} = B\big(*,\aut(X),\piaut{X}/G\big)$$
defined using the geometric bar construction \cite[\S7]{may75} of the topological monoid $\aut(X)$ acting on $\piaut{X}/G$ via the canonical map $\aut(X)\to \piaut{X}$.
\end{defn}
The space $\Bautfunc{G}{X}$ can be thought of as a functorial model for the cover of $\Baut{}{X}$ associated to $G\leqslant \piaut{X}$. The map
$$B\aut_G(X) = B\big(*,\aut_G(X),*\big)\to B\big(*,\aut(X),\piaut{X}/G\big) = \Bautfunc{G}{X},$$
induced by the inclusion of monoids $\aut_G(X)\to \aut(X)$ and the map of $\aut_G(X)$-spaces $* \to \piaut{X}/G$ that selects the coset $G$, is a weak equivalence, so both source and target are models for the classifying space of the monoid $\aut_G(X)$. The advantage of $\Bautfunc{G}{X}$ is that it carries an action of the group $\deckgrp{G}{X}$ of automorphisms of the $\piaut{X}$-set $\piaut{X}/G$. We remind the reader that this group, the `deck transformation group', may be identified with $N_{\piaut{X}}(G)/G$, the normalizer of $G$ in $\piaut{X}$ modulo $G$.



In what follows, we fix a subgroup $G \leqslant \piaut{X}$ that acts nilpotently on $SH_*(X;\Q)$ (or equivalently on $H_*(X;\Q)$ or $\pi_*(X)\tensor \Q$) and we let $U\leqslant \piaut{X_\Q}$ be the minimal unipotent algebraic subgroup that contains $\rat(G)$ as in Proposition \ref{prop:unipotent groups of self-equivalences}. 

\begin{lemma} \label{lemma:deck groups}
The homomorphism $\rat\colon \piaut{X} \to \piaut{X_\Q}$ carries the normalizer of $G$ to the normalizer of $U$. In particular, there is an induced group homomorphism $\deckgrp{G}{X} \to \deckgrp{U}{X_\Q}$.
\end{lemma}

\begin{proof}
Let $g\in N_{\piaut{X}}(G)$. We need to show that $\rat(g) \in N_{\piaut{X_\Q}}(U)$, i.e., that ${}^gU = \rat(g) U \rat(g)^{-1}$ agrees with $U$. By assumption, $U$ is the minimal unipotent algebraic subgroup of $\piaut{X_\Q}$ that contains $\rat(G)$. It follows that ${}^gU$ is a minimal unipotent algebraic subgroup that contains $\rat(g)\rat(G) \rat(g)^{-1} = \rat(gGg^{-1}) = \rat(G)$, so ${}^gU = U$ by the uniqueness part of Proposition \ref{prop:unipotent groups of self-equivalences}.
\end{proof}

The previous lemma implies that $\deckgrp{G}{X}$ acts on $\Bautfunc{U}{X_\Q}$ via the induced homomorphism $\rat\colon\deckgrp{G}{X} \to \deckgrp{U}{X_\Q}$. We remind the reader that we call a map between connected spaces a \emph{rational equivalence} if it induces a $\Q$-isomorphism on all homotopy groups \emph{and} all homology groups (see Definition \ref{def:rational equivalence}).

\begin{prop} \label{lemma:upgraded main lemma}
There is a zig-zag of $\deckgrp{G}{X}$-equivariant rational equivalences between $\Bautfunc{G}{X}$ and $\Bautfunc{U}{X_\Q}$.
\end{prop}

\begin{proof}
Let $r\colon X\to X_\Q$ be a rationalization. We may assume that $r$ is a cofibration. This holds in many models for rationalizations (\emph{e.g.}~cellular rationalization \cite[Theorem 9.7]{fht01} or the $H_*(-;\Q)$-localization of Bousfield \cite{bousfield75}), but if necessary, one can achieve this by abstract nonsense, \emph{e.g.}, by factoring $r$ as a cofibration followed by a weak homotopy equivalence. Now consider the pullback square
\[
\begin{tikzcd}
\aut(r) \ar[r, "q"] \ar[d] & \aut(X_\Q) \ar[d, "r^*"] \\
\aut(X) \ar[r, "r_*"] & \map(X,X_\Q)_{re},
\end{tikzcd}
\]
where \(\map(X,X_\Q)_{re}\) denotes the space of rational equivalences from $X$ to $X_\Q$ and \(\aut(r)\) denotes the space of self-equivalences of $r$ viewed as an object in the category of maps. The map $r^*$ is a fibration since $r$ is a cofibration, and it is a weak homotopy equivalence by standard properties of localizations. Hence, the left vertical map is a weak homotopy equivalence as well. The map $r_*$ is in general not a bijection on $\pi_0$, but its restriction to each component is a rational equivalence to the component it hits by \cite[Theorem II.3.11]{hmr75}. It follows that the top horizontal map $q$ has the same property.

This yields a zig-zag of grouplike monoids
\begin{equation} \label{eq:q}
\begin{tikzcd}
\aut(X) & \ar[l, "\sim"'] \aut(r) \ar[r, "q"] & \aut(X_\Q),
\end{tikzcd}
\end{equation}
where the left map is a weak equivalence and the right map $q$ induces an isomorphism on $\pi_k(-)\tensor \Q$ for all $k>0$ and may be identified with $\rat \colon \piaut{X} \to \piaut{X_\Q}$ on $\pi_0$. This induces a zig-zag
\begin{equation} \label{eq:zz}
B\big(*,\aut(X),\piaut{X}/G\big) \xleftarrow{\sim} B\big(*,\aut(r),\piaut{X}/G\big)  \xrightarrow{\psi} B\big(*,\aut(X_\Q),\piaut{X_\Q}/U\big),
\end{equation}
where $\psi$ is induced by the monoid map $q\colon \aut(r) \to \aut(X_\Q)$ and the $\aut(r)$-equivariant map $\rat \colon \piaut{X}/G \to  \piaut{X_\Q}/U$.
The homomorphism $\pi_1(\psi)$ may be identified with $\rat\colon G \to U$, which is a $\Q$-isomorphism by Proposition \ref{prop:unipotent groups of self-equivalences}. For $k>1$, we may identify $\pi_k(\psi)$ with $\pi_{k+1}(q)\colon \pi_{k+1}(\aut(r))\to \pi_{k+1}(\aut(X_\Q))$, which is a $\Q$-isomorphism by the above.

We recognize the leftmost term of \eqref{eq:zz} as $\Bautfunc{G}{X}$ and the rightmost term as $\Bautfunc{U}{X_\Q}$. Moreover, the maps in the zig-zag are clearly $\deckgrp{G}{X}$-equivariant. We have just shown that $\psi$ induces a $\Q$-isomorphism on all homotopy groups. By Proposition \ref{prop:virtually nilpotent}, the source of $\psi$ is virtually nilpotent and the target is nilpotent, so it follows from Lemma \ref{lemma:virtual nilpotence lemma} that $\psi$ is a rational homology equivalence as well.

\end{proof}

\subsection{Equivariant Lie models for covers} \label{sec:lie models}
As before, $X$ is a simply connected finite CW-complex, $G \leqslant \piaut{X}$ is a subgroup that acts nilpotently on $SH_*(X;\Q)$, and $U\leqslant \piaut{X_\Q}$ is the unique minimal unipotent algebraic subgroup of $\piaut{X_\Q}$ that contains $\rat(G)$.

We identify $X_\Q$ with the nerve of the minimal Quillen model $L$ for $X$. Recall from Theorem \ref{thm:sul-wil}\eqref{thm:sul-wil:lie algebra} that the Lie algebra of $\piaut{X_\Q}$ is isomorphic to $H_0(\Der^c L)$, so the Lie algebra of $U$ may be identified with a Lie subalgebra $\lie u \leqslant H_0(\Der^c L)$. 
The following definition can be viewed as a dg Lie algebra companion to Definition~\ref{def: autU}.


\begin{defn}\label{def: Deru}
Let $\Der_{\lie u}^c L$ denote the dg Lie subalgebra of $\Derc L \langle 0\rangle$ defined by the pullback
\[
\begin{tikzcd}
\Der_{\lie u}^c L \ar[d] \ar[r] & \lie u \ar[d] \\
\Derc L\langle 0\rangle \ar[r] & H_0(\Derc L).
\end{tikzcd}
\]
\end{defn}

\begin{prop} \label{prop:nilpotence of the dgl}
The dg Lie algebra $\Der_{\lie u}^c L$ is nilpotent.
\end{prop}

\begin{proof}
The degree zero component, $(\Der_{\lie u}^c L)_0$, may be identified with the Lie algebra of the preimage $\widetilde{U} = p^{-1}(U)$ under the projection $p\colon \AUT L \to \AUT^h L$. By Theorem~\ref{thm:sul-wil}\eqref{thm:sul-wil:aut L}, the homomorphism $p$ has unipotent kernel. It follows that $\widetilde{U}$ is an extension of unipotent groups and is hence unipotent (see \emph{e.g.}~\cite[\S6.45]{mil17}). Since $(\Der^c L)_k$ is an algebraic representation of $\widetilde{U}$ for each $k>0$, it follows that $\widetilde{U}$ acts unipotently on it (see \cite[Proposition~14.3]{mil17}). Hence, $(\Der_{\lie u}^c L)_0 = \Lie{\widetilde{U}}$ is nilpotent and acts nilpotently on $(\Der^c L)_k$ for each $k>0$.
\end{proof}

Let $\AutNorm{U}{L}\leqslant \Aut(L)$ denote the preimage of the normalizer of $U\leqslant \piaut{X_\Q}$ under the homomorphism $\Aut(L) \to \piaut{X_\Q}$. By design, the group $\AutNorm{U}{L}$ acts on $\piaut{X_\Q}$ and on $\Derc L$ by conjugation, fixing the subgroup $U$ and the Lie subalgebra $\lie u$, respectively. If $U$ is normal in $\piaut{X_\Q}$ then $\AutNorm{U}{L} = \Aut(L)$.

\begin{prop} \label{prop:ubridge}
There is a weak equivalence of grouplike topological monoids,
\begin{equation} \label{eq:ubridge}
|\exp_\bullet (\Der_{\lie u}^c L)| \to \aut_U(X_\Q),
\end{equation} 
which is equivariant with respect to the conjugation action of \(\AutNorm{U}{L}\) on the domain and codomain. In particular, $\Baut{U}{X_\Q}$ is weakly equivalent to the nerve of the dg Lie algebra $\Der_{\lie u}^c (L)$.

\end{prop}

\begin{proof}
The action of the nilpotent dg Lie algebra $\lie g = \Der_{\lie u}^c L$ on $L$ by outer derivations induces an action of the simplicial group $\exp_\bullet(\lie g)$ on the simplicial set $\MC_\bullet(L)$ (\emph{cf.}~\cite[\S3.5]{be20}), giving rise to a morphism of grouplike monoids
\begin{equation} \label{eq:bridge}
\alpha\colon |\exp_\bullet(\lie g)| \to \aut\big(|\MC_\bullet(L)|\big).
\end{equation}
The map \eqref{eq:bridge} is known to induce an isomorphism on $\pi_k(-)$ for $k>0$. To see this, one can apply \cite[Proposition 3.7]{be20} to the Sullivan model $\Lambda = C^*(L)$. Let us add that this is essentially equivalent to the statement, going back to Tanr\'e \cite[VII.4.(4)]{tan83}, that $\Der^c L\langle 1 \rangle$ is a Lie model for the simply connected cover of $\Baut{}{\MC_\bullet(L)}$.

By construction, $\pi_0(\aut_U(X_\Q)) = U$ and $H_0(\Der_{\lie u}^c L) = \lie u$. By Proposition \ref{prop:homotopy of exp}, there is an isomorphism \[\pi_0(\exp_\bullet(\Der_{\lie u}^c L)) \cong \exp(H_0(\Der_{\lie u}^c L)) = \exp(\lie u) = U.\]
One checks that the map induced by \eqref{eq:bridge} on $\pi_0$ may be identified with the inclusion of $U$ into $\piaut{X_\Q}$. Hence, \eqref{eq:bridge} corestricts to a weak equivalence $|\exp_\bullet(\lie g)| \to \aut_U(X_\Q)$.
The statement about $\AutNorm{U}{L}$-equivariance is quickly verified by inspection. The last statement follows by applying the classifying space functor to the monoid map \eqref{eq:ubridge} and noting that $B|\exp_\bullet(\lie g)|$ is weakly equivalent to the nerve $| \MC_\bullet(\lie g)|$.
\end{proof}

\begin{rmk}
Note that Proposition \ref{prop:homotopy of exp} implies that $B|\exp_\bullet(\lie g)|$ is nilpotent and rational whenever $\lie g$ is a nilpotent dg Lie algebra, so the above gives an alternative proof that the space $\Baut{U}{X_\Q}$ is nilpotent and rational.
\end{rmk}

In particular, forgetting the group actions, we get a Lie model for $\Baut{G}{X}$.
\begin{cor} \label{cor:nonequivariant lie model}
The space $\Baut{G}{X}$ is rationally equivalent to the nerve of the dg Lie algebra $\Der_{\lie u}^c(L)$.
\end{cor}

\begin{proof}
Combine Proposition \ref{prop:ubridge} and Proposition \ref{lemma:upgraded main lemma}.
\end{proof}

\begin{eg}
Let $\torelli{X}$ denote the \emph{homotopy Torelli monoid}, meaning the submonoid of $\aut(X)$ of those self-equivalences that act trivially on $H_*(X;\Z)$.
The space $B\torelli{X}$ is nilpotent by Proposition \ref{prop:virtually nilpotent}. Corollary \ref{cor:nonequivariant lie model} yields a Lie model for $B\torelli{X}$ described by the following: it agrees with $\Der^c L$ in positive degrees and in degree zero it consists of all derivations $\theta$ of $L$ that commute with the differential and are decomposable in the sense that $\theta(L) \subseteq [L,L]$. To see this, note that action of a derivation $\theta\in \Der L$ on the reduced homology of $X$ may be identified with the induced action on the indecomposables $L/[L,L]$.
\end{eg}

\begin{rmk} \label{remark:recover ffm}
Corollary \ref{cor:nonequivariant lie model} recovers the Lie model for $\Baut{G}{X}$ of \cite[Theorem 0.1]{ffm21}.
Indeed, one can check that $\Der_{\lie u}^c L$ agrees with the model $\Der^ {\mathcal{G}}L \widetilde{\times} sL$ of \cite[Theorem 0.1]{ffm21}, where $\mathcal{G}\leqslant \Aut^h(L)$ corresponds to $U\leqslant \piaut{X_\Q}$ under the isomorphism $\Aut^h(L) \cong \piaut{X_\Q}$. 
Let us also remark that when $U$ is the unipotent radical of $\piaut{X_\Q}$, one can show that the dg Lie algebra \(\Der_u^c L\) of Definition \ref{def: Deru} agrees with the dg Lie algebra \(\mathcal{D}er L\)
considered in \cite[Definition 7.4]{ffm21} if the filtration \cite[(30)]{ffm21} is chosen to be a composition series as in \eqref{eq:composition series}.


\end{rmk}




\subsection{A lemma on conjugation actions on bar constructions}\label{sec: bar lemma}
In this section, we will prove a lemma that will be a key ingredient in the proof of the main results. Before stating it, let us give a few words of motivation. Suppose that
\[
    1 \to G' \to G \xrightarrow{f} G''  \to 1
\]
is a split short exact sequence of groups, giving rise to a homotopy fiber sequence
\[
    BG' \to BG \xrightarrow{Bf} BG''.
\]
It is well-known that the conjugation action of \(G''\) on \(BG'\) models the holonomy action of \(G''\) on the homotopy fiber of \(Bf\).
On the other hand, the left action of \(G''\) on \(G''\hOrbits G = B(G'',G,*)\) also models the holonomy action.
The evident map \(BG' \to G''\hOrbits G\) is, however, not equivariant; it becomes equivariant if \(G''\) acts by simultaneous left multiplication and conjugation on the target. The lemma below offers a resolution of this seeming incongruity by showing that these two different $G''$-actions on \(G''\hOrbits G\) give weakly equivalent \(G''\)-spaces.

Now, let $\monoid$ be a topological monoid, let $\lspace$ be a left $\monoid$-space and $\rspace$ a right $\monoid$-space. Let $\inv{\monoid}$ denote the group of invertible elements in $\monoid$. There is an action of $\inv{\monoid}$ on the triple $(\rspace,\monoid,\lspace)$ defined by letting $g\in \inv{\monoid}$ act by
\begin{eqnarray*}
&\monoid \xrightarrow{c_g} \monoid, & h \mapsto ghg^{-1}, \\
&\lspace\xrightarrow{\ell_g} c_g^*(\lspace), & x \mapsto gx, \\
&\rspace \xrightarrow{r_g} c_g^*(\rspace), & y \mapsto yg^{-1}.
\end{eqnarray*}
This induces an action of $\inv{\monoid}$ on the geometric bar construction $B\big(\rspace,\monoid,\lspace\big)$, which we will refer to as the conjugation action.

\begin{lemma} \label{lemma:key lemma 2}
There is a natural zig-zag of $\inv{\monoid}$-equivariant homotopy equivalences connecting $B\big(\rspace,\monoid,\lspace\big)$, with the conjugation action, and the same space with the trivial action.
\end{lemma}

\begin{rmk}
The lemma implies the fact that, for each individual $g\in \inv{\monoid}$, the map $B(r_{g},c_g,\ell_g)$ is homotopic to the identity of $B\big(\rspace,\monoid,\lspace\big)$, but note that the lemma is not a formal consequence of this fact. Indeed, it is easy to find spaces with a group action where each individual multiplication map is homotopic to the identity but where the action can not be trivialized via a zig-zag of equivariant homotopy equivalences, e.g., $S^1$ with the antipodal action of $\Z/2\Z$.
\end{rmk}

\begin{proof}
We exploit the fact that the geometric bar construction extends to topological categories (see e.g.~\cite[Chapter V.2]{may96}). Consider the topological category $\monoid\times [1]$, where $[1]$ is the discrete category associated to the poset $\{0<1\}$. Specifying a representation $A$ of $\monoid\times [1]$, by which we mean a continuous functor from $\monoid\times [1]$ to the category of spaces, amounts to specifying a morphism of $\monoid$-spaces
\[ A_0 \xrightarrow{f} A_1.\]
We extend $\lspace$ and $\rspace$ to representations $\underline{\lspace}$ and $\underline{\rspace}$ of $\monoid\times [1]$ by using the respective identity maps. The inclusion functors
\[\monoid  = \monoid\times \{0\} \xrightarrow{i_0} \monoid\times [1] \xleftarrow{i_1} \monoid \times \{1\} = \monoid \]
together with the identity maps $\lspace\to i_0^*(\underline{\lspace})$, $\rspace\to i_0^*(\underline{\rspace})$, $\lspace\to i_1^*(\underline{\lspace})$, $\rspace\to i_1^*(\underline{\rspace})$,
induce maps on bar constructions
\begin{equation} \label{eq:equivariant zig-zag}
B\big(\rspace,\monoid,\lspace\big) \xrightarrow{Bi_0} B\big(\underline{\rspace},\monoid\times [1],\underline{\lspace}\big)  \xleftarrow{Bi_1} B\big(\rspace,\monoid,\lspace\big).
\end{equation}
The middle term is isomorphic to the cylinder $B\big(\rspace,\monoid,\lspace\big) \times I$ and the maps $Bi_0$ and $Bi_1$ may be identified with the bottom and top inclusions. In particular, both $Bi_0$ and $Bi_1$ are homotopy equivalences.

The idea is now to define an action of $\inv{\monoid}$ on the triple $(\underline{\rspace}, \monoid\times [1],\underline{\lspace})$ so that the maps \eqref{eq:equivariant zig-zag} become equivariant if the left copy of $B\big(\rspace,\monoid,\lspace\big)$ is given the trivial action and the right copy the conjugation action. Once we have found such an action, the proof will be complete.

For $g\in \inv{\monoid}$, there is a unique continuous functor
$$H_g\colon \monoid\times [1] \to \monoid\times [1]$$
such that $H_g^*(A)$ is the $\monoid\times [1]$-representation
\begin{equation*}
A_0 \xrightarrow{g f} c_g^*(A_1), \quad a \mapsto gf(a),
\end{equation*}
whenever $A$ is a representation of $\monoid\times [1]$ as above. The existence and uniqueness of $H_g$ follows from the Yoneda lemma for topologically enriched categories. One checks that $H_1 = 1$ and $H_g H_h = H_{gh}$ for all $g,h\in\inv{\monoid}$ so this defines an action of $\inv{\monoid}$ on $\monoid\times [1]$.
Next, the commutative square of $\monoid$-spaces
$$
\begin{tikzcd}
\lspace \arrow[r, "1"] \arrow[d, "1"]
& \lspace \arrow[d, "g"] \\
\lspace \arrow[r, "g"]
& c_g^*(\lspace)
\end{tikzcd}
$$
defines a morphism of $\monoid\times [1]$-representations $L_g\colon \underline{\lspace} \to H_g^*(\underline{\lspace})$. A morphism of contravariant representations $R_g\colon \underline{\rspace} \to H_g^*(\underline{\rspace})$
is defined similarly. Verifying that
$$(R_g,H_g,L_g)\colon (\underline{\rspace},\monoid\times [1],\underline{\lspace})\to (\underline{\rspace},\monoid\times [1],\underline{\lspace})$$
defines an action of $\inv{\monoid}$ with the desired properties is routine and left to the reader.
\end{proof}

\subsection{Algebraicity of cohomology and homotopy groups}
\label{sec:algebraicity}
In this section, we will give Theorem \ref{thm:A} a more precise formulation (Theorem \ref{thm:general algebraicity} below) and prove it. Let us first review the notion of simple homotopy groups.
\begin{defn} \label{def:simple homotopy groups}
For a path-connected space $B$, we define the \emph{simple homotopy groups} \(\pi_n'(B)\) to be the abelian groups
\[\pi_n'(B) = \pi_n(B)/[\pi_1(B),\pi_n(B)]\]
obtained by trivializing the action of $\pi_1(B)$. Explicitly, $\pi_1'(B) = \pi_1(B)^{ab}$ and $\pi_n'(B) = [S^n,B]$ (free homotopy classes of maps) for $n>1$.
\end{defn}

Note that $B$ is a simple space if and only if the canonical homomorphism $\pi_n(B) \to \pi_n'(B)$ is an isomorphism for all $n$. The simple homotopy groups are functorial for maps that do not necessarily preserve the basepoint. In particular, if a group $\mathcal{E}$ acts on $B$, then $\pi_n'(B)$ is a representation of $\mathcal{E}$. This is the reason we consider the simple homotopy groups instead of the ordinary homotopy groups.

We now proceed towards the proof of Theorem \ref{thm:general algebraicity}. We fix the following notation for the remainder of the section.

{\noindent {\bf Notation.}}
\begin{itemize}
\item[--] $X$ is a simply connected finite CW-complex.
\item[--] $L$ is the minimal Quillen model of $X$.
\item[--] $G\leqslant \piaut{X}$ is a subgroup that acts nilpotently on $SH_*(X;\Q)$.
\item[--] $U\leqslant \piaut{X_\Q}$ is the minimal unipotent algebraic subgroup with $\rat(G)\leqslant U$.
\item[--] $\lie u$ is the Lie algebra of $U$. 
\end{itemize}
Recall that the existence and uniqueness of $U$ is guaranteed by Proposition \ref{prop:unipotent groups of self-equivalences} and that $\lie u$ may be viewed as a subalgebra of $H_0(\Der^c L)$ by Theorem \ref{thm:sul-wil}(\ref{thm:sul-wil:lie algebra}).

We also remind the reader that $\deckgrp{U}{X_\Q}$ denotes the group of automorphisms of the $\piaut{X_\Q}$-set $\piaut{X_\Q}/U$ and that this group may be identified with $N_{\piaut{X_\Q}}(U)/U$.

Recall that $\AutNorm{U}{L}$ denotes the preimage of the normalizer $N_{\piaut{X_\Q}}(U)$ under the homomorphism $\Aut(L) \to \piaut{X_\Q}$. There is a surjective homomorphism
\[p\colon \AutNorm{U}{L} \to \deckgrp{U}{X_\Q}.\]
By construction, the space
\[\Bautfunc{U}{X_\Q} = B\big(*,\aut(X_\Q),\piaut{X_\Q}/U \big)\]
carries an action of $\deckgrp{U}{X_\Q}$ and hence an action of $\AutNorm{U}{L}$ via $p$. On the other hand, the group $\AutNorm{U}{L}$ acts on $\Der_{\lie u}^c L$ by conjugation, and hence it acts on the nerve of $\Der_{\lie u}^c L$.




\begin{lemma} \label{lemma:weak model}
The $\AutNorm{U}{L}$-space $p^*\Bautfunc{U}{X_\Q}$ is weakly equivalent to the nerve of the dg Lie algebra $\Der_{\lie u}^c L$, on which $\AutNorm{U}{L}$ acts by conjugation.
\end{lemma}


\begin{proof}
The rationalization $X_\Q$ may be identified with $|\MC_\bullet(L)|$. Proposition \ref{prop:ubridge} shows that $B|\exp_\bullet(\Der_{\lie u}^c L)|$ is weakly equivalent to $B\aut_U(X_\Q)$ as an $\AutNorm{U}{L}$-space, where $\AutNorm{U}{L}$ acts by conjugation. Now consider the weak equivalence
\begin{equation} \label{eq:unconjugating}
B\aut_U(X_\Q) = B\big(*,\aut_U(X_\Q),*\big) \to B\big(*,\aut(X_\Q),\piaut{X_\Q}/U\big),
\end{equation}
induced by the inclusion of $\aut_U(X_\Q)$ into $\aut(X_\Q)$ and the map $*\to \piaut{X_\Q}/U$ that selects the coset $U$.

The space $B\big(*,\aut(X_\Q),\piaut{X_\Q}/U\big)$ has two commuting actions of $\AutNorm{U}{L}$: the conjugation action as in Lemma \ref{lemma:key lemma 2} via the homomorphism $\AutNorm{U}{L} \to \inv{\aut(X_\Q)}$, and the action by automorphisms on the left $\aut(X_\Q)$-space $\piaut{X_\Q}/U$ via the homomorphism $\AutNorm{U}{L}\to \deckgrp{U}{X_\Q}$. The weak equivalence \eqref{eq:unconjugating} is $\AutNorm{U}{L}$-equivariant if $B\big(*,\aut(X_\Q),\piaut{X_\Q}/U\big)$ is given the diagonal action. (It is not equivariant with respect to the conjugation action, because the coset $U\in \piaut{X_\Q}/U$ is not fixed by the left action of $\AutNorm{U}{L}$.) Lemma \ref{lemma:key lemma 2} gives a zig-zag of $\AutNorm{U}{L}$-equivariant weak equivalences connecting $B\big(*,\aut(X_\Q),\piaut{X_\Q}/U\big)$, with the diagonal action of $\AutNorm{U}{L}$, and same space where $\AutNorm{U}{L}$ only acts on $\piaut{X_\Q}/U$.
\end{proof}

Lemma \ref{lemma:weak model} does not give full information about the homotopy type of the $\deckgrp{U}{X_\Q}$-space $\Bautfunc{U}{X_\Q}$, but it is enough for proving that the cohomology groups and the simple homotopy groups are algebraic representations.

\begin{thm} \label{thm:general algebraicity} \leavevmode
\begin{enumerate}[(i)]
\item The group $\deckgrp{U}{X_\Q}$ may be identified with the $\Q$-points of an affine algebraic group.
\item \label{thm:general algebraicity (ii)} The induced homomorphism $\deckgrp{G}{X}\to \deckgrp{U}{X_\Q}$ has finite kernel and image an arithmetic subgroup.

\item \label{thm:general algebraicity (iii)} The representations of $\deckgrp{G}{X}$ in the rational cohomology groups and the simple rational homotopy groups of the space $\Baut{G}{X}$ are restrictions of algebraic representations of $\deckgrp{U}{X_\Q}$.
\end{enumerate}
\end{thm}



\begin{proof}
By Theorem \ref{thm:sul-wil}\eqref{thm:sul-wil:Q-points}, $\piaut{X_\Q}$ may be identified with the $\Q$-points of the algebraic group $\mathcal{G}  = \AUT(L)/\AUT_h(L)$. By hypothesis, $U$ corresponds to the $\Q$-points of a unipotent algebraic subgroup $\mathcal{U}$ of $\mathcal{G}$. It follows from Lemma \ref{lemma:normalizer} and Lemma \ref{lemma:quotient} that $\deckgrp{U}{X_\Q} = N_{\piaut{X_\Q}}(U)/U$ may be identified with the $\Q$-points of the algebraic group $N_{\mathcal{G}}(\mathcal{U})/\mathcal{U}$. This proves the first claim.

The kernel of $N_{\piaut{X}}(G)/G \to N_{\piaut{X_\Q}}(U)/U$ injects into the set $\rat^{-1}(U)/G$, which is finite by Proposition \ref{prop:unipotent groups of self-equivalences}. Lemma \ref{lemma:deck groups} implies that there is an inclusion
\begin{equation} \label{eq:inclusion into arithmetic}
\rat(N_{\piaut{X}}(G)) \leqslant N_{\piaut{X_\Q}}(U) \cap \rat(\piaut{X}).
\end{equation}
The latter group is an arithmetic subgroup of $N_{\piaut{X_\Q}}(U)$ by Theorem~\ref{thm:arithmeticity} and Lemma \ref{lemma:arithmetic preimage}, so if we can show that \eqref{eq:inclusion into arithmetic} is the inclusion of a finite index subgroup, then it follows that $\rat(N_{\piaut{X}}(G))$ is an arithmetic subgroup of $N_{\piaut{X_\Q}}(U)$. It then follows from Lemma \ref{lemma:image arithmetic} that the image of $\rat(N_{\piaut{X}}(G))$ in $\deckgrp{U}{X_\Q} = N_{\piaut{X_\Q}}(U)/U$ is arithmetic, and then we are done because this agrees with the image of $\deckgrp{G}{X} = N_{\piaut{X}}(G)/G$. To show that the index is finite, observe that \eqref{eq:inclusion into arithmetic} is obtained by applying $\rat$ to the inclusion
\begin{equation} \label{eq:finite index}
N_{\piaut{X}}(G) \leqslant \rat^{-1}\big(N_{\piaut{X_\Q}}(U)\big).
\end{equation}
The latter group acts on the set of subgroups of $\rat^{-1}(U)$ by conjugation. The stabilizer of $G$ is $N_{\piaut{X}}(G)$ and the orbit of $G$ consists of subgroups of the form $gGg^{-1}$ where $g\in \piaut{X}$ satisfies $g\rat^{-1}(U)g^{-1} = \rat^{-1}(U)$. Note that all subgroups of this form have the same (finite) index in $\rat^{-1}(U)$. There are only finitely many subgroups of a given finite index in a finitely generated group, so the orbit must be finite. It follows that the index of the inclusion \eqref{eq:finite index} is finite and in turn that the same holds for \eqref{eq:inclusion into arithmetic}.



The action of $\deckgrp{G}{X}$ on the homology of $\Baut{G}{X}$ may be computed as the induced action on the homology of the $\deckgrp{G}{X}$-space $\Bautfunc{G}{X}$. The latter is connected to $\Bautfunc{U}{X_\Q}$ by a sequence of $\deckgrp{G}{X}$-equivariant rational homology equivalences by Proposition \ref{lemma:upgraded main lemma}, so it suffices to show that the cohomology groups of $\Bautfunc{U}{X_\Q}$ are algebraic representations of $\deckgrp{U}{X_\Q}$.


Since the cohomology of $B\exp_\bullet(\lie g)$ is naturally isomorphic to the Chevalley-Eilenberg cohomology of $\lie g$ (see Proposition \ref{prop:ce cochains}), Lemma \ref{lemma:weak model} implies that there is an $\AutNorm{U}{L}$-equivariant isomorphism
\begin{equation} \label{eq:rep}
H^k\big(\Bautfunc{U}{X_\Q}\big) \cong H_{\CE}^k(\Der_{\lie u}^c L)
\end{equation}
for every $k$. Note that $\AutNorm{U}{L} = \AUTNorm{\mathcal{U}}{L}(\Q)$ and that there is an isomorphism of algebraic groups $\AUTNorm{\mathcal{U}}{L}/\mathcal{W} \to N_{\mathcal{G}}(\mathcal{U})/\mathcal{U}$, where $\AUTNorm{\mathcal{U}}{L}$ and $\mathcal{W}$ are the preimages of $N_{\mathcal{G}}(\mathcal{U})$ and $\mathcal{U}$ under $\pi\colon \AUT(L) \to \mathcal{G}$. Since $\pi$ has unipotent kernel, it follows that $\mathcal{W}$ is unipotent. The right-hand side of \eqref{eq:rep} is manifestly an algebraic representation of the algebraic group $\AUTNorm{\mathcal{U}}{L}$. The action of $\mathcal{W}(\Q)$ is trivial, because the action of $\AUTNorm{\mathcal{U}}{L}(\Q) = \AutNorm{U}{L}$ on the left-hand side factors through $\deckgrp{U}{X_\Q}$, which may be identified with $\AUTNorm{U}{L}(\Q)/\mathcal{W}(\Q)$. Since $\mathcal{W}(\Q)$ is dense in $\mathcal{W}$ (by \emph{e.g.}~\cite[Theorem 17.93]{mil17}), it follows that the representation is trivial as an algebraic representation of $\mathcal{W}$. Thus, it is an algebraic representation of $\AUTNorm{U}{L}/\mathcal{W} \cong N_{\mathcal{G}}(\mathcal{U})/\mathcal{U}$.
This shows that the cohomology groups of $\Bautfunc{U}{X_\Q}$ are algebraic representations of $\deckgrp{U}{X_\Q}$.

To show algebraicity of the simple homotopy groups, one argues as above using the $\AutNorm{U}{L}$-equivariant isomorphism
\[\pi_{k+1}'(\Bautfunc{U}{X_\Q}) \cong H_k(\Der^c L)/\big[H_0(\Der^c L),H_k(\Der^c L)\big],\]
which can be deduced from Lemma \ref{lemma:weak model} and Proposition \ref{prop:homotopy of exp}.
\end{proof}


For later use, we record the following fact that was observed in the above proof.
\begin{slemma} \label{lemma:unipotent W}
The kernel of $p\colon \AutNorm{U}{L}\to \deckgrp{U}{X_\Q}$ is unipotent.
\end{slemma}

Before we proceed to the proof of Theorem \ref{thm:main}, let us discuss two interesting applications of Theorem \ref{thm:general algebraicity}.

\subsubsection{Arithmeticity of higher homotopy groups} \label{sec:higher arithmeticity}
As discussed in the introduction, Sullivan and Wilkerson proved that the homomorphism from $\piaut{X} = \pi_1\Baut{}{X}$ to $\piaut{X_\Q} = \pi_1\Baut{}{X_\Q}$ has finite kernel and image an arithmetic subgroup. It is natural to ask whether there is a similar statement for the higher homotopy groups of $\Baut{}{X}$. The following shows that these are `arithmetic representations' of $\piaut{X}$, in the sense that they each admit a map, with finite kernel, onto a lattice inside an algebraic representation of the group $\piaut{X_\Q}$.

\begin{cor} \label{cor:algebraicity}
Let $X$ be a simply connected finite CW-complex.
The representations of $\piaut{X}$ in the rational cohomology and homotopy groups of the universal cover of $\Baut{}{X}$ are restrictions of finite-dimensional algebraic representations of the algebraic group $\piaut{X_\Q}$.
\end{cor}

\begin{proof}
Apply Theorem \ref{thm:general algebraicity} to the trivial group $G=1$. Since $\Baut{1}{X}$ is simply connected, the simple homotopy groups agree with the homotopy groups.
\end{proof}



\subsubsection{Homotopy Torelli spaces} \label{sec:torelli}
Kupers and Randal-Williams \cite{krw20} recently established algebraicity and nilpotence results for Torelli groups of the manifolds $W_g$, \emph{i.e.}, groups of diffeomorphisms that act trivially on integral homology. By analogy, we can consider the homotopy Torelli monoid of the space $X$, i.e., the monoid $\torelli{X}$ of self-homotopy equivalences of $X$ that act trivially on $H_*(X;\Z)$.
The following shows that a homotopical counterpart of the main result of \cite{krw20} is valid for \emph{arbitrary} simply connected finite complexes.


Note that the homotopy Torelli space sits in a homotopy fiber sequence
\[B\torelli{X} \to \Baut{}{X} \to B\GL^{\piaut{X}}(H_*(X;\Z)),\]
where $\GL^{\piaut{X}}(H_*(X;\Z))$ is the image of the homomorphism $\piaut{X} \to \GL(H_*(X;\Z))$.

\begin{cor} \label{cor:torelli}
Let $X$ be a simply connected finite CW-complex. 

\begin{enumerate}[(i)]
\item The group $\GL^{\piaut{X_\Q}}(H_*(X;\Q))$ of isomorphisms of $H_*(X;\Q)$ that are realizable by a self-homotopy equivalence of $X_\Q$ is a linear algebraic group.

\item The homomorphism $\GL^{\piaut{X}}(H_*(X;\Z)) \to \GL^{\piaut{X_\Q}}(H_*(X;\Q))$ has finite kernel and image an arithmetic subgroup.

\item The representations of $\GL^{\piaut{X}}(H_*(X;\Z))$ in
the rational (co)homology and rational simple homotopy groups of the space $B\torelli{X}$ are restrictions of algebraic representations of the linear algebraic group $\GL^{\piaut{X_\Q}}(H_*(X;\Q))$.

\end{enumerate}
\end{cor}

\begin{proof}
Clearly, $\torelli{X} = \aut_G(X)$, where $G\leqslant \piaut{X}$ is the kernel of the action on $H_*(X;\Z)$. Let $U\leqslant \piaut{X_\Q}$ be the kernel of the algebraic representation $H_*(X;\Q)$. Then $U$ is unipotent by Theorem~\ref{thm:sul-wil}\eqref{thm:sul-wil:unipotent ker}. The cokernel of $G\to \rat^{-1}(U)$ may be identified with the kernel of
\[ \GL^{\piaut{X}}(H_*(X;\Z)) \to \GL^{\piaut{X_\Q}}(H_*(X;\Q)),\]
which is clearly finite. Hence, Proposition \ref{prop:unipotent groups of self-equivalences} implies that $U$ is the minimal algebraic unipotent subgroup of $\piaut{X_\Q}$ containing $G$, so Theorem \ref{thm:general algebraicity} applies.

\end{proof}

\subsection{Space-level algebraicity}
\label{sec:algebraic lie model}
In this section, we will prove Theorem \ref{thm:main}. We retain the notation of the previous section. The following two lemmas are key.

\begin{lemma} \label{lemma:section}
If the algebraic group $\deckgrp{U}{X_\Q}$ is reductive, then the homomorphism
$$p\colon \AutNorm{U}{L} \to \deckgrp{U}{X_\Q}$$
admits a section $\sigma$ that is a morphism of algebraic groups.
\end{lemma}

\begin{proof}
The kernel of $p$ is unipotent by Lemma \ref{lemma:unipotent W}, so if $\deckgrp{U}{X_\Q}$ is reductive, then $p$ may be identified with the map from $\AutNorm{U}{L}$ to its maximal reductive quotient and an algebraic section exists by Theorem \ref{thm:Levi}.
\end{proof}

\begin{lemma} \label{lemma:final zig-zag}
If the homomorphism
$$p\colon \AutNorm{U}{L} \to \deckgrp{U}{X_\Q}$$
admits a section $\sigma$, then there is a zig-zag of $\deckgrp{G}{X}$-equivariant rational equivalences that connects $\Bautfunc{G}{X}$ to the nerve of the dg Lie algebra $\sigma^*\Der_{\lie u}^c(L)$.
Moreover, if $\sigma$ is a morphism of algebraic groups, then the graded components of $\sigma^* \Der_{\lie u}^c(L)$ are algebraic representations of $\deckgrp{U}{X_\Q}$.
\end{lemma}

\begin{proof}
By Proposition \ref{lemma:upgraded main lemma}, the spaces $\Bautfunc{G}{X}$ and $\Bautfunc{U}{X_\Q}$ are connected by a zig-zag of $\deckgrp{G}{X}$-equivariant rational equivalences. That $\sigma$ is a section of $p$ means that $p\sigma =1$, so by Lemma \ref{lemma:weak model}, we get weak equivalences of $\deckgrp{U}{X_\Q}$-spaces
$$\Bautfunc{U}{X_\Q} = \sigma^* p^* \Bautfunc{U}{X_\Q} \sim \sigma^* B|\exp_\bullet(\Der_{\lie u}^c L)| = B|\exp_\bullet(\sigma^* \Der_{\lie u}^c L)|.$$
The graded components of $\Der_{\lie u}^c(L)$ are manifestly algebraic representations of the algebraic group $\AutNorm{U}{L}$. Hence, if $\sigma$ is a morphism of algebraic groups, then the graded components of $\sigma^*\Der_{\lie u}^c(L)$ are algebraic representations of $\deckgrp{U}{X_\Q}$.
\end{proof}





We will now prove Theorem \ref{thm:main}. First, we need to specify the cover $\Baut{u}{X}$ and the groups $\Gamma(X)$ and $R(X)$.

\begin{defn}
We define
\begin{itemize}
\item[--] $\Baut{u}{X} = \Baut{G}{X}$,
\item[--] $\Gamma(X) = \deckgrp{G}{X}$,
\item[--] $R(X) = \deckgrp{U}{X_\Q}$,
\end{itemize}
where $U$ is the unipotent radical of $\piaut{X_\Q}$ and $G$ is the preimage of $U$ under the homomorphism $\rat\colon \piaut{X} \to \piaut{X_\Q}$.
\end{defn}

\begin{prop}
$R(X)$ is a reductive algebraic group and $\Gamma(X)$ may be identified with an arithmetic subgroup of $R(X)$.
\end{prop}

\begin{proof}
The unipotent radical $U$ is the maximal normal unipotent algebraic subgroup of $\piaut{X_\Q}$. In particular, $U$ is normal in $\piaut{X_\Q}$ and $G = \rat^{-1}(U)$ is normal in $\piaut{X}$. Hence, $\Gamma(X) = \piaut{X}/G$ and $R(X)= \piaut{X_\Q}/U$. We recognize the latter as the maximal reductive quotient of $\piaut{X_\Q}$ (cf.~\S\ref{sec: unipotent and reductive gps}). Since $G = \rat^{-1}(U)$, we may identify $\Gamma(X) = \piaut{X}/G$ with the image of $\piaut{X}$ in $R(X) = \piaut{X_\Q}/U$. This is an arithmetic subgroup of $R(X)$ by Theorem \ref{thm:general algebraicity}\eqref{thm:general algebraicity (ii)}.
\end{proof}
By that, the first part of Theorem \ref{thm:main} has been verified. Next, we need to construct the dg Lie algebra $\lie g (X)$.

\begin{defn} \label{def:g(X)}
Define
$$\lie g (X) = \sigma^* \Der_{\lie u}^c(L),$$
where $\lie u$ is the Lie algebra of the unipotent radical $U$ of $\piaut{X_\Q}$, and where
$$\sigma\colon R(X) \to \AutNorm{U}{L}$$
is a section of $p$ as in Lemma \ref{lemma:section}, which exists since $R(X)$ is reductive.
\end{defn}

\begin{thm} \label{thm:main zig-zag}
The dg Lie algebra $\lie g(X)$ is nilpotent, its graded components are algebraic representations of $R(X)$, and
there is a zig-zag of $\Gamma(X)$-equivariant rational equivalences that connects $\Baut{u}{X}$ to the nerve of $\lie g(X)$.
\end{thm}

\begin{proof}
Nilpotence is shown in Proposition \ref{prop:nilpotence of the dgl}. The rest is immediate from Lemma \ref{lemma:section} and Lemma \ref{lemma:final zig-zag}.
\end{proof}

\begin{cor} \label{cor:homotopy orbits}
The space $\Baut{}{X}$ is rationally equivalent to the homotopy orbit space $\nerve{\lie g(X)}_{h\Gamma(X)}$.
\end{cor}

\begin{proof}
We can recover $\Baut{}{X}$ as the homotopy orbit space
\begin{equation*}
\Baut{}{X} \sim \Baut{u}{X}_{h\Gamma(X)}.
\end{equation*}
Applying homotopy orbits to the zig-zag in Theorem \ref{thm:main zig-zag} yields the result.
\end{proof}
This finishes the proof of Theorem \ref{thm:main}.

\begin{rmk} \label{rmk:challenge}

It is natural to ask whether the rational homotopy type of the $\deckgrp{G}{X}$-space $\Bautfunc{G}{X}$ can be represented by a dg Lie algebra of algebraic representations of $\deckgrp{U}{X_\Q}$ for arbitrary subgroups $G\leqslant \piaut{X}$ that act nilpotently on $SH_*(X;\Q)$. The main problem is to rectify the action up to homotopy of $\deckgrp{U}{X_\Q}$ on $\Der_{\lie u}^c L$. As we have seen, this can be done if $\deckgrp{U}{X_\Q}$ is reductive. One can show that $\deckgrp{U}{X_\Q}$ is reductive if and only if $U$ is the unipotent radical of a parabolic subgroup of $\piaut{X_\Q}$, but $\deckgrp{U}{X_\Q}$ will not be reductive in general. Still, even if the action can not be rectified, one can ask whether a suitable modification of $\Der_{\lie u}^c L$ could accommodate an action of $\deckgrp{U}{X_\Q}$. Since our main applications are in the case when $\deckgrp{U}{X_\Q}$ is reductive, we will not pursue this question further in this paper. We leave it as a challenge to the interested reader.
\end{rmk}



\subsection{A commutative cochain algebra model for the classifying space} \label{sec:cdga models}
In this section, we will use Theorem \ref{thm:main} to construct a commutative differential graded algebra model for $\Baut{}{X}$. Before we can formulate the result, we need to review the definition of the commutative cochains of a simplicial set with local coefficients, following Halperin \cite{halperin83}.
\subsubsection{Commutative cochains with local coefficients} \label{sec:local systems}
For a simplicial set $K$, we let $(\Delta \downarrow K)$ denote the simplex category. An object is a simplex $\sigma\colon \Delta^n \to K$ and a morphism from $\sigma$ to $\tau$ is a commutative diagram
\[
\begin{tikzcd}
\Delta^n \ar[dr, "\sigma"] \ar[d, "\varphi"']  \\ \Delta^m \ar[r, "\tau"'] & K.
\end{tikzcd}
\]

\begin{defn}[{\cite[pp.~150--151]{halperin83}}]
A local system on a simplicial set $K$ with values in a category $\CC$ is a functor \[F\colon (\Delta \downarrow K)^{op} \to \CC.\]
The global sections of a local system $F$ on $K$ is defined as the inverse limit
\[F(K) = \varprojlim_{\sigma \in (\Delta \downarrow K)} F_\sigma,\]
provided this limit exists in $\CC$.
\end{defn}

\begin{rmk} \label{remark:local systems}
To compare this with the notion of a local system of coefficients defined in terms of representations of the fundamental groupoid, one should observe that the groupoidification of the category $(\Delta \downarrow K)$ is a model for the fundamental groupoid of $K$, \emph{cf.}~\cite[\S III.1]{gj99}. This means that the category of local systems $F$ such that $F(\varphi)$ is invertible for every morphism $\varphi$ in $(\Delta\downarrow K)$ is equivalent to the category of representations of the fundamental groupoid. However, since we do not require invertibility of $F(\varphi)$ in general, the notion of a local system is more general.
\end{rmk}

\begin{eg}\leavevmode
\begin{enumerate}[(i)]
\item If $\Gamma$ is a discrete group and $M$ is an object of $\CC$ with an action of $\Gamma$, then there is an associated local system on the nerve $B\Gamma$, also denoted $M$, where $M_\sigma = M$ for all $\sigma = (\gamma_1,\ldots,\gamma_n)\colon \Delta^n \to B\Gamma$ and $d_i\colon M_\sigma\to M_{d_i\sigma}$ is multiplication by $\gamma_n$ for $i=n$ and the identity otherwise.
\item A simplicial object $X_\bullet \colon \Delta^{op} \to \CC$ determines a local system on every simplicial set $K$ by restriction along $(\Delta \downarrow K)^{op}\to \Delta^{op}$. In particular, the simplicial cdga $\Omega_\bullet$ determines a local system $\Omega^*$ of cdgas on every simplicial set $K$. The global sections $\Omega^*(K)$ is the usual model for Sullivan's cdga of polynomial differential forms on $K$, also denoted $A_{PL}^*(K)$.

\item More generally, if $F$ is a local system of cochain complexes on $K$, then $\Omega^*(K;F)$ may be defined as the global sections of the local system $\Omega^*\tensor F$, \emph{cf.}~\cite[Definition 13.10]{halperin83}.
If $A$ is local system of cochain algebras, then $\Omega^*(K;A)$ is a cochain algebra, which is commutative if $A$ is.
\end{enumerate}
\end{eg}

\begin{prop}
Let $K$ be a simplicial set and let $A$ be a local system of $\Q$-cochain complexes on $K$. The cochain complex $\Omega^*(K;A)$ is naturally quasi-isomorphic to $C^*(K;A)$. If $A$ is a local system of dg algebras, then $\Omega^*(K;A)$ and $C^*(K;A)$ are quasi-isomorphic as dg algebras.
\end{prop}

\begin{proof}
Integration provides a natural quasi-isomorphism 
\[\Omega^*(K;A) \to C^*(K;A)\]
\emph{cf.}~\cite[Theorem 14.18]{halperin83}. The integration map is not multiplicative on the cochain level, but the argument in \cite[\S 10(d)]{fht01} goes through with coefficients $A$, providing a zig-zag of dg algebra quasi-isomorphisms.
\end{proof}

In particular, for a discrete group $\Gamma$ and a cochain complex $M$ of $\Q[\Gamma]$-modules, the cohomology of $\Omega^*(\Gamma,M) = \Omega^*(B\Gamma;M)$
agrees with group cohomology $H^*(\Gamma,M)$ (as defined in \emph{e.g.}~\cite[VII.5]{bro82}).
If $A$ is a commutative cochain algebra over $\Q$ with an action of $\Gamma$, then $\Omega^*(\Gamma,A)$ is a commutative cochain algebra model for $C^*(\Gamma,A)$. The following is a commutative cochain algebra version, in characteristic zero, of well-known results for homology and singular chains; \emph{cf.}~\emph{e.g.}~\cite[\S VII.7]{bro82} or \cite[\S 13]{may75}.
\begin{prop} \label{prop:commutative cochains}
Let $X$ be a space with a $\Gamma$-action.
The commutative cochain algebras $\Omega^*(X_{h\Gamma})$ and $\Omega^*(\Gamma,\Omega^*(X))$ are quasi-isomorphic.
\end{prop}

\begin{proof}
The local system $F$ of the fibration $X_{h\Gamma} \to B\Gamma$
in the sense of \cite[p.248]{halperin83}
may be identified up to quasi-isomorphism with the local system associated to the $\Gamma$-cdga $\Omega^*(X)$. Granted this, the combination of \cite[Lemma 19.21]{halperin83} and \cite[Theorem 13.12]{halperin83} applied to the local system $F$ shows that
\[\Omega^*(X_{h\Gamma}) \cong F(B\Gamma) \sim \Omega^*(B\Gamma;F) \sim \Omega^*(B\Gamma;\Omega^*(X)). \]
\end{proof}

Now we are ready to formulate and prove Corollary \ref{cor:main} from the introduction.

\begin{thm}\label{thm:cdga model}
The commutative dg algebra $\Omega^*(\Baut{}{X})$ is quasi-isomorphic to the commutative dg algebra $ \Omega^*(\Gamma(X),C_{\CE}^*(\lie g(X)))$.

\end{thm}

\begin{proof}
It follows from Corollary \ref{cor:homotopy orbits} that $\Omega^*(\Baut{}{X})$ is quasi-isomorphic to $\Omega^*(\nerve{\lie g(X)}_{h\Gamma(X)})$. By Proposition \ref{prop:commutative cochains}, the latter commutative cochain algebra is quasi-isomorphic to $\Omega^*(\Gamma(X),\Omega^*(\nerve{\lie g(X)}))$. The proof is completed by using the natural quasi-isomorphism $C_{\CE}^*(\lie g(X)) \to \Omega^*(\nerve{\lie g(X)})$ of Proposition \ref{prop:ce cochains}.

\end{proof}



Let us compare with the approach of G\'omez-Tato--Halperin--Tanr\'e \cite{gtht00}. They study the fiberwise rationalization of $\Baut{}{X} \to B\piaut{X}$ by using local systems, in the sense of \S\ref{sec:local systems}, of cdgas over $B\piaut{X}$. As discussed in Remark \ref{remark:local systems}, such local system are more general, and more complicated, than cdgas with an action of $\piaut{X}$.
Our approach is to instead consider the fiberwise rationalization of $\Baut{}{X} \to B\Gamma(X)$. This has two significant advantages: Firstly, it lets us work with cdgas with an action of $\Gamma(X)$ rather than local systems of cdgas over $B\Gamma(X)$. Secondly, semisimplicity of algebraic $\Gamma(X)$-representations lets us construct minimal models using familiar methods. Indeed, the counterpart of Problem 3 of \cite[\S7.3]{gtht00} in our setting admits the following solution:


\begin{thm}
The minimal Sullivan model $\MM(X)$ for $\Baut{u}{X}$ admits an algebraic action of $R(X)$ such that the spatial realization $\spatial{\MM(X)}$ is rationally equivalent to $\Baut{u}{X}$ as a $\Gamma(X)$-space and $\Omega^*(\Baut{}{X})$ is quasi-isomorphic to $\Omega^*(\Gamma(X),\MM(X))$ as a commutative cochain algebra.
\end{thm}

\begin{proof}
The minimal model is unique up to isomorphism, so it is enough to construct a particular minimal model that admits an algebraic $R(X)$-action with the desired properties. Let $\lie g(X)$ be the Lie model for $\Baut{u}{X}$ in Theorem \ref{thm:main zig-zag}. By semisimplicity of $\Rep_\Q(R(X))$, we can find a contraction of chain complexes of algebraic $R(X)$-representations
\[
    \begin{tikzcd}
        \lie g(X) \ar[r, shift left, "f"] \ar[loop left, "h"] & H_*(\lie g(X)). \ar[l, shift left, "g"]
    \end{tikzcd}
\]
Applying \cite[Theorem 1.3]{be14}, we get a contraction of cocommutative chain coalgebras in $\Rep_\Q(R(X))$,
\[
    \begin{tikzcd}
        C_*^{\CE}(\lie g(X)) \ar[r, shift left, "F'"] \ar[loop left, "H'"] & (\Lambda^c(V),d'), \ar[l, shift left, "G'"]
    \end{tikzcd}
\]
where $V$ is the suspension of $H_*(\lie g(X))$. The differential $d'$ is given by the explicit formula
\[d' = FtG +FtHtG + Ft(Ht)^2G + \cdots,\]
and $F'$, $G'$, $H'$ are given by similar formulas.
Here, $t$ is the quadratic part of the differential of $C_*^{\CE}(\lie g(X))$ that encodes the Lie bracket, $F$ and $G$ are the morphisms of dg coalgebras induced by $f$ and $g$, and $H$ is the `symmetrized tensor trick homotopy' \cite[\S5]{be14}. The dual of  $(\Lambda^c(V),d')$ is a minimal model $\MM(X)$ for $\Baut{u}{X}$. By construction, it has an algebraic action of $R(X)$ and the dual of $G'$ is a quasi-isomorphism of cochain algebras of $R(X)$-representations $C_{CE}^*(\lie g (X)) \to \MM(X)$. This induces a quasi-isomorphism of commutative cochain algebras $\Omega^*(\Gamma(X),C_{CE}^*(\lie g (X))) \to \Omega^*(\Gamma(X),\MM(X))$, so it follows from Theorem \ref{thm:cdga model} that $\Omega^*(\Gamma(X),\MM(X))$ is a cdga model for $\Omega^*(\Baut{}{X})$.
Applying spatial realization, we get an $R(X)$-equivariant weak equivalence $\spatial{\MM(X)} \to \spatial{C_{CE}^*(\lie g (X) )} \cong \langle \lie g (X) \rangle$ and the latter is a $\Gamma(X)$-equivariant model for $\Baut{u}{X}$ by Theorem \ref{thm:main zig-zag}.

We stress that this construction is completely explicit once the contraction between $\lie g(X)$ and its homology has been fixed.
\end{proof}

\subsection{The cohomology ring of the classifying space} \label{sec:cohomology}

The following result, which was stated as Corollary \ref{cor:cohomology ring} in the introduction, is perhaps the most striking consequence of Theorem \ref{thm:cdga model}.

\begin{thm} \label{thm:cohomology ring with coefficients}
There is an isomorphism of graded algebras
\[H^*(\Baut{}{X};\Q) \cong H^*(\Gamma(X),H_{\CE}^*(\lie g(X))).\]
\end{thm}

\begin{proof}
The Chevalley--Eilenberg complex \(C_{\CE}^*(\lie g(X))\) is a cochain complex of finite-dimensional algebraic representations of \(R(X)\).
Since \(R(X)\) is reductive, every such representation is semisimple, and thus the cochain complex \(C_{\CE}^*(\lie g(X))\) is split, which means that there is a contraction of cochain complexes in $\Rep_\Q(R(X))$,
\[
    \begin{tikzcd}
        C_{\CE}^*(\lie{g}(X)) \ar[r, shift left, "p"] \ar[loop left, "h"] & H_\CE^*(\lie{g}(X)). \ar[l, shift left, "\nabla"]
    \end{tikzcd}
\]
That is, \(p\) and \(\nabla\) are chain maps such that \(p(z) = [z]\) whenever \(z \in C_{\CE}^*(\lie{g}(X))\) is a cocycle, \(p \nabla = \id\), and \(h\) satisfies \(\id - \nabla p = d h + h d\), see \emph{e.g.}~\cite[Lemma~B.1]{bm20}.

Applying the dg functor \(\Omega^*(\Gamma(X);-)\) yields a new contraction
\[
    \begin{tikzcd}
        \Omega^*(\Gamma(X),C_{\CE}^*(\lie{g}(X))) \ar[r, shift left, "p_*"] \ar[loop left, start anchor={[yshift=-1.8ex]west}, end anchor={[yshift=1.8ex]west}, distance=4.5em, "h_*"] & \Omega^*(\Gamma(X),H_\CE^*(\lie{g}(X))). \ar[l, shift left, "\nabla_*"]
    \end{tikzcd}
\]
We can now use the homotopy transfer theorem (see \emph{e.g.}~\cite{be14}) to produce a $C_\infty$-algebra structure $\{\mu_n\}_{n\geq 2}$ on $$A=\Omega^*(\Gamma(X);H_{\CE}^*(\lie g(X)))$$ and an extension of  $\nabla_*$ to a $C_\infty$-quasi-isomorphism. The multiplication in this $C_\infty$-algebra structure is given by $\mu_2 = p_* m_2 (\nabla_*\tensor \nabla_*)$, where $m_2$ is the multiplication on $\Omega^*(\Gamma(X);C_{\CE}^*(\lie g(X)))$, and one checks that this agrees with the multiplication on $\Omega^*(\Gamma(X);H_{\CE}^*(\lie g(X)))$, coming from viewing $H_{\CE}^*(\lie g (X))$ as a $\Gamma(X)$-cdga with trivial differential. The cdga $(A,\mu_2)$ is of course not equivalent to the $C_\infty$-algebra $(A,\mu_2,\mu_3,\ldots)$ in general, but they have isomorphic cohomology rings.
\end{proof}

\section{The ingredients of the algebraic model}
\label{sec:concrete descriptions}
In this section, we will give more concrete descriptions of the ingredients $R(X)$, $\Gamma(X)$, $\mathfrak{g}(X)$, $\Baut{u}{X}$ of Theorem \ref{thm:main} and we will state counterparts of our main results for finite Postnikov stages. In \S\ref{sec:case studies}, the results of the present section will be used to make explicit calculations.

\subsection{The groups $R(X)$ and $\Gamma(X)$}
Let $H_*(X;\Q)^{ss}$ denote the `semisimplification' of $H_*(X;\Q)$ as a representation of the algebraic group $\piaut{X_\Q}$. Recall from \S\ref{sec: unipotent and reductive gps} that this is the, unique up to isomorphism, semisimple $\piaut{X_\Q}$-representation
\[H_*(X;\Q)^{ss} = \bigoplus_{i=1}^n V_i/V_{i-1},\]
where
\begin{equation} \label{eq:composition series}
0 = V_0 \subset V_1 \subset \cdots \subset V_n = H_*(X;\Q)
\end{equation}
is any choice of composition series.
For a group $G$ and a representation $V$ of $G$, let $\GL^G(V)$ denote the image of the homomorphism $G\to \GL(V)$.

\begin{thm} \label{thm:groups}
There are group isomorphisms
\[R(X) \cong \GL^{\piaut{X_\Q}} \big( H_*(X;\Q)^{ss} \big),\]
\[\Gamma(X) \cong \GL^{\piaut{X}}\big(H_*(X;\Q)^{ss} \big).\]
In other words, $R(X)$, respectively $\Gamma(X)$, may be identified with the group of automorphisms of $H_*(X;\Q)^{ss}$ that are induced by a self-homotopy equivalence of $X_{\Q}$, respectively $X$.
\end{thm}


\begin{proof}
By definition, the group $R(X)$ is the maximal reductive quotient of $\piaut{X_\Q}$.
By Theorem~\ref{thm:sul-wil}(\ref{thm:sul-wil:unipotent ker}), the rational homology \(H_*(X; \Q)\) is an algebraic representation of \(\piaut{X_\Q}\) with unipotent kernel.
Hence by Lemma~\ref{lemma:gr V}, the maximal reductive quotient $R(X)$ of $\piaut{X_\Q}$ may be identified with the image of \(\piaut{X_\Q}\) in \(\GL(H_*(X;\Q)^{ss})\). By definition, the group \(\Gamma(X)\) is the image of \(\piaut{X}\) in \(R(X)\). In view of the above, this agrees with \(\GL^{\piaut{X}}(H_*(X;\Q)^{ss})\).
\end{proof}

Often, in fact in all cases we consider in \S\ref{sec:case studies} except the last, the equivalent conditions in the corollary below are satisfied, which simplifies the descriptions.

\begin{cor} \label{cor:simplification}
The following are equivalent:
\begin{enumerate}[(i)]
\item The $\piaut{X_\Q}$-representation $H_*(X;\Q)$ is semisimple.
\item There is an isomorphism of algebraic groups
$$R(X) \cong \GL^{\piaut{X_\Q}}(H_*(X;\Q)).$$
\item The algebraic group $\GL^{\piaut{X_\Q}}(H_*(X;\Q))$ is reductive.
\end{enumerate}
If the above conditions hold, then $\Gamma(X)$ may be identified with the group of automorphisms of $H_*(X;\Q)$ that are induced by self-homotopy equivalences of $X$. 
\end{cor}

\begin{proof}
If $H_*(X;\Q)$ is semisimple, then it is isomorphic to $H_*(X;\Q)^{ss}$ so Theorem \ref{thm:groups} identifies $R(X)$ with $\GL^{\piaut{X_\Q}}(H_*(X;\Q))$. The second condition implies the third since $R(X)$ is reductive by construction. If $\GL^{\piaut{X_\Q}}(H_*(X;\Q))$ is reductive, then $H_*(X;\Q)$ is a semisimple representation of it by Theorem \ref{thm:linearly reductive}, but then it is also semisimple as a representation of $\piaut{X_\Q}$.
\end{proof}

\begin{rmk}
If the equivalent conditions in Corollary \ref{cor:simplification} are satisfied and in addition $H_*(X;\Z)$ is torsion-free, then $\Gamma(X)$ may be identified with the group of automorphisms of $H_*(X;\Z)$ that are induced by a self-homotopy equivalence.
\end{rmk}

An obvious necessary condition for an automorphism of $H_*(X;\Q)$ to be induced by a self-homotopy equivalence is that it preserves the coproduct in homology, so there is an inclusion of algebraic groups
\begin{equation} \label{eq:aut coalg}
\GL^{\piaut{X_\Q}}(H_*(X;\Q)) \leq \Aut_{coalg}(H_*(X;\Q)),
\end{equation}
where the latter is the group of automorphisms of the homology coalgebra. This may also be identified with the group $\Aut_{alg}(H^*(X;\Q))$ of automorphisms of the cohomology algebra.

\begin{prop} \label{prop:formality}
The inclusion \eqref{eq:aut coalg} is an equality if and only if $X$ is formal.
\end{prop}

\begin{proof}
This follows from \cite[Theorem 12.7]{sul77}.
\end{proof}

\begin{cor}
If $X$ is formal, then there is an isomorphism of algebraic groups
$$R(X) \cong \Aut_{alg}(H^*(X;\Q))$$
if and only if $\Aut_{alg}(H^*(X;\Q))$ is reductive.
\end{cor}

\begin{proof}
Combine Proposition \ref{prop:formality} and Corollary \ref{cor:simplification}.
\end{proof}


\subsection{Normal unipotent fibrations} \label{sec:unipotent fibrations}
In this section we will show that $\Baut{u}{X}$ can be characterized as the classifying space for what we call \emph{normal unipotent fibrations}, defined below.

Consider a fibration of path-connected spaces
\begin{equation} \label{eq:fibration}
X \to E \xrightarrow{p} B.
\end{equation}
The fibration is classified by a homotopy class of maps $f\colon B\to \Baut{}{X}$, which induces a homomorphism on fundamental groups $\pi_1(B)\to \piaut{X}$ and an action of $\pi_1(B)$ on the homology of $X$.

\begin{defn}
We say that the fibration \eqref{eq:fibration} is a \emph{(normal) unipotent fibration} if the image of $\pi_1(B)\to \piaut{X_\Q}$ is contained in a (normal) unipotent algebraic subgroup.
\end{defn}

\begin{prop}
Consider a fibration as in \eqref{eq:fibration}.
\begin{enumerate}
\item The fibration is unipotent if and only if for each $n$, the $\pi_1(B)$-module $H_n(X;\Q)$ is nilpotent, i.e., there exists a filtration of $\pi_1(B)$-submodules
\begin{equation*} \label{eq:filtration}
0 = F_0 \subseteq F_1 \subseteq \cdots \subseteq F_r = H_n(X;\Q)
\end{equation*}
such that $F_i/F_{i-1}$ is a trivial $\pi_1(B)$-module for every $i$.

\item The fibration is normal unipotent if and only if for each $n$, the $\pi_1(B)$-module $H_n(X;\Q)$ is nilpotent and moreover a filtration as above can be chosen to be a filtration of algebraic $\piaut{X_\Q}$-representations.
\end{enumerate}
\end{prop}

\begin{proof}
For the first statement, apply Proposition \ref{prop:nilpotent action} to the image of the homomorphism $\pi_1(B)\to \piaut{X}$.

For the second statement, note that if the fibration is normal unipotent, then the image of $\pi_1(B)\to \piaut{X_\Q}$ is contained in the unipotent radical, which may be described as the kernel of the action of $\piaut{X_\Q}$ on the semisimplification $H_*(X;\Q)^{ss}$ by Lemma \ref{lemma: uniradical} applied to the algebraic $\piaut{X_\Q}$-representation $H_*(X;\Q)$. In particular, $\pi_1(X)$ acts trivially on $H_*(X;\Q)^{ss}$, which is the associated graded with respect to a filtration of $H_*(X;\Q)$ by algebraic $\piaut{X_\Q}$-representations. Conversely, if $\pi_1(B)$ acts trivially on the associated graded of a filtration of $H_*(X;\Q)$ by algebraic $\piaut{X_\Q}$-representations, then the image of $\pi_1(B)\to \piaut{X_\Q}$ is contained in the unipotent radical, because any such filtration can be refined to a composition series.
\end{proof}

\begin{rmk}
As in Proposition \ref{prop:nilpotent action}, one can replace $H_n(X;\Q)$ by spherical homology $SH_n(X;\Q)$ or rational homotopy $\pi_n(X)\tensor\Q$ in the proposition above. The latter can be used to show that a fibration \eqref{eq:fibration} is unipotent if and only if its fiberwise rationalization is nilpotent in the sense of \cite[p.67]{hmr75} or \cite[II.4.3]{bk72}. 
\end{rmk}

\begin{prop}
The fibration \eqref{eq:fibration} is normal unipotent if and only if the classifying map factors over $\Baut{u}{X}$ up to homotopy,
\[
\begin{tikzcd}
& \Baut{u}{X} \ar[d] \\
B \ar[r] \ar[ur, dashed] & \Baut{}{X}.
\end{tikzcd}
\]
Thus, $\Baut{u}{X}$ may be interpreted as the classifying space for normal unipotent fibrations with fiber $X$.
\end{prop}

\begin{proof}
A lift exists if and only if the image of $\pi_1(B) \to \pi_1\Baut{}{X} = \piaut{X}$ is contained in $\pi_1\Baut{u}{X}$. The latter group is equal to the preimage of the unipotent radical of $\piaut{X_\Q}$ under $\piaut{X}\to \piaut{X_\Q}$ by definition of $\Baut{u}{X}$.
\end{proof}


\subsection{The dg Lie algebra $\lie g(X)$}
We will now give a more concrete description of the dg Lie algebra $\lie g(X)$ of Definition \ref{def:g(X)}. Recall the notion of a nilradical of a dg Lie algebra from Definition \ref{def:nilradical}.
\begin{prop}\label{lemma: lie of uniradical}
The dg Lie algebra $\lie g (X)$ agrees with the nilradical $\nil \Der^c L$ and its graded components are given by
$$\lie g(X)_n =
\begin{cases}
(\Der^c L)_n, & n>0, \\
\Lie(\AUT_u L), & n=0,\\
0, & n<0.
\end{cases}
$$
Here, $\Lie(\AUT_u L)$ stands for the Lie algebra of the unipotent radical of $\AUT L$ and this Lie algebra may be identified with
$\nil_{QL} Z_0(\Der L)$, the maximal ideal of $Z_0(\Der L)$ of derivations that act nilpotently on $QL = L/[L,L]$.
\end{prop}

\begin{proof}
We first verify the statement about the graded components of $\lie g(X)$.
Suppressing the action of $R(X)$, we have that $\lie g(X) = \Der_{\lie u}^c L$ is the preimage of $\lie u\subseteq H_0(\Der^c L) = \Lie(\AUT^h L)$ under the morphism of dg Lie algebras $q\colon \Der^c L \langle 0 \rangle \to H_0(\Der^c L)$, where $\lie u$ is the Lie algebra of the unipotent radical $\AUT_u^h L$ of $\AUT^h L$.  In degree zero, $q$ is the morphism $Z_0(\Der^c L) \to H_0(\Der^c L)$, which may be identified with $\Lie(p)\colon \Lie(\AUT L) \to \Lie(\AUT^h L)$, where $p\colon \AUT L \to \AUT^h L$ is the quotient map. Since $p$ has unipotent kernel (see Theorem \ref{thm:sul-wil}), the preimage $p^{-1}(\AUT_u^h L)$ must be equal to $\AUT_u L$, the unipotent radical of $\AUT L$.
Hence, $\lie g(X)_0 = \Lie(p)^{-1}(\Lie (\AUT_u^h L)) = \Lie(p^{-1}(\AUT_u^h L)) = \Lie(\AUT_u L)$.
Since the Lie algebra $H_0(\Der^c L)$ is concentrated in degree zero, it follows that $\lie g(X)_n = \Der^c L\langle 0 \rangle_n$ for $n\ne 0$.

Since $\AUT_u^h L$ is normal in $\AUT^h L$, the Lie algebra $\lie u$ is an ideal in $H_0(\Der^c L)$. It follows that $\Der_{\lie u}^c L = q^{-1}(\lie u)$ is an ideal in $\Der^c L\langle 0 \rangle$. Moreover, $\Der_{\lie u}^c L$ is nilpotent by Proposition \ref{prop:nilpotence of the dgl}. By definition, $\nil \Der^c L$ is the maximal nilpotent ideal, so it follows that $\Der_{\lie u}^c L \subseteq \nil \Der^c L$.

We now show the reverse inclusion $\nil \Der^c L \subseteq \Der_u^c L$. The components in degrees $\ne 0$ are clearly equal, so we only need to show that $(\nil \Der^c L)_0 \subseteq \Lie(\AUT_u L)$. Note that since \(L\) is finitely generated as a graded Lie algebra and of finite type, there is some \(n\) such that the finite-dimensional graded vector space \(L_{\leqslant n} = \bigoplus_{i=1}^n L_i\) is a faithful algebraic representation of \(\AUT L\).
Thus, by Lemma \ref{lemma: uniradical}, \(\Lie(\AUT_u L)\) is precisely the maximal ideal of \(\Lie(\AUT L) = Z_0(\Der^c L)\) consisting of derivations of \(L\) that are nilpotent when restricted to \(L_{\leqslant n}\).
Since \((\nil \Der^c L)_0\) is an ideal in $Z_0(\Der^c L)$ which acts nilpotently on $(\Der^c L)_k \cong (\Der L)_k \oplus L_{k-1}$ for each $k>0$, it in particular acts nilpotently on $L_{\leq n}$. Hence, \((\nilrad)_0 \subseteq \Lie(\AUT_u L)\).

By Lemma \ref{lemma: uniradical} applied to $V=L/[L,L]$, the Lie algebra of $\AUT_u(L)$ may be described as $\nil_{QL} Z_0(\Der^c L)$.
\end{proof}

\begin{cor} \label{cor:Aut L reductive}
If $\Aut L$ is reductive, then
\begin{align*}
R(X) & = \Aut L, \\
\lie g(X) & = \Der^c L\langle 1 \rangle,
\end{align*}
and the action of $R(X)$ on $\lie g(X)$ is the conjugation action.
\end{cor}

\begin{proof}
Each of the homomorphisms $\Aut L \to \piaut{X_\Q} \to R(X)$ has unipotent kernel. If $\Aut L$ is reductive, it has no normal unipotent subgroups. This forces $\Aut L = \piaut{X_\Q} = R(X)$. The description of $\lie g(X)$ follows directly from Proposition \ref{lemma: lie of uniradical}.
\end{proof}



\subsection{Finite Postnikov stages}
The results we have stated have analogs for simply connected finite Postnikov stages $X$ of finite type. We state the results and briefly indicate the necessary modifications of the proofs.

The minimal Sullivan model $\Lambda$ is finitely generated and we may identify $\piaut{X_\Q}$ with the $\Q$-points of the algebraic group $\AUT^h(\Lambda)$ by Theorem \ref{thm:sul-wil II}\eqref{thm:sul-wil:Q-points II}.

As before, let $G\leqslant \piaut{X}$ be a subgroup that acts nilpotently on $SH_*(X;\Q)$ and let $U\leqslant \piaut{X_\Q}$ be the minimal unipotent algebraic subgroup that contains $\rat(G)$ as in Proposition \ref{prop:unipotent groups of self-equivalences}. By Theorem \ref{thm:sul-wil II}\eqref{thm:sul-wil:lie algebra II}, we may regard the Lie algebra $\lie u$ of $U$ as a Lie subalgebra of $H_0(\Der \Lambda)$. Analogously to Definition~\ref{def: Deru}, we define the dg Lie algebra $\Der_{\lie u} \Lambda$ by declaring that there is a pullback square
\[
    \begin{tikzcd}
        \Der_{\lie u} \Lambda \ar[r] \ar[d] & \lie u \ar[d]\\
        \Der \Lambda \langle 0 \rangle \ar[r] & H_0(\Der \Lambda).
    \end{tikzcd}
\]
After identifying $X_\Q$ with the spatial realization of $\Lambda$, Proposition~\ref{prop:ubridge} admits the following analog, where $\AutNorm{U}{\Lambda}$ denotes the preimage of the normalizer of $U\leqslant \piaut{X_\Q}$ under the homomorphism $\Aut (\Lambda) \to \piaut{X_\Q}$.


\begin{prop} \label{prop:ubridge sullivan}
The dg Lie algebra \(\Der_{\lie u} \Lambda\) is nilpotent and there is a weak equivalence of topologial monoids
\begin{equation} \label{eq:sullivan model map}
    |\exp_\bullet\left(\Der_{\lie u} \Lambda \right)|\longrightarrow \aut_U(X_\Q)
\end{equation}
that is equivariant with respect to the conjugation action of \(\AutNorm{U}{\Lambda}\) on the domain and codomain.
\end{prop}

\begin{proof}
Nilpotence of $\Der \Lambda$ is proved as in Proposition \ref{prop:nilpotence of the dgl}. The map \eqref{eq:sullivan model map} is defined as in \cite[Proposition 3.7]{be20}: the action of the nilpotent dg Lie algebra $\lie g = \Der_{\lie u} \Lambda$ on $\Lambda$ by derivations induces an action of the simplicial nilpotent group
$$\exp_\bullet(\lie g) = \exp( Z_0(\Omega_\bullet \tensor \lie g))$$
on $\Omega_\bullet \tensor \Lambda$ by $\Omega_\bullet$-linear automorphisms, which in turn induces an action of $\exp_\bullet(\lie g)$ on the simplicial set
$$\Hom_{cdga(\Omega_\bullet)}(\Omega_\bullet \tensor \Lambda, \Omega_\bullet),$$
the geometric realization of which is the spatial realization of $\Lambda$. The rest of the proof is entirely analogous to Proposition~\ref{prop:ubridge}. We omit the details.
\end{proof}

\begin{scor}
The space $\Baut{G}{X}$ has dg Lie model $\Der_{\lie u}(\Lambda)$.
\end{scor}

A dg Lie algebra $\lie g(X)$ as in Theorem \ref{thm:main zig-zag} can be described in terms of Sullivan models as well. Since $\Aut(\Lambda)\to \piaut{X_\Q}$ has unipotent kernel, the group $R(X)$ may be identified with the maximal reductive quotient of $\Aut(\Lambda)$ and the quotient map $\Aut(\Lambda) \to R(X)$ admits a splitting by Theorem \ref{thm:Levi}. 

\begin{thm} \label{thm:sullivan}
There is a zig-zag of $\Gamma(X)$-equivariant rational equivalences that connects $\Bautfunc{u}{X}$ to the nerve of the dg Lie algebra $\Der_{\lie u}(\Lambda)$, on which $\Gamma(X)$ acts through any choice of splitting of $\Aut(\Lambda) \to R(X)$. Here, $\lie u \leqslant H_0(\Der \Lambda)$ is the Lie algebra of the unipotent radical of $\piaut{X_\Q}$.
\end{thm}


The following is an analog of Proposition \ref{lemma: lie of uniradical} for Sullivan models. It gives a more concrete description of the dg Lie algebra in Theorem \ref{thm:sullivan}.
\begin{prop}\label{lemma: lie of uniradical sullivan}
If $\lie u$ is the Lie algebra of the unipotent radical of $\piaut{X_\Q}$, then the dg Lie algebra $\Der_{\lie u} \Lambda$ agrees with the nilradical $\nil \Der \Lambda$. The graded components of this dg Lie algebra are given by
$$(\nil \Der \Lambda)_n =
\begin{cases}
(\Der \Lambda)_n, & n>0, \\
\Lie(\AUT_u \Lambda), & n=0, \\
0, & n<0.
\end{cases}
$$
Here, $\Lie(\AUT_u \Lambda)$ stands for the Lie algebra of the unipotent radical of $\AUT \Lambda$ and this Lie algebra may be identified with $\nil_{H} Z_0(\Der \Lambda)$, the maximal ideal of $Z_0(\Der \Lambda)$ of derivations that act nilpotently on $H = H^*(\Lambda)$.
\end{prop}



\begin{proof}
The proof is entirely analogous to the proof of Proposition \ref{lemma: lie of uniradical} so we omit most of it. The only thing that requires a different argument is the inclusion 
$\nil \Der \Lambda \subseteq \Der_{\lie u} \Lambda$.
For this, we again only need to consider the degree zero component. We have that $(\nil \Der \Lambda)_0$ is an ideal in $Z_0(\Der \Lambda)$ that acts nilpotently on each graded component $(\Der \Lambda)_k$ for $k>0$.
On the other hand, $(\Der_{\lie u} \Lambda)_0$ may be identified with the Lie algebra of the unipotent radical of $\AUT \Lambda$, so the desired inclusion will follow from Lemma~\ref{lemma: uniradical} as soon as we prove that $(\Der \Lambda)_{\leq n}$ is a faithful representation of $\AUT \Lambda$ for $n$ sufficiently large.

We are assuming that $\Lambda$ is finitely generated as an algebra. Fix algebra generators $x_1,\ldots,x_k$ and pick $n$ so that $n\geq |x_i|$ for all $i$. Suppose that $\phi\in \Aut \Lambda$ acts trivially on $(\Der \Lambda)_{\leq n}$. In particular, we then have an equality in $\Der \Lambda$ for every $i$,
$$\phi^{-1} \circ \frac{\partial}{\partial x_i} \circ \phi = \frac{\partial}{\partial x_i}.$$
Evaluating on $x_j$, we get equalities in $\Lambda$ for all $i,j$,
$$\phi^{-1}\left(\frac{\partial \phi(x_j)}{\partial x_i}\right)= \delta_{ij},$$
Applying the algebra automorphism $\phi$, we obtain
$$\frac{\partial \phi(x_j)}{\partial x_i} = \delta_{ij}.$$
This can only happen if $\phi(x_i) = x_i$ for all $i$, which means that $\phi =1$.

\end{proof}

\begin{scor} \label{cor:Aut Lambda reductive}
If $\Aut \Lambda$ is reductive, then
\begin{align*}
R(X) & = \Aut \Lambda, \\
\nil \Der \Lambda & = \Der \Lambda \langle 1 \rangle,
\end{align*}
and the action of $R(X)$ on $\nil \Der \Lambda$ is the conjugation action.
\end{scor}

\section{Case studies} \label{sec:case studies}
In this section we offer a few case studies. In addition to showcasing how the main results can be used in practice, they illustrate certain general points:

\begin{enumerate}[(i)]
\item Determining $\Gamma(X)$ typically entails some non-trivial integral homotopy theory, but is often easier than determining $\piaut{X}$. In many cases, $\Gamma(X)$ is the group of automorphisms of $H_*(X;\Z)$ that are induced by self-homotopy equivalences. In the literature, this group often appears as a stepping stone for computing $\piaut{X}$ or other groups of automorphisms of $X$.

\item The Lie algebra cohomology $H_{\CE}^*(\lie g(X))$ is sometimes explicitly computable, sometimes not. For elliptic spaces, such as products of spheres, one can often compute it explicitly. On the other hand, a complete computation for the manifolds $W_{g,1}$ would entail the computation of the homology of Kontsevich's Lie graph complex, which is a hard problem.

\item Even in cases where $H_{\CE}^*(\lie g(X))$ is explicitly computable, a complete calculation of the cohomology $H^*(\Gamma(X),H_{\CE}^*(\lie g(X)))$ is in general out of reach, due to the difficulty of computing cohomology of arithmetic groups. However, in some cases the cohomology, or parts of it, can be understood via automorphic forms. A paradigmatic example is the Eichler--Shimura isomorphism.

\item By contrast, all that is left modulo nilpotent elements is the invariant ring
\[H^0(\Gamma(X),H_{\CE}^*(\lie g(X))) = H_{\CE}^*(\lie g(X))^{\Gamma(X)}\]
and this is more tractable. By employing structural results for affine algebraic groups over $\Q$ and density results \cite{borel66}, this can often be reduced to classical invariant theory for finite or reductive groups, which is well understood.

\item For a graded representation $H$ of a reductive group $G$ and an arithmetic subgroup $\Gamma$ of $G(\Q)$, the split exact sequence
$$0 \to H^G \to H \to H/H^G \to 0$$
gives rise to a split exact sequence
$$0\to H^*(\Gamma,\Q) \tensor H^G \to H^*(\Gamma,H) \to H^*(\Gamma,H/H^G) \to 0.$$
In many cases of interest, stability and vanishing results for the cohomology of arithmetic groups with coefficients in non-trivial algebraic representations, as in Borel's work \cite{borel74,borel81}, show that the cokernel vanishes in a range of degrees.
Thus, in this `stable range' the cohomology of $\Baut{}{X}$ is isomorphic to
$$H^*(\Gamma(X),\Q) \tensor H_{\CE}^*(\lie g(X))^{R(X)}.$$
However, how non-trivial this `stable range' is depends on the group $\Gamma(X)$ and the representations $H_{\CE}^*(\lie g(X))$.
\end{enumerate}

\subsection{Products of spheres} \label{sec:products of spheres}
Consider the $n$-fold product of a $d$-dimensional sphere,
\[\Xdimpower{d}{n} = S^d \times \ldots \times S^d.\]
We begin by describing the reductive group $R(\Xdimpower{d}{n})$. The minimal Quillen model can be described explicitly, see \cite[V.2.(3)]{tan83}, but the minimal Sullivan model is finitely generated and even easier to describe in this case, so we will work with the latter.
\begin{prop} \label{prop:auto sphere}
The automorphism group of the minimal Sullivan model $\Lambda$ for $\Xdimpower{d}{n}$ is given by
\[
\Aut \Lambda \cong
\begin{cases}
\GL_n(\Q), & \textrm{$d$ odd}, \\
\Sigma_n \ltimes (\Q^\times)^n, & \textrm{$d$ even}.
\end{cases}
\]
In particular, $\Aut \Lambda$ is reductive, whence $R(\Xdimpower{d}{n}) \cong \Aut \Lambda$.
\end{prop}

\begin{proof}
For $d$ odd, the minimal model $\Lambda$ is an exterior algebra $\Lambda(x_1,\ldots,x_n)$ on generators of degree $d$ with zero differential. Clearly, $\Aut \Lambda \cong \GL_n(\Q)$. For $d$ even, the minimal model has the form
$$\Lambda = \big(\Lambda(x_1,\ldots,x_n,y_1,\ldots,y_n),d \big),$$
with $|x_i|=d$, $|y_i|=2d-1$ and $dx_i = 0$, $dy_i = x_i^2$. Let $\varphi$ be an automorphism of \(\Lambda\). Then
$$\varphi(x_i) = \sum_{j} a_{ij} x_j$$
for some $A=(a_{ij}) \in \GL_n(\Q)$. In cohomology, the equality
$$0 = \varphi(x_i^2) = \varphi(x_i)^2 = \sum_{j<k} 2a_{ij} a_{ik}x_jx_k,$$
implies $a_{ij}a_{ik} = 0$ for all $j\ne k$, whence exactly one entry in each row $(a_{i1},\ldots,a_{in})$ must be non-zero. Therefore, $\varphi(x_i) = \lambda_i x_{\sigma(i)}$ for some permutation $\sigma$ and some $\lambda_i\in\Q^\times$. Since $d\varphi(y_i) = \varphi(dy_i) = \lambda_i^2 x_{\sigma(i)}^2$, the only possibility is $\varphi(y_i) = \lambda_i^2 y_{\sigma(i)}$.
This shows that every automorphism $\varphi$ of $\Lambda$ is of the form
$$\varphi(x_i) = \lambda_i x_{\sigma(i)}, \quad \varphi(y_i) = \lambda_i^2 y_{\sigma(i)},$$
for some $\sigma\in \Sigma_n$ and $\lambda_i\in \Q^\times$. One checks that this yields an isomorphism $\Aut \Lambda \cong \Sigma_n\ltimes (\Q^\times)^n$. The description of $R(\Xdimpower{d}{n})$ follows from Corollary \ref{cor:Aut Lambda reductive}.
\end{proof}

\begin{rmk}
The group $\Sigma_n \ltimes (\Q^\times)^n$ may be identified with the group of `monomial matrices', i.e., invertible matrices with exactly one non-zero entry in each row. This is an example of a disconnected reductive group. The identity component is the torus $(\Q^\times)^n$ and the group of components is $\Sigma_n$.
\end{rmk}

\begin{rmk}
The preceding result, as well as the remainder of this section, goes through even for \(d=1\).
The cautious reader will object on the grounds that the space \(\Xdimpower{1}{n} = \left( S^1 \right)^n\) is not simply connected. It is, however, nilpotent (indeed a topological group), and therefore
amenable to analysis by our methods.
\end{rmk}

We now turn to the determination of the group $\Gamma(\Xdimpower{d}{n})$. For this, we first need to work out some elementary homotopy theory of maps between products of spheres.
\begin{defn}
Let us call an integer vector $(a_1,\ldots,a_n)\in \Z^n$ \emph{realizable} if there is a map
\[S^d \times \ldots \times S^d \to S^d\]
such that the restriction to the $i$th factor is a degree $a_i$ self-map of $S^d$.
This is equivalent to asking the $n$-fold higher order Whitehead product
\[ \big[ a_1 \iota_d, \ldots, a_n \iota_d \big] \subseteq \pi_{nd-1}(S^d)\]
to be defined and contain $0$, where $\iota_d\in \pi_d(S^d)$ is the class of the identity map.
\end{defn}

We would be surprised if the following has not been observed before, but we have not found a reference (except for the simplest case $n=2$, which is discussed in \emph{e.g.}~\cite[Example 5.1]{bb58}), so we supply a proof.
\begin{prop} \label{prop:realizability}\leavevmode
\begin{enumerate}[(i)]
\item For $d=1,3,7$, every integer vector $(a_1,\ldots,a_n)$ is realizable.

\item For $d$ odd $\ne 1,3,7$, an integer vector $(a_1,\ldots,a_n)$ is realizable if and only if at most one $a_i$ is odd.

\item For $d$ even, an integer vector $(a_1,\ldots,a_n)$ is realizable if and only if at most one $a_i$ is non-zero.
\end{enumerate}
\end{prop}

\begin{lemma} \label{lemma:realizable}
If $(a_1,\ldots,a_n)$ is realizable, then so are the vectors
\begin{gather*}
(a_{\sigma_1},\ldots,a_{\sigma_k}), \quad
(\lambda_1 a_1,\ldots, \lambda_n a_n), \quad
(a_1,\ldots,a_n,0),
\end{gather*}
for all injective maps $\sigma\colon \{1,\ldots,k\} \to \{1,\ldots,n\}$ and all $\lambda_1,\ldots,\lambda_n \in \Z$.
\end{lemma}

\begin{proof}
Precompose the given realizable map $S^d \times \ldots \times S^d \to S^d$ with the map that includes the $k$-fold product of $S^d$ according to $\sigma$ and inserts the basepoint in the other factors, or with the map $\lambda_1 \times \ldots \times \lambda_n$, or with the projection onto the first $n$ factors, respectively.
\end{proof}

\begin{lemma} \label{lemma:realizable2}
If $(1,a_2,\ldots,a_n)$ and $(b_1,\ldots,b_n)$  are realizable, then so is
\[(b_1,b_2 + a_2,\ldots,b_n + a_n).\]
\end{lemma}

\begin{proof}
By hypothesis, the matrices on the left-hand side of the equation
{\small $$
\begin{pmatrix}
1 & a_2 & \cdots & a_n \\
0 & 1 & \cdots & 0 \\
\cdots & \cdots & \cdots & \cdots \\
0 & 0 & \cdots &  1
\end{pmatrix}
\begin{pmatrix}
b_1 & b_2 & \cdots & b_n \\
0 & 1 & \cdots & 0 \\
\cdots & \cdots & \cdots & \cdots \\
0 & 0 & \cdots & 1
\end{pmatrix}
=
\begin{pmatrix}
b_1 & b_2 + a_2 & \cdots &  b_n+a_n \\
0 & 1 & \cdots &  0 \\
\cdots & \cdots & \cdots & \cdots \\
0 & 0 & \cdots &  1
\end{pmatrix}
$$}
can be realized as self-maps of $\Xdimpower{d}{n}$.
It follows that the same is true of the matrix in the right-hand side. In particular its first row is realizable.
\end{proof}

\begin{lemma} \label{lemma:realizable3}
If $(1,a)$ is realizable, then so is $(1,a,\ldots,a) \in \Z^n$ for every $n\geq 2$.
\end{lemma}

\begin{proof}
Assume by induction that $(1,a,\ldots,a)\in \Z^{n-1}$ is realizable. Then both $(1,a\ldots,a,0)\in \Z^n$ and $(1,0,\ldots,0,a)\in \Z^n$ are realizable by Lemma \ref{lemma:realizable}, and hence $(1,a,\ldots,a)\in \Z^n$ is realizable by Lemma \ref{lemma:realizable2}.
\end{proof}

\begin{proof}[Proof of Proposition \ref{prop:realizability}]
For $d=1,3,7$, the fact that $S^d$ is an $H$-space means precisely that $(1,1)$ is realizable. It follows from Lemma \ref{lemma:realizable3} that $(1,\ldots,1)\in \Z^n$ is realizable and then from Lemma \ref{lemma:realizable} that $(a_1,\ldots,a_n)$ is realizable for all $a_1,\ldots, a_n\in \Z$.

For $d$ odd $\ne 1,3,7$, it is well known that the Whitehead product
\begin{equation} \label{eq:Whitehead product}
[\iota_d,\iota_d] \in \pi_{2d-1}(S^d)
\end{equation}
is a non-zero class of order $2$ (this can be seen, \emph{e.g.}, by inspecting the EHP sequence). As noted above, $(a_1,a_2)$ is realizable precisely when the Whitehead product $[a_1\iota_d,a_2\iota_d]$ is trivial. Since the binary Whitehead product is bilinear, this happens if and only if $a_1a_2$ is even. By the first part of Lemma \ref{lemma:realizable}, this implies that $(a_1,\ldots,a_n)$ is realizable only if $a_ia_j$ is even for all $i\ne j$, which implies that at most one $a_i$ is odd.
Conversely, we have that $(1,2)$ is realizable since $[\iota_d,2\iota_d] = 0$. Hence so is $(1,2,\ldots,2)$ by Lemma \ref{lemma:realizable3}. Now one can use Lemma \ref{lemma:realizable} to deduce that $(a_1,\ldots,a_n)$ is realizable if at most one of the entries is odd.

For $d$ even, the Whitehead product \eqref{eq:Whitehead product} is a non-zero element of infinite order. As above, this implies that $(a_1,\ldots,a_n)$ is realizable only if at most one $a_i$ is non-zero. Conversely, Lemma \ref{lemma:realizable} shows that $(a_1,\ldots,a_n)$ is realizable if at most one entry is non-zero.
\end{proof}

Now we are ready to compute $\Gamma(\Xdimpower{d}{n})$.

\begin{prop} \label{prop:gamma}
We have
\[
\Gamma(\Xdimpower{d}{n}) \cong
\begin{cases}
\GL_n(\Z), & d=1,3,7, \\
\GL_n^\Sigma(\Z), & \textrm{$d$ odd} \ne 1,3,7, \\
\Sigma_n^{\pm}, & \textrm{$d$ even},
\end{cases}
\]
where \(\GL_n^\Sigma(\Z)\) denotes the group of invertible $n\times n$ integer matrices with exactly one odd entry in each row, and $\Sigma_n^{\pm}$ denotes the group of $n\times n$ signed permutation matrices.
\end{prop}

\begin{proof}
By Corollary \ref{cor:simplification}, we may identify $\Gamma(\Xdimpower{d}{n})$ with the group of automorphisms of $H_*(\Xdimpower{d}{n};\Z)$ that are realizable by a self-homotopy equivalence.

Given $A\in \GL_n(\Z)$, it is clear how to write down a map
$$S^d \vee \cdots \vee S^d \to S^d \times \cdots \times S^d$$
that realizes $A$ on $H_d(-;\Z)$. This extends to the product if and only if the projection to each factor $S^d$ does, which is precisely the condition that each row in $A$ is realizable. When an extension exists, it follows from the Whitehead theorem that it is a homotopy equivalence. Thus, $\Gamma(\Xdimpower{d}{n})$ may be identified with the group of invertible $n\times n$ integer matrices in which each row is realizable. To finish the proof, invoke Proposition \ref{prop:realizability}. (In the case $d$ odd $\ne 1,3,7$, note that invertibility of the matrix implies that at least one entry in each row of must be odd. Similarly, in the case $d$ even, invertibility of the matrix implies that there is a unique non-zero entry in each row and that this must be a unit.)
\end{proof}

\begin{rmk}
The group $\GL_n^\Sigma(\Z)$ may be identified with the semidirect product,
$$\GL_n^\Sigma(\Z) \cong \Sigma_n \ltimes \cGL{n}{2},$$
where $\cGL{n}{2} \leqslant \GL_n(\Z)$ denotes the principal level $2$ congruence subgroup, i.e., the kernel of the homomorphism $\GL_n(\Z) \to \GL_n(\Z/2\Z)$
that reduces the entries mod $2$, and where the symmetric group $\Sigma_n$ acts by simultaneous permutation of the rows and columns.

The signed permutation group $\Sigma_n^{\pm}$, also known as the hyperoctahedral group, admits a similar decomposition,
$$\Sigma_n^{\pm} \cong \Sigma_n \ltimes D_n(\Z),$$
where $D_n(\Z) \cong (\Z^\times)^n$ is the group of diagonal matrices in $\GL_n(\Z)$.
\end{rmk}

\begin{rmk}
The group of homology isomorphisms that are realizable by a self-homotopy equivalence of $\Xdimpower{d}{n}$ has also been determined by Basu--Farrell \cite[\S2]{basufarrell16}. The proof given here is simpler because we do not need to argue using generators for the groups involved.

Work of Lucas--Saeki \cite{lucassaeki02} shows that $\Gamma(\Xdimpower{d}{n})$ also agrees with the group of homology isomorphisms of $\Xdimpower{d}{n}$ that are realizable by a diffeomorphism.

The group $\Gamma(\Xdimpower{d}{2})$ agrees with the group $G_d$ that is used as a stepping stone in Baues' computation of the group of self-homotopy equivalences of $S^d \times S^d$ \cite[\S6]{baues96}.
\end{rmk}

\begin{rmk} \label{rmk:+}
If we let $\Gamma^+(\Xdimpower{d}{n})$ denote the image in $R(\Xdimpower{d}{n})$ of the orientation preserving self-homotopy equivalences, then it is easily seen that $\Gamma^+(\Xdimpower{d}{n}) = \Gamma(\Xdimpower{d}{n})\cap \SL_n(\Z)$ for $d$ odd, and that $\Gamma^+(\Xdimpower{d}{n})\leqslant \Sigma_n \ltimes (\Z^\times)^n$ is the subgroup of all $(\sigma,\lambda)$ such that $\lambda_1 \ldots \lambda_n = 1$ for $d$ even.
\end{rmk}

Next, we determine an algebraic Lie model for $\Baut{u}{\Xdimpower{d}{n}}$. 

\begin{prop} \label{prop:lie model}
Let \(d\) be odd.
The space $\Baut{u}{\Xdimpower{d}{n}}$
admits an algebraic Lie model of the form
\[V_n^*[-d],\]
where the right-hand side is the dual of the standard representation $V_n =\Q^n$ concentrated in homological degree $d$. The differential and the Lie bracket are trivial.
\end{prop}

\begin{proof}
When \(d\) is odd, the Sullivan minimal model of \(\Xdimpower{d}{n}\) is the exterior algebra \(\Lambda = \Lambda(x_1, \ldots, x_n)\) with zero differential and with $x_i$ in cohomological degree $d$. By Proposition \ref{prop:auto sphere}, the group $\Aut \Lambda \cong \GL_n(\Q)$ is reductive, so by Corollary \ref{cor:Aut Lambda reductive}, an algebraic Lie model is given by $\Der \Lambda \langle 1 \rangle$, with the conjugation action of $R(\Xdimpower{d}{n}) = \Aut \Lambda$.
Recall that we are using the convention that cohomological degrees are regarded as negative homological degrees. The only derivations of $\Lambda$ of positive homological degree are
\begin{equation} \label{eq:basis}
\frac{\partial}{\partial x_1}, \ldots, \frac{\partial}{\partial x_n}.
\end{equation}
Thus, $\Der \Lambda \langle 1 \rangle$ is an abelian dg Lie algebra with basis \eqref{eq:basis} in degree $d$ and zero differential. As a representation of $\GL_n$, it is dual to the standard representation.
\end{proof}

\begin{rmk}
The Chevalley--Eilenberg cochain algebra of the abelian Lie algebra $V_n^*[-d]$ may be identified with $\Sym(V_n[d+1])$, the polynomial algebra on the standard $\GL_n(\Q)$-representation $V_n = \Q^n$ concentrated in cohomological degree $d+1$ with trivial differential. Hence, there are isomorphisms of graded algebras of \(\Gamma(\Xdimpower{d}{n})\)-modules,
$$H^*(\Baut{u}{\Xdimpower{d}{n}};\Q) \cong H_{\CE}^*(\mathfrak{g}(\Xdimpower{d}{n})) \cong \Sym(V_n[d+1]).$$
This can be given a more geometric interpretation. The above implies that the evident map
$$\Baut{u}{S^d} \times \ldots \times \Baut{u}{S^d} \to \Baut{u}{S^d\times \ldots \times S^d}$$
is a rational equivalence. This may be interpreted as a splitting principle of sorts: every normal unipotent $\Xdimpower{d}{n}$-fibration is rationally equivalent to the `Whitney sum' of $n$ normal unipotent $S^d$-fibrations. Since $H^*(\Baut{u}{S^d};\Q)$ is a polynomial ring in the Euler class, we can say that the ring of rational characteristic classes of normal unipotent $\Xdimpower{d}{n}$-fibration is a polynomial ring in the Euler classes $e_1,\ldots,e_n$ of the associated $S^d$-fibrations.
\end{rmk}

By combining Proposition \ref{prop:lie model}, Proposition \ref{prop:gamma} and Corollary \ref{cor:cohomology ring}, we obtain

\begin{sthm} \label{thm:X_n}
For $d$ odd, there is an isomorphism of graded algebras
\begin{equation} \label{eq:x_n cohomology}
H^*(B\aut^+(\Xdimpower{d}{n});\Q) \cong
H^*\Big(\Gamma^+(\Xdimpower{d}{n}); \Sym^\bullet(V_n[d+1])\Big),
\end{equation}
where $V_n = \Q^n$ is the standard representation of $\GL_n(\Q)$, where
\[\Gamma^+(\Xdimpower{d}{n}) = \begin{cases}
\SL_n(\Z) & d=1,3,7, \\
\SL_n^\Sigma(\Z) & d \ne 1, 3, 7,
\end{cases}\]
and where \(\SL_n^\Sigma(\Z)\leqslant \SL_n(\Z)\) is the subgroup of matrices with exactly one odd entry in each row.
\end{sthm}

For $n=2$ the right-hand side of \eqref{eq:x_n cohomology} can be computed in terms of modular forms via the Eichler--Shimura isomorphism, as we now will discuss.

Let $V = \C^2$ denote the standard representation of $\GL_2(\C)$ and let $\widetilde{\Gamma}$ be a congruence subgroup of $\GL_2(\Z)$ that contains $-I$ and that strictly contains $\Gamma = \widetilde{\Gamma} \cap \SL_2(\Z)$, so that we have an exact sequence
$$1\to \Gamma \to \widetilde{\Gamma} \xrightarrow{\det} \Z^\times \to 1.$$
This gives rise to an action of $\Z^\times$ on $H^*(\Gamma,\Sym^{\bullet}(V))$, \emph{i.e.}~an involution, such that
\[ H^*(\widetilde{\Gamma},\Sym^{\bullet}(V)) \cong H^*(\Gamma,\Sym^{\bullet}(V))_+,\]
where $+$ indicates the $+1$ eigenspace of the involution.

\begin{lemma}
There is an isomorphism of graded vector spaces with involution,
\begin{equation} \label{eq:SL2}
H^*(\Gamma,\Sym^\bullet(V)) \cong \C[0]^+ \oplus M_{\bullet+2}(\Gamma)[1]^- \oplus S_{\bullet+2}(\Gamma)[1]^+,
\end{equation}
where $M_k(\Gamma)$ and $S_k(\Gamma)$ denote the spaces of modular forms and cusp forms of weight $k$ for $\Gamma$, and where $\C[0]$ is $\C$ concentrated in $\bullet=0$ and $*=0$. A superscript $\pm$ indicates how the involution acts. In particular, extraction of the $+1$ eigenspace yields
\begin{equation} \label{eq:GL2}
H^*(\widetilde{\Gamma},\Sym^\bullet(V)) \cong \C[0] \oplus S_{\bullet+2}(\Gamma)[1].
\end{equation}

\end{lemma}

\begin{proof}
The Eichler--Shimura isomorphism gives an isomorphism
\begin{equation} \label{eq:ES}
ES\colon M_k(\Gamma) \oplus \overline{S_k(\Gamma)} \xrightarrow{\cong} H^1(\Gamma,\Sym^{k-2}(V)),
\end{equation}
where $M_k(\Gamma)$ is the space of modular forms of weight $k$ for $\Gamma$ and $\overline{S_k(\Gamma)}$ is the space of antiholomorphic cusp forms of weight $k$, see \emph{e.g.}~\cite[\S6]{wiese19}. Identifying the involution on the left-hand side of \eqref{eq:ES} requires a little care.

The space of modular forms decomposes as $M_k(\Gamma) = E_k(\Gamma)\oplus S_k(\Gamma)$, where $E_k(\Gamma)$ denotes the Eisenstein space. Letting $r$ denote the automorphism of the upper half plane given by $r(z) = -\overline{z}$, one can check that the composite
\[
E_k(\Gamma) \oplus S_k(\Gamma) \oplus S_k(\Gamma) \xrightarrow{\varphi} M_k(\Gamma) \oplus \overline{S_k(\Gamma)} \xrightarrow{ES} H^1(\Gamma,\Sym^{k-2}(V)),
\]
where
\[\varphi(e,f,g) = (e + f + g, f r - g r)\]
is an isomorphism of vector spaces with involution, where the involution on the left-hand side is given by $(e,f,g) \mapsto (-e,-f,g)$. This explains the isomorphism \eqref{eq:SL2} for $*=1$. Restriction to the $+1$ eigenspace yields an isomorphism
$$S_k(\Gamma) \to H^1(\Gamma,\Sym^{k-2}(V))_+$$
given by $g\mapsto ES(g,-gr)$.

For the other cohomological degrees, one checks that
\[H^0(\Gamma,\Sym^{k-2}(V)) = \Sym^{k-2}(V)^{\Gamma} \cong \C,\]
and $H^i(\Gamma,\Sym^{k-2}(V)) = 0$ for $i>1$, because the virtual cohomological dimension of any finite index subgroup of $\SL_2(\Z)$ is $1$.
\end{proof}

Assembly of the above considerations yields
\begin{sthm} \label{thm:n=2}
For $d$ odd, there is an isomorphism of graded vector spaces with involution
\[ \widetilde{H}^*(B\aut^+(S^d \times S^d);\Q) \cong \bigoplus_{k} \big(M_k(\Gamma)^- \oplus S_k(\Gamma)^+\big) [(k-2)(d+1)+1],\]
where \(M_k(\Gamma)\) and \(S_k(\Gamma) \) denote the spaces of modular forms and cusp forms of weight $k$ for $\Gamma$, and where $\Gamma = \SL_2(\Z)$ for $d=1,3,7$ and $\Gamma = \SL_2^\Sigma(\Z)$ for $d\ne 1,3,7$.

In particular, since the reduced cohomology is concentrated in odd degrees, it follows that $B\aut^+(S^d\times S^d)$ is formal with trivial cohomology ring.
\end{sthm}

Define the Poincar\'e series of a graded vector space $H^*$ with involution to be the formal power series in $z$ with coefficients in $\Z[\epsilon]/(\epsilon^2-1)$ given by
$$\sum_{k\geq 0}\Big( \dim(H_+^k) + \dim(H_-^k) \epsilon \Big)z^k.$$
There are well-known dimension formulas for spaces of modular forms, so as a corollary we get a formula for the Poincar\'e series of the cohomology.

\begin{cor}
The Poincar\'e series of the graded vector space with involution
\[H^*(B\aut^+(S^d\times S^d);\Q)\]
is given by the following formulas, where we set $\ell = d+1$.

For $d=1,3,7$:
\[
1+z^{2\ell +1} \frac{\epsilon \big(1+z^{4\ell} - z^{6\ell}\big) + z^{8\ell}}{(1-z^{4\ell})(1-z^{6\ell})}.
\]

For $d$ odd $\ne 1,3,7$:
\[
1+z \frac{\epsilon \big(1+z^{2\ell} - z^{4\ell}\big) + z^{6\ell}}{(1-z^{2\ell})(1-z^{4\ell})}.
\]
\end{cor}

\begin{proof}
As is well known, the ring of modular forms $M(\SL_2(\Z))$ is freely generated by the Eisenstein series $E_4$ and $E_6$, and the space of cusp forms $S(\SL_2(\Z))$ is the principal ideal generated by the discriminant $\Delta$, which is of weight $12$.
In particular, the Poincar\'e series are given by
\begin{align*}
\sum_{k\geq 0} \dim M_k(\SL_2(\Z)) t^k & = \frac{1}{(1-t^4)(1-t^6)},\\
\sum_{k\geq 0} \dim S_k(\SL_2(\Z)) t^k & = \frac{t^{12}}{(1-t^4)(1-t^6)}.
\end{align*}
The group $\SL_2^\Sigma(\Z)$ is an index 3 subgroup of $\SL_2(\Z)$, sometimes referred to as the `theta group'. It is generated by the two matrices
\[
\begin{pmatrix}
0 & -1 \\ 1 & 0
\end{pmatrix},
\quad
\begin{pmatrix}
1 & 2 \\ 0 & 1
\end{pmatrix}.
\]
The Poincar\'e series of the modular forms and the cusp forms for the theta group are
\begin{align*}
\sum_{k\geq 0} \dim M_k(\SL_2^\Sigma(\Z)) t^k & = \frac{1}{(1-t^2)(1-t^4)},\\
\sum_{k\geq 0} \dim S_k(\SL_2^\Sigma(\Z)) t^k & = \frac{t^{8}}{(1-t^2)(1-t^4)}.
\end{align*}
This can be seen from the dimension formulas in \cite[Proposition 1]{kohler88}, or alternatively by observing that the matrix
$\begin{psmallmatrix}
1 & 1 \\ 0 & 2
\end{psmallmatrix}$ conjugates $\SL_2^{\Sigma}(\Z)$ to the congruence subgroup $\Gamma_0(2) = \Gamma_1(2)$; dimension formulas for the latter can be found in \cite[p.108]{ds05}. Theorem \ref{thm:n=2} together with the above formulas yield the desired result after some manipulations with generating functions.
\end{proof}

\begin{rmk}
Computations of $H^*(\BDiff^+(T^2);\Q)$ have been carried out in \cite{morita87} and \cite{fty88} (the latter using the Eichler--Shimura isomorphism). The inclusion of $\Diff^+(T^2)$ into $\aut^+(T^2)$ is a weak homotopy equivalence, so we recover this computation by setting $d=1$ and $\epsilon =1$ in the above.
\end{rmk}

The cohomology of $\GL_3(\Z)$ and $\SL_3(\Z)$ with coefficients in irreducible algebraic representations has recently been computed in many cases \cite{bhhg20}. Let $V = \C^3$ be the standard representation of $\GL_3(\C)$. The combination of \cite[Corollary 18]{bhhg20} and \cite[Theorem 16]{bhhg20} specialized to $\mathcal{M}_{k,0} = \Sym^k(V)$ (which is not self-dual for $k>0$) shows
$$
H^q(\SL_3(\Z),\Sym^k(V)) \cong
\begin{cases}
S_{k+2}, & \text{for $q=3$ and $k>0$ even,} \\
S_{k+3}\oplus \C, & \text{for $q=2$ and $k$ odd,} \\
\C, & \text{for $q=0$ and $k=0$,} \\
0, & \text{otherwise,}
\end{cases}
$$
where $S_k$ denotes the space of cusp forms of weight $k$ for $\SL_2(\Z)$. If we take the liberty of writing $M_{k+3}$ for the isomorphic vector space $S_{k+3} \oplus \C$ in the second case above, we can summarize the calculation as an isomorphism of bigraded vector spaces with involution:
\begin{equation} \label{eq:SL3}
H^*(\SL_3(\Z),\Sym^\bullet(V)) \cong \C[0]^+ \oplus M_{\bullet +3}[2]^- \oplus S_{\bullet +2}[3]^+.
\end{equation}
The action of the involution on the right-hand side is implicit in \cite[Lemma 17]{bhhg20}; this lemma asserts that
\begin{equation} \label{eq:GL3}
H^*(\GL_3(\Z),\Sym^\bullet(V)) \cong \C[0] \oplus S_{\bullet + 2}[3].
\end{equation}

As before, this leads to

\begin{sthm}
For $d=1,3,7$, there is an isomorphism of graded vector spaces with involution
\[\widetilde{H}^*(\BSaut{}{\Xdimpower{d}{3}};\Q) \cong \bigoplus_k M_k^-[(k-3)(d+1)+2] \oplus S_k^+[(k-2)(d+1)+3)],\]
and the Poincar\'e series is given by
\[
    z^{\ell +2}\frac{\epsilon\big(1 + z^{2\ell} - z^{6\ell} \big) + z^{9\ell +1}}{(1-z^{4\ell})(1-z^{6\ell})}.\qedhere
\]
\end{sthm}

\begin{rmk} \label{rmk:formality}
It seems plausible that computations similar to those of \cite{bhhg20} can be carried out for the group $\SL_3^\Sigma(\Z)$, but this is beyond the scope of this paper. However, even if we currently lack complete calculations, certain qualitative conclusions can be drawn. By Theorem \ref{thm:cdga model}, the cdga $\Omega^*(\Baut{}{X})$ is quasi-isomorphic to  \[\Omega^*(\Gamma(X),C_{\CE}^*(\lie g(X))).\]
If $C_{\CE}^*(\lie g(X))$ is formal as a cdga in $\Rep_\Q(R(X))$, then the above is quasi-isomorphic to \[\Omega^*(\Gamma(X);H_{\CE}^*(\lie g(X)))\]
as a cdga. An application of the homotopy transfer theorem for $C_\infty$-algebras yields a $C_\infty$-algebra structure $\{m_n\}$ on $H^*(\Gamma(X);H_{\CE}^*(\lie g(X)))$ that is bigraded in the sense that $m_n$ has bidegree $(2-n,0)$, and that is $C_\infty$-quasi-isomorphic to $\Omega^*(\Baut{}{X})$.

If the cohomology of $\Gamma(X)$ with coefficients in algebraic representations is concentrated in a small range of degrees, this can force these $C_\infty$-operations to be trivial, implying that the cdga $\Omega^*(\Baut{}{X})$ is formal. For example, this happens if there is an $r$ such that $\widetilde{H}^i(\Gamma(X),V)=0$ unless $r\leq i\leq 3r-2$, for all $V\in \Rep_\Q(R(X))$.

For example, if $\Gamma \leqslant \SL_3(\Z)$ is a finite index subgroup, then $H^1(\Gamma,V) = 0$ for any finite-dimensional $\Q$-vector space $V$ with an action of $\Gamma$ by \cite{bms67}, and $H^i(\Gamma,V) = 0$ for $i>3$, because $\Gamma$ has virtual cohomological dimension $3$.
\end{rmk}

These considerations lead to

\begin{sthm}
For $d$ odd, the space $\BSaut{}{S^d \times S^d \times S^d}$ is formal and the rational cohomology has trivial ring structure.
\end{sthm}

Another remark is that the invariant ring $$H^*(\Baut{u}{\Xdimpower{d}{n}};\Q)^{\Gamma(\Xdimpower{d}{n})}$$
is the trivial ring $\Q$ in this case, because $\Sym^k(V_n)^{\Gamma(\Xdimpower{d}{n})}$ is easily seen to be trivial. In particular, $H^*(B\aut{}{(\Xdimpower{d}{n})};\Q)$ consists entirely of nilpotent elements.

For $n>3$, calculations become increasingly difficult, but one can at least say something about the limit as $n\to\infty$. Consider the maps
\begin{align*}
\sigma\colon \Baut{}{\Xdimpower{d}{n}}  & \to \Baut{}{\Xdimpower{d}{(n+1)}}, \\
\pi\colon \Baut{}{\Xdimpower{d}{n}} & \to B\Gamma(\Xdimpower{d}{n}),
\end{align*}
where $\sigma$ is induced by extending a self-homotopy equivalence of $\Xdimpower{d}{n}$ by the identity on the second factor of $\Xdimpower{d}{(n+1)} = \Xdimpower{d}{n} \times S^d$ and $\pi$ is the evident map. It is a consequence of Borel's results on the stable cohomology of arithmetic groups \cite{borel74,borel81} that these maps induce an isomorphism in $H^i(-;\Q)$ for $n \gg i$. The explicit ranges stated below rely on more recent results due to Kupers--Miller--Patzt \cite{kmp21} and Putman \cite{putman21}.

\begin{thm}
Let $d$ be odd. The map
\begin{align*}
\sigma^*\colon H^i(\Baut{}{\Xdimpower{d}{(n+1)}};\Q) & \to H^i(\Baut{}{\Xdimpower{d}{n}};\Q)
\end{align*}
is an isomorphism for $n\geq i+1$ if $d=1,3,7$ and for $n\geq 2i+6$ for all odd $d$ and the map
\begin{align*}
\pi^* \colon H^i(\Gamma(\Xdimpower{d}{n}),\Q) & \to H^i(\Baut{}{\Xdimpower{d}{n}};\Q).
\end{align*}
is an isomorphism for $n\geq i-d+1$ if $d=1,3,7$ and for $n\geq 2i-2d+6$ for all odd $d$.
\end{thm}

\begin{proof}
By Theorem \ref{thm:X_n}, the map $\pi^*$ is injective and its cokernel is isomorphic to
\begin{equation} \label{eq:cokernel}
\bigoplus_{\substack{j+k(d+1) = i \\ k>0}}
H^j(\Gamma(\Xdimpower{d}{n}),\Sym^k(V_n)).
\end{equation}
By Borel's vanishing theorem \cite[Theorem 4.4]{borel81}, the summands vanish for $n$ large. Explicitly, \cite[Theorem 7.6]{kmp21} implies vanishing for $n\geq j+k+1$ for $d=1,3,7$, and \cite[Theorem C]{putman21} implies vanishing for $n\geq 2j+2k+6$ for $d\ne 1,3,7$.
It follows that \eqref{eq:cokernel} vanishes for $n\geq i-d+1$ for $d=1,3,7$ and for $n\geq 2i-2d+6$ for $d\ne 1,3,7$.

The statement about $\sigma^*$ now follows from the corresponding statement for
\[H^i(\Gamma(\Xdimpower{d}{(n+1)}),\Q) \to H^i(\Gamma(\Xdimpower{d}{n}),\Q).\]
By Borel's work \cite{borel74} this is an isomorphism for $n$ large. According to \cite[Theorem A]{kmp21} it is an isomorphism for $n\geq i+1$ for $d=1,3,7$, and \cite[Theorem C]{putman21} implies it is an isomorphism for $n\geq 2i+6$ for $d\ne 1,3,7$.
\end{proof}

By Borel's calculation \cite{borel74}, the stable cohomology may be identified with an exterior algebra
\[\varprojlim_n H^*(\Baut{}{\Xdimpower{d}{n}};\Q) \cong \varprojlim_n H^*(\Gamma(\Xdimpower{d}{n}),\Q) \cong \Lambda[x_5,x_9,x_{13},\ldots].\]

We now turn to the case $d$ even, which behaves very differently.
\begin{prop}
For $d$ even, the dg Lie algebra $\lie g (\Xdimpower{d}{n})$ is formal with homology the abelian Lie algebra with basis
\[ \alpha_1,\ldots, \alpha_n, \quad \beta_{ij}, \quad 1\leq i \ne j \leq n,\]
in degrees $|\alpha_i| = 2d-1$ and $|\beta_{ij}| = d-1$.
The action of
\[(\sigma,\lambda)\in R(\Xdimpower{d}{n}) = \Sigma_n \ltimes (\Q^\times)^n\]
is given by
\[(\sigma,\lambda)\cdot \alpha_i = \lambda_i^{-2}\alpha_{\sigma(i)},\quad (\sigma,\lambda) \cdot \beta_{ij} =  \lambda_i \lambda_j^{-2} \beta_{\sigma(i)\sigma(j)}.\] 
\end{prop}

\begin{proof}
Let $\Lambda$ denote the minimal model. Since $\Aut \Lambda$ is reductive, the positive truncation $\Der \Lambda \langle 1 \rangle$ is a Lie model for $\Baut{u}{X}$ by Corollary \ref{cor:Aut Lambda reductive}. This is spanned by
derivations of the form
$$\frac{\partial}{\partial x_i}, \quad x_i \frac{\partial}{\partial y_j}, \quad \frac{\partial}{\partial y_i},$$
and the only non-trivial differential is given by
$$\left[d,\frac{\partial}{\partial x_i} \right] = -2x_i \frac{\partial}{\partial y_i}.$$
The action of the group $\Aut \Lambda$ is easily computed, \emph{e.g.}, if $\varphi$ corresponds to $(\sigma,\lambda)\in \Sigma_n\ltimes (\Q^\times)^n$, then $\varphi\big(x_i\frac{\partial}{\partial y_j}\big) = \lambda_i\lambda_j^{-2} x_{\sigma(i)}\frac{\partial}{\partial y_{\sigma(j)}}$.
A basis for the homology is represented by the cycles $\alpha_i = \frac{\partial}{\partial y_i}$ and $\beta_{ij} = x_i\frac{\partial}{\partial y_i}$ for $i\ne j$.
These span an abelian dg Lie subalgebra stable under the action of $\Aut \Lambda$, and the inclusion into $\Der \Lambda \langle 1 \rangle$ is a quasi-isomorphism.
\end{proof}

\begin{cor} \label{cor:cohomology d even}
For $d$ even, the space $\Baut{}{\Xdimpower{d}{n}}$ is formal and there is an isomorphism of graded algebras
\[H^*(\Baut{}{\Xdimpower{d}{n}};\Q) \cong \Q[a_i,b_{ij}]^{\Gamma(\Xdimpower{d}{n})},\]
where $|a_i| = 2d$, $|b_{ij}| = d$ and $(\sigma,\lambda)\in\Gamma(\Xdimpower{d}{n}) = \Sigma_n\ltimes (\Z^\times)^n$ acts by
\[(\sigma,\lambda)\cdot a_i = a_{\sigma(i)},\quad (\sigma,\lambda)\cdot b_{ij} = \lambda_i b_{\sigma(i)\sigma(j)}.\]
Similarly, the space $B\aut^+(\Xdimpower{d}{n})$ is formal and there is an isomorphism of algebras
\[H^*(B\aut^+(\Xdimpower{d}{n});\Q) \cong \Q[a_i,b_{ij}]^{\Gamma^+(\Xdimpower{d}{n})},\]
where $\Gamma^+(\Xdimpower{d}{n})\leqslant \Sigma_n \ltimes (\Z^\times)^n$ is the subgroup consisting of all $(\sigma,\lambda)$ such that $\lambda_1 \ldots \lambda_n = 1$.
\end{cor}

\begin{proof}
Clearly, $C_{\CE}^*(\lie g(\Xdimpower{d}{n})) = \Q[a_i,b_{ij}]$, where $a_i$ and $b_{ij}$ are the dual $1$-cochains of $\alpha_i$ and $\beta_{ij}$, and the differential is trivial since $\lie g(\Xdimpower{d}{n})$ is abelian with trivial differential. In particular, $C_{\CE}^*(\lie g(\Xdimpower{d}{n}))$ is formal, so Remark \ref{rmk:formality} shows that the $C_\infty$-algebra structure on $H^*(\Gamma(\Xdimpower{d}{n});H_{\CE}^*(\lie g(\Xdimpower{d}{n})))$ respects the bigrading. But since the group $\Gamma(\Xdimpower{d}{n})$ is finite, the cohomology reduces to the invariants concentrated in bidegree $(0,*)$, whence the $C_\infty$-operations $m_n$ vanish for $n>2$.
\end{proof}

The computation of invariant subrings $\Q[V]^{G}$ for finite-dimensional representations $V$ of finite or reductive groups $G$ is classical and well-understood in principle (see \emph{e.g.}~\cite{derksenkemper02}). The invariant subring is a Cohen--Macaulay ring, which means that one can find a regular sequence $\theta_{1},\ldots,\theta_m$ of invariant polynomials (the `primary invariants') such that $\Q[V]^{G}$ is a finitely generated free module over $\Q[\theta_1,\ldots,\theta_m]$ on certain invariant polynomials $\eta_1,\ldots,\eta_t$ (the `secondary invariants').

\begin{scor}
For $d$ even,
$H^*(B\aut^+(\Xdimpower{d}{n});\Q)$ is a Cohen--Macaulay ring of Krull dimension $n^2$, concentrated in degrees that are multiples of $d$.
\end{scor}

Let us examine the case $n=2$ closer to illustrate. The invariant theory calculation can be carried out in two steps using
\[\Q[a_1,a_2,b_{12},b_{21}]^{\Gamma(S^{d\times 2})} = \Big(\Q[a_1,a_2,b_{12},b_{21}]^{(\Z^\times)^2}\Big)^{\Sigma_2}.\]
The invariant ring with respect to the action of $(\Z^\times)^2$ is easily seen to be a polynomial ring in $a = a_1$, $b = b_{12}^2$, $c=b_{21}^2$, $d=a_2$, so we are left to identify the invariant ring
\[ \Q[a,b,c,d]^{\Sigma_2},\]
where the non-trivial element of $\Sigma_2$ acts as the permutation $(ad)(bc)$. By using Molien's theorem (specialized to permutation representations as in \cite[Proposition 4.3.4]{smith95}), the Poincar\'e series with respect to word-length in the generators can be computed to be
\[
\frac{1+t^2}{(1-t)^2(1-t^2)^2}.
\]
The invariant polynomials
$$a+d,\quad b+c, \quad ad, \quad bc,$$
form a regular sequence of length equal to the Krull dimension, so these form a set of primary invariants. A choice of secondary invariants is given by the two polynomials
\[1,\quad  ab + cd.\]
The ring structure is determined by writing $(ab + cd)^2$ in the basis; one has
\[
(ab + cd)^2 = \Big((a+d)(b+c)\Big)\cdot (ab+cd) - \Big((a+d)^2bc + ad(b+c)^2 - 4(ad)(bc)\Big) \cdot 1.
\]
The calculation for $\Gamma^+(X_2)$ is similar. We omit the details; the only difference is that there is one additional invariant $\alpha_0 = b_{12}b_{21}$, which squares to $bc = b_{12}^2 b_{21}^2$. Writing $\alpha_1$, $\alpha_2$, $\beta_1$, $\beta_2$, $\eta$ for $a+d$, $b+c$, $ad$, $bc$, $ab+cd$, respectively, these considerations lead to

\begin{sthm}
For $d$ even, there is an isomorphism of graded algebras with involution
\[H^*(B\aut^+(S^d \times S^d);\Q) \cong  \Q[\alpha_0,\alpha_1,\alpha_2,\beta_1,\beta_2,\eta]/I,\]
where the generators $\alpha_0,\alpha_1,\alpha_2$ are of degree $2d$ and $\beta_1,\beta_2,\eta$ are of degree $4d$,
where $I$ is the ideal generated by the two elements
\[\eta^2 -\alpha_1\alpha_2 \eta + \big(\alpha_1^2\beta_2 + \beta_1\alpha_2^2-4\beta_1\beta_2\big), \quad \alpha_0^2 - \beta_2,\]
and where the involution acts by $\alpha_0 \mapsto - \alpha_0$ and trivially on the other generators. The Poincar\'e series of this graded vector space with involution is given by
\[
(1+\epsilon z^{2d})\frac{1+z^{4d}}{(1-z^{2d})^2(1-z^{4d})^2}.
\]

A presentation for the cohomology ring of $\Baut{}{S^d \times S^d}$ is obtained by removing the generator  $\alpha_0$ and the relation $\alpha_0^2 - \beta_2$ from the above presentation. The Poincar\'e series is obtained by removing the term involving $\epsilon$.
\end{sthm}

The case $n>2$ can in principle be treated similarly but we stop here.

\subsection{Highly connected even-dimensional manifolds} \label{sec:highly connected manifolds}
The rational homotopy theory of the topological monoid of self-homotopy equivalences of highly connected even-dimensional manifolds has been thoroughly studied in \cite[\S5]{bm20}. We will now discuss how the methods of the present paper lead to simplifications of certain arguments in \cite{bm20}, as well as to some new results.

Let $M$ be an $(n-1)$-connected $2n$-dimensional manifold ($n>1$). We will consider the space $\aut^+(M)$ of orientation preserving self-homotopy equivalences. Let $\Gamma^+(M)$ denote the image of $\piauto{M} = \pi_0\aut^+(M)$ in $R(M)$, and let $\Aut^+ \LL$ denote the group of automorphisms of the minimal Quillen model $\LL$ of $M$ that represent orientation preserving self-homotopy equivalences. 

As discussed in \cite[\S3.5]{bm20}, the minimal Quillen model of $M$ can be presented as
\[\LL = \big( \LL(\alpha_1,\ldots,\alpha_r,\gamma), \delta \gamma = \omega\big),\]
where $\alpha_1,\ldots,\alpha_r$ is a basis for $H_n(M;\Q)[1-n]$, where $\gamma$ corresponds to the fundamental class of $M$, and where $\omega$ is dual to the cup product pairing
\[\langle \, ,\, \rangle \colon H^n(M;\Q)\tensor H^n(M;\Q) \to \Q, \quad \langle x,y\rangle = \langle x\smile y, [M] \rangle.\]
Explicitly,
\[\omega = \frac{1}{2} \sum_{i} \big[\alpha_i^\#, \alpha_i\big],\]
where $\alpha_i^\#$ is the dual basis with respect to the intersection form.

\begin{prop} \label{prop:hcr}
Let $M$ be an $(n-1)$-connected $2n$-dimensional manifold ($n>1$). The group $\Aut^+ \LL$ is isomorphic to the group of automorphisms of $H^n(M;\Q)$ that preserve the cup product pairing. In particular, $\Aut^+ \LL$ is reductive, whence $\Aut^+ \LL \cong \piauto{M_\Q} \cong R^+(M)$.
\end{prop}

\begin{proof}
An automorphism of $\LL$ represents an orientation preserving self-homotopy equivalence if and only if it fixes $\gamma$. Using this, one sees that $\Aut^+ \LL$ is isomorphic to the group of automorphisms of the space spanned by $\alpha_1,\ldots,\alpha_r$ that fix $\omega$. Since $\omega$ is dual to the cup product pairing, this is in turn isomorphic to $\Aut\big(H^n(M;\Q),\langle\,,\,\rangle\big)$. The group of automorphisms of a non-degenerate symmetric or skew-symmetric bilinear form is well-known to be reductive.
\end{proof}

Corollary \ref{cor:simplification} implies that the group $\Gamma^+(M)$ may be identified with the group of automorphisms of $H_*(M;\Z)$ that are induced by an orientation preserving homotopy equivalence of $M$. This group is known, \emph{cf.}~\cite[Theorem 8.14]{baues96}, \cite[Theorem 2.12]{bm13} or \cite[\S5.1]{bm20}. To describe it, recall the cohomology operation $\psi\colon H^n(M) \to \pi_{2n-1}(S^n)$ defined by Kervaire--Milnor \cite[\S8]{km63}; the class $\psi(x)$ is the obstruction for the existence of a map $f\colon M\to S^n$ such that $f^*(s) = x$, where $s$ is a generator for $H^n(S^n)$.

\begin{sprop}
Let $M$ be an $(n-1)$-connected $2n$-dimensional manifold ($n>1$).
The group $\Gamma^+(M)$ may be identified with the group of automorphisms of $H^n(M)$ that preserve the cup product pairing
$$\langle -, -\rangle \colon H^n(M) \tensor H^n(M) \to \Z$$
and the Kervaire--Milnor cohomology operation
\[\psi\colon H^n(M) \to \pi_{2n-1}(S^n).\qedhere\]
\end{sprop}

For example, for the manifold \(W_g = \#^g S^n \times S^n\), the group $\Gamma^+(W_g)$ coincides with the group denoted $\Gamma_g$ in \cite{bm20}. As explained in \cite[Example 5.5]{bm20} it may be described by
\[
\Gamma_g \cong \begin{cases} \orth_{g,g}(\Z), & \textrm{$n$ even,} \\ \Sp_{2g}(\Z), & \textrm{$n=1,3,7$,} \\ 
\Sp_{2g}^q(\Z), & \textrm{$n$ odd $\ne 1,3,7$,} \\ 
\end{cases}
\]
where $\Sp_{2g}^q(\Z)$ denotes the group of block matrices
\[ \begin{pmatrix} A & B \\ C & D \end{pmatrix} \in \Sp_{2g}(\Z),\]
such that the diagonal entries of the $g\times g$-matrices $C^tA$ and $D^tB$ are even.

The next result should be viewed as an upgrade of \cite[Theorem 5.8]{bm20}, which asserts that the Lie model for the space $B\aut_\circ(M)$ is formal. The novelty is that the formality can be made equivariant.

\begin{prop} \label{prop:coformal}
Let $M$ be an $(n-1)$-connected $2n$-dimensional manifold ($n>1$) with $\dim H^n(M;\Q)>2$. The algebraic Lie model $\lie g(M)$ for $B\aut_u^+(M)$
is formal as a dg Lie algebra of algebraic representations of $R^+(M)$ and its homology is the graded Lie algebra 
$$H_*(\lie g(M)) \cong \Der L / \ad L \langle 1\rangle,$$
where $L = \pi_*(\Omega M) \tensor \Q$ is the rational homotopy Lie algebra of $M$. This Lie algebra admits the presentation
$$L \cong \Lbb(\alpha_1,\ldots,\alpha_r)/(\omega),$$
with $\alpha_1,\ldots,\alpha_r$ and $\omega$ as above. The action of $R^+(M)$ is induced by the action on $H_*(M;\Q)$.
\end{prop}

\begin{proof}
Since $\Aut^+ \LL$ is reductive, we can show that $\Der^c \LL \langle 1 \rangle$ is an $R^+(M)$-algebraic Lie model for $B\aut_u^+(M)$ as in Corollary \ref{cor:Aut L reductive}.
Given this, it is straightforward to inspect that the zig-zag of quasi-isomorphisms that connects this dg Lie algebra with its homology given in the proof of Theorem 5.9 of \cite{bm20} can be made into a zig-zag of algebraic $R^+(M)$-representations.
\end{proof}

The action of $\Gamma^+(M)$ on the rational homotopy groups was also identified in \cite{bm20}, but the methods of \cite{bm20} were insufficient for proving the following result, which is a direct consequence of Corollary \ref{cor:cohomology ring} and Proposition \ref{prop:coformal}.

\begin{sthm} \label{thm:hcm}
Let $M$ be an $(n-1)$-connected $2n$-dimensional manifold ($n>1$) with $\dim H^n(M;\Q) > 2$. There is an isomorphism of graded algebras
$$H^*(B\aut^+(M);\Q) \cong H^*(\Gamma^+(M),H_{\CE}^*(\lie g)),$$
where $\Gamma^+(M)$ is the group $\Aut(H^n(M;\Z),\langle\,,\,\rangle,\psi)$, 
$\lie g$ is $\Der L /\ad L\langle 1 \rangle$,
and $L$ is the graded Lie algebra $\pi_*(\Omega M)\tensor \Q$.
\end{sthm}

When adapted to the space $\aut_\partial(W_{g,1})$ of self-homotopy equivalences of
\[W_{g,1} = W_g \setminus \operatorname{int} D^{2n}\]
relative to the boundary, our methods yield a significant simplification of the computation of the stable cohomology of $B\aut_\partial(W_{g,1})$ of \cite{bm20}. By applying our methods to the homotopy fiber sequence
\begin{equation} \label{eq:fs}
\Gamma_g \hOrbits \aut_\partial(W_{g,1}) \to B\aut_\partial(W_{g,1}) \to B\Gamma_g,
\end{equation}
one can upgrade \cite[Theorem 3.12]{bm20} to show that the $\Gamma_g$-space
\(\Gamma_g \hOrbits \aut_\partial(W_{g,1})\)
has $R_g$-algebraic Lie model
\[\lie g_g = \Der_\omega \LL \langle 1 \rangle,\]
where $\Der_\omega \LL$ denotes the graded Lie algebra of derivations of of the free graded Lie algebra $\LL = \LL(\alpha_1,\beta_1,\ldots,\alpha_g,\beta_g)$ that annihilate the element
\[\omega = [\alpha_1,\beta_1] + \ldots + [\alpha_g,\beta_g].\]
As above, this implies
\begin{sthm} \label{thm:unstable ring iso}
There is an isomorphism of graded algebras
\begin{equation} \label{eq:unstable iso}
H^*(\Baut{\partial}{W_{g,1}};\Q) \cong H^*(\Gamma_g,H_{\CE}^*(\lie g_g)).\qedhere
\end{equation}
\end{sthm}

\begin{rmk}
In the stable range, i.e., for $g$ large compared to the cohomological degree, the right-hand side can be simplified. Let $H = H_{\CE}^*(\lie g_g)$. 
By semisimplicity of $\Rep_\Q(R_g)$ there is a split exact sequence
$$0 \to H^{R_g} \to H \to H/H^{R_g} \to 0.$$
This gives rise to a ring homomorphism
$$H^*(\Gamma_g,\Q)\tensor H^{R_g} \cong H^*(\Gamma_g,H^{R_g}) \to H^*(\Gamma_g,H),$$
whose cokernel may be identified with
$$H^*(\Gamma_g,H/H^{R_g}).$$
Since $H/H^{R_g}$ is a direct sum of non-trivial irreducible algebraic representations of $R_g$, the cokernel vanishes in degrees that are small compared to $g$ by Borel's vanishing theorem \cite[Theorem 4.4]{borel81}. Note that this entails the statement that
$H^{R_g} = H^{\Gamma_g}$ for $g$ large (set $*=0$). Thus, in the stable range, we obtain an isomorphism
$$H^*(\Baut{\partial}{W_{g,1}};\Q) \cong H^*(\Gamma_g,\Q)\tensor H^{\Gamma_g}.$$

This essentially recovers Theorem \cite[Theorem 1.3]{bm20}.
\end{rmk}

\begin{rmk} \label{rmk:extraneous}
The isomorphism \eqref{eq:unstable iso}, and in particular the collapse of the spectral sequence of \eqref{eq:fs}, is only obtained in the stable range (cohomological degrees below $g/2-2$) in \cite{bm20}. Moreover, the argument in \cite{bm20} depends on several extraneous ingredients. One step in the argument is to show that the map
\[\Baut{\partial}{W_{g,1}} \to B\Gamma_g\]
is injective on indecomposables in rational cohomology in the stable range \cite[Theorem 8.6]{bm20}. This is proved by establishing the stronger statement \cite[Theorem 8.2]{bm20} that
\[ B\Diff_\partial(W_{g,1}) \to B\Gamma_g\]
is injective on indecomposables in stable rational cohomology, which in turn relies on deep results on the stable cohomology of \(B\Diff_\partial(W_{g,1})\) due to Galatius--Randal-Williams \cite{grw14}, as well as on non-vanishing results for the coefficients of the Hirzebruch $L$-polynomials due to Berglund--Bergstr\"om \cite{bb18}. Our proof of Theorem \ref{thm:unstable ring iso} does not depend on these extraneous ingredients. In fact, Theorem \ref{thm:unstable ring iso} implies a strengthening of \cite[Theorem 8.6]{bm20}, see Corollary \ref{cor:strengthening} below.

Furthermore, many arguments of \cite{bm20} rely on deep results on almost algebraicity of finite-dimensional representations of $\Gamma_g$ due to Bass--Milnor--Serre \cite[\S16]{bms67} (see also ~\cite[\S1.3(9)]{serre79} and \cite[Appendix A]{bm20}). The existence of algebraic Lie models shows that the representations in question are in fact algebraic, which in particular removes the necessity to invoke such results.
\end{rmk}

\begin{scor} \label{cor:strengthening}
The ring homomorphism
\[ H^*(\Gamma_g,\Q) \to H^*(\Baut{\partial}{W_{g,1}};\Q)\]
is split injective for all $g$. In particular,
the induced map on indecomposables \[QH^*(\Gamma_g,\Q) \to QH^*(\Baut{\partial}{W_{g,1}};\Q)\]
is injective for all $g$.
\end{scor}



\subsection{Highly connected odd-dimensional manifolds}\label{sec:odd}
The geometry and rational homotopy theory of \((n-1)\)-connected \((2n+1)\)-dimensional manifolds has also been thoroughly studied (\emph{e.g.}~in Wall's classification \cite{wall67}), and we make some remarks on the odd-dimensional case in this section.

Let \(M\) be an \((n-1)\)-connected \((2n+1)\)-dimensional manifold (\(n > 8\)).
Using the techniques of \cite[\S3.5]{bm20}, one can show that the minimal Quillen model of \(M\) is given by
\[
    \Lbb = (\Lbb(\alpha_1, \ldots, \alpha_r, \alpha_1^\#, \ldots, \alpha_r^\#, \gamma), \delta \gamma = \omega)
\]
where the \(\alpha_i\) are a basis for \(\s^{-1}\widetilde{H}_n(M; \Q)\), the \(\alpha_i^\#\) are the dual basis of the space \(\s^{-1}\widetilde{H}_{n+1}(M;\Q)\) with respect to the intersection form, and
\[
    \omega = \sum_i [\alpha_i^\#, \alpha_i]
\]
is dual to the intersection form.
(So up to rational equivalence, such manifolds are completely classified by the torsion-free rank of \(H_n(M)\).)
We use the notation of the preceding section freely.

\begin{lemma}\label{lemma: R of odd manifolds}
Let \(M\) be an \((n-1)\)-connected \((2n+1)\)-dimensional manifold with \(n>1\).
The group \(\Aut^+\Lbb\) is isomorphic to the group \(\GL(H^n(M; \Q))\) of linear automorphisms of \(H^n(M; \Q)\).
In particular, it is reductive, wherefore \(\Aut^+\Lbb \cong \piauto{M_\Q} \cong R^+(M)\).
\end{lemma}
\begin{proof}
An automorphism of \(\Lbb\) is orientation-preserving precisely when it fixes \(\gamma\).
Thus an automorphism \(\phi \in \Aut^+\Lbb\) is uniquely determined by its action on \(\Lbb_{n-1} = \langle \alpha_1, \ldots, \alpha_r \rangle\): it preserves the element \(\omega = \delta \gamma\), so its action on the \(\alpha_i^\#\) is \(\omega\)-dual to the action on the \(\alpha_i\).
\end{proof}

The identification of the group \(\Gamma^+(M)\) is somewhat more involved than in the preceding case, and we will not elaborate on it here.
We remark that the closely related group
\[
    \Gamma^+_\Z(M) = \Aut^{\piauto{M}}(H_*(M; \Z))
\]
of orientation-preserving automorphisms of \emph{integral} homology that are realizable by homotopy automorphisms of \(M\), which has been studied by several authors, provides a ready substitute for \(\Gamma^+(M)\) in our results: the space \(\BSaut{}{M}\) is rationally equivalent to the homotopy orbit space \(\nerve{\lie g(M)}_{\mathrm{h} \Gamma^+_\Z(M)}\).
The group \(\Gamma^+_\Z(M)\) has been determined in many cases by Floer \cite{floer} in the course of his classification of the homotopy types of \((n-1)\)-connected \((2n+1)\)-dimensional Poincar\'e complexes.
See also Barden \cite{barden65} and Baues--Buth \cite{babu96} for the case \(n=2\), and Crowley--Nordstr\"om \cite{cronor19} for the case \(n=3\).

Just like in the case of even-dimensional manifolds, \(\lie g(M)\) enjoys a close connection to the derivations of the homotopy Lie algebra \(\pi_*(\Omega M) \otimes \Q\).
The following proposition is an easy consequence of the methods of \cite[\S 5.4]{bm20}, especially Theorem~5.9 therein.

\begin{prop}
Let \(M\) be an \((n-1)\)-connected \((2n+1)\)-dimensional manifold (\(n > 1\)) with \(\dim H^n(M; \Q) > 1\).
The algebraic Lie model \(\lie g(M)\) for the space $\Baut{u}{M}$
is formal as a dg Lie algebra of algebraic representations of \(R(M)\) and its homology is the graded Lie algebra
\[
    H_*(\lie g(M)) \cong \Der L / \ad L \langle 1 \rangle
\]
where \(L = \pi_*(\Omega M) \otimes \Q\) is the rational homotopy Lie algebra of \(M\), which admits the presentation
\[
    L \cong \Lbb(\alpha_1, \ldots, \alpha_r, \alpha^\#_1, \ldots, \alpha^\#_r)/(\omega)
\]
with \(\alpha_i\), \(\alpha^\#_i\), and \(\omega\) as above.
The action of \(R(M)\) is induced by the action on \(H_*(M; \Q)\).
\end{prop}

Thus, we obtain the following consequence of Corollary~\ref{cor:cohomology ring}.
\begin{sthm}
Let \(M\) be an \((n-1)\)-connected \((2n+1)\)-dimensional manifold with \(\dim H^n(M;\Q) > 1\).
There is an isomorphism of graded algebras
\[
    H^*(\BSaut{}{M}; \Q) \cong H^*(\Gamma_\Z^+(M), H^*_{\CE}(\lie g(M)))
\]
where \(\lie g = \Der L / \ad L \langle 1 \rangle\) is the truncated dg Lie algebra of outer derivations of \(L = \pi_*(\Omega M) \otimes \Q\), and \(\Gamma_\Z^+(M) = \Aut^{\piauto{M}}(H^*(M; \Z))\).
\end{sthm}

We conclude this section by making some remarks about the manifolds
\begin{align*}
    Z_g & = \#^g S^n \times S^{n+1},\\
    Z_{g, 1} & = Z_g \setminus \operatorname{int} D^{2n+1},
\end{align*}
for \(n>1\).
By Lemma~\ref{lemma: R of odd manifolds}, we have that
\[R^+(Z_g) \cong \GL(H_n(Z_g; \Q)) \cong \GL_g(\Q).\]
Let
\[
    \Gamma_\partial(Z_{g, 1}) \cong \GL^{\aut_\partial(Z_{g, 1})}(H_*(Z_{g, 1}; \Q))
\]
be the group of automorphisms of the rational homology of \(Z_{g, 1}\) which can be realized by a boundary-preserving homotopy automorphism of \(Z_{g, 1}\).
Each such homotopy automorphism can be extended by the identity on the interior of the removed disc to an orientation-preserving homotopy automorphism of \(Z_g\), yielding an injection
\begin{equation}\label{eqn: gamma partial}
    \Gamma_\partial(Z_{g,1}) \hookrightarrow \Gamma^+(Z_g).
\end{equation}

\begin{prop}\label{prop: Gamma of odd}
The homomorphism \eqref{eqn: gamma partial} is an isomorphism, and the group \(\Gamma^+(Z_g)\) corresponds to the subgroup \(\Aut(H_n(Z_g; \Z))\) of the group \[\GL(H_n(Z_g; \Q)) \cong R^+(Z_g).\]
\end{prop}
\begin{proof}

It is evident that \(\Gamma^+(Z_g)\) injects into \(\Aut(H_n(Z_g; \Z))\).
Both claims can be established simultaneously by showing that every automorphism of the group \(H_n(Z_g; \Z) \cong \Z^g\) may be realized by a homotopy automorphism of \(Z_g\) that fixes the given embedded disc \(D^{2n+1} \subset Z_g\) pointwise.
This is an easy consequence of the Hilton--Milnor theorem once one notes that \(Z_{g, 1} \simeq \bigvee^g S^n \vee \bigvee^g S^{n+1}\) and that \(Z_g\) is obtained from \(Z_{g, 1}\) by attaching a \((2n+1)\)-cell along
\[
    \omega = \sum_{i=1}^g [\alpha_i, \alpha^\#_i] \in \pi_{2n}(Z_{g,1}),
\]
where \(\alpha_i \in \pi_n(Z_{g, 1})\) are the classes of the \(S^n\)-factors of \(Z_{g, 1}\) and \(\alpha^\#_i \in \pi_{n+1}(Z_{g,1})\) are the classes of the \(S^{n+1}\)-factors, suitably indexed.
See \cite[\S 5]{grey19} where the more general case of connected sums of products of spheres of varying dimensions is investigated.
\end{proof}

When applied to the homotopy fiber sequence
\[
    \Gamma_\partial(Z_{g,1}) \hOrbits \aut_\partial(Z_{g,1}) \longrightarrow B\aut_\partial(Z_{g,1}) \longrightarrow B\Gamma_\partial(Z_{g,1}),
\]
our methods show that the \(\Gamma_\partial(Z_{g,1})\)-space \(\Gamma_\partial(Z_{g,1})\hOrbits\aut_\partial(Z_{g,1})\) admits an \(R^+(Z_g)\)-algebraic Lie model
\[
    \lie{g}_\partial(Z_{g,1}) = \left(\Der_\omega \Lbb(\alpha_1, \ldots, \alpha_g, \alpha^\#_1, \ldots, \alpha^\#_g) \right)\langle 1 \rangle
\]
where \(\Der_\omega\) indicates derivations that annihilate the element \(\omega = \sum_i [\alpha_i, \alpha^\#_i]\).
The action of \(R^+(Z_g)\) (and thus by restriction also the action of \(\Gamma_\partial(Z_{g,1})\)) comes from the action on \(H_*(Z_g; \Q)\).
Thus we obtain
\begin{sthm}
There is an isomorphism of graded algebras
\[
    H^*(B \aut_\partial(Z_{g,1});\Q) \cong H^*(\Gamma_\partial(Z_{g,1}); H^*_\CE(\lie{g}_\partial(Z_{g,1}))).\qedhere
\]
\end{sthm}
Using this result, the stable cohomology of $\Baut{\partial}{Z_{g,1}}$ is computed by Stoll \cite{stoll}.

\subsection{A counterexample} \label{sec:non-formal}
In all the examples discussed so far, the space \(X\) was always formal, the group $\piaut{X_\Q}$ was reductive (hence equal to $R(X)$) and isomorphic to the group of algebra automorphisms of $H^*(X;\Q)$, the group $\Gamma(X)$ was the group of automorphisms of $H_*(X;\Q)$ that are induced by a self-homotopy equivalence of $X$, and the dg Lie algebra $\lie g(X)$ was formal. In this section, we will discuss an example of a space $X$ that has none of these properties. In fact, $X$ will be the space that is discussed in \cite[Example~6.7]{gtht00}.

Let \(X\) be the space \(F \times K(\Z, 3) \times K(\Z, 6)\), where \(F\) is the homotopy fiber of the cup product
\[
    K(\Z, 3) \times K(\Z, 3) \xrightarrow{\smile} K(\Z, 6).
\]
This is evidently a finite Postnikov section with finitely generated homotopy groups, and its Sullivan model is the cdga
\[
    \Lambda = \big(\Lambda(x, y, z, u, w), d\big),
\]
whose underlying graded commutative algebra is freely generated by cocycles \(x, y, z\) in degree 3, a cocycle \(w\) in degree \(6\), and an element \(u\) in degree \(5\) which satisfies \(du = yz\).

\begin{rmk}
Note that \(X\) is indeed not rationally formal: the homomorphism
\[
    \Aut \Lambda \longrightarrow \Aut_{alg}H^*(\Lambda)
\]
is not surjective.
For any \(\lambda \in \Q^\times\), there is a graded algebra automorphism \(\phi\) of \(H^*(\Lambda)\) given by \(\phi(a) = \lambda^{\lvert a \rvert}a\).
But \(\phi\) affords no lift to \(\Aut \Lambda\) unless \(\lambda = 1\): any such lift \(\widetilde{\phi}\) would have to have \(\widetilde{\phi}(u) = \lambda^6 u\) so as to commute with \(d\).
But then \(\widetilde{\phi}(uy) = \lambda^9 uy\) whereas \(\phi(uy) = \lambda^8 uy\).
\end{rmk}

\begin{prop}\label{prop: non-formal R}
The group \(R(X)\) is isomorphic to
\[
    \GL_1^2 \times \GL_2.
\]
A section of the surjection \(\Aut \Lambda \to R(X)\) is given by letting the two \(\GL_1\) factors act by scalar multiplication on \(x\) and on \(w\), respectively, and by letting \(A \in \GL_2\) act via the standard representation on the linear span of \(y\) and \(z\), and by scalar multiplication by \(\det A\) on \(u\).
The subgroup \(\Gamma(X)\) of \(R(X)\) corresponds to
\[
    \left(\Z^\times\right)^2 \times \GL^\Sigma_2(\Z)
\]
under this isomorphism.
\end{prop}

\begin{rmk}
In \cite[Example~6.7]{gtht00} the authors construct two automorphisms \(\phi, \psi \in \Aut \Lambda\). \(\phi\) is the identity on all generators but \(w\), where \(\phi(w) = w+xy\), and \(\psi\) is the identity on all generators but \(x\), where \(\psi(x) = x+z\).
As observed in \emph{loc. cit.}, these two automorphisms span a copy of \(\Z^2\) in \(\Aut^h \Lambda \cong \piaut{X_\Q}\), but there exists no pair \((\phi', \psi')\) of homotopic automorphisms (or even automorphisms with the same effect on \(H^*(\Lambda)\) as \(\phi\) and \(\psi\), respectively) which would already commute in \(\Aut \Lambda\).
This shows that we cannot in general expect to be able to lift the action of \(\piaut{X_\Q}\) on \(\pi_*X\otimes\Q\) or on \(H^*(X; \Q)\) to an action on the minimal Sullivan model.
It is clear from the proof of Proposition~\ref{prop: non-formal R} that both \(\phi\) and \(\psi\) lie in the kernel of \(\Aut \Lambda V \to R(X)\), so they do not obstruct the existence of an action of \(R(X)\) on \(\Lambda\).
\end{rmk}

\begin{prop}\label{prop: non-formal g}
The dg Lie algebra \(\lie g(X)\) is not formal: the Massey product
\[
    \left\langle y\frac{\partial}{\partial x}, x\frac{\partial}{\partial u}, z\frac{\partial}{\partial w}\right\rangle
\]
is non-trivial.
\end{prop}

The rest of the section is devoted to proving Propositions \ref{prop: non-formal R} and \ref{prop: non-formal g}.

We write \(V\) for the graded vector subspace of \(\Lambda\) which is spanned by the generators \(x, y, z, u\) and \(w\), and we write \(\Lambda V\) for the underlying graded commutative algebra of \(\Lambda\).
The group \(\GL(V)\) of graded linear automorphisms of \(V\) evidently injects into the group \(\Aut \Lambda V\) of graded commutative automorphisms of \(\Lambda V\), and we will abuse notation by referring to its image by \(\GL(V)\) as well.

\begin{lemma}\label{lemma: uniradical of Aut Lambda V}
The automorphism group \(\Aut\Lambda V\) of the underlying graded commutative algebra is a semidirect product
\[
    \GL(V) \ltimes \Hom(V^6, \Lambda^2 V^3)
\]
where the abelian group \(\Hom(V^6, \Lambda^2 V^3)\) injects into \(\Aut \Lambda V\) by sending \(f \colon V^6 \to \Lambda^2 V^3\) to the automorphism that is the identity on all generators except \(w\) and that takes \(w\) to \(w + f(w)\).
\end{lemma}
\begin{proof}
There is an obvious identification of \(V\) with the indecomposables \(Q\Lambda V = I/I^2\), where \(I = (\Lambda V)^{>0}\) is the augmentation ideal.
This yields a retraction \(\Aut \Lambda V \twoheadrightarrow \GL(Q\Lambda V) \cong \GL(V)\).
Let \(\Aut_1\Lambda V\) be the kernel of this retraction.
Every \(\phi \in \Aut_1\Lambda V\) fixes \(x, y, z\) and \(u\), and is determined by the value \(\phi(w)-w\in\Lambda^2V^3\subset(\Lambda V)^6\).
It can be seen that there is an isomorphism
\[
    \Aut_1\Lambda V \cong \Hom(V^6, \Lambda^2 V^3)
\]
of groups with a \(\GL(V)\)-action which sends \(\phi \in \Aut_1\Lambda V\) to the homomorphism \(w \mapsto \phi(w) - w\).
\end{proof}

The group \(\Aut \Lambda\) of dgla automorphisms of \(\Lambda\) is the isotropy subgroup of the differential \(d\) inside \(\Aut \Lambda V\).
The subgroup \(\Aut_1 \Lambda V\) of the preceding lemma clearly commutes with \(d\), so it suffices to determine \(\Aut \Lambda \cap \GL(V)\).

\begin{lemma}\label{lemma: differential automs}
Let \(W = \langle y, z \rangle \subset V^3\) be the \(\Q\)-linear span of \(y\) and \(z\).
An automorphism \(\phi \in \GL(V)\) commutes with \(d\) if and only if \(W\) is \(\phi\)-stable, and \(\phi(u) = (\det \phi\rvert_W) \cdot u\).
\end{lemma}
\begin{proof}
Note that the line \(\Lambda^2 W\) is precisely the span of \(yz\), and also the image of the differential \(d \colon (\Lambda V)^5 \hookrightarrow (\Lambda V)^6\).
Hence \(\phi\) commutes with \(d\) if and only if it carries \(\Lambda^2 W\) to itself and it acts on \(\Lambda^2 W\) and \((\Lambda V)^5\) by the same scalar.
But \(\phi(\Lambda^2 W) = \Lambda^2 \phi(W)\) is equal to \(\Lambda^2 W\) if and only if \(\phi(W) = W\).
If this is the case, then \(\phi\) scales \(\Lambda^2 W\) by \(\det \phi\rvert_W\).
\end{proof}

\begin{proof}[Proof of Proposition~\ref{prop: non-formal R}]
The unipotent radical of \(\Aut \Lambda V\), which we identified in Lemma~\ref{lemma: uniradical of Aut Lambda V}, is contained in \(\Aut \Lambda\), and hence in its unipotent radical.
Thus the maximal reductive quotient of \(\Aut \Lambda\) agrees with that of \(\Aut \Lambda \cap \GL(V)\).
From Lemma~\ref{lemma: differential automs} it is easily seen that the unipotent radical of \(\Aut \Lambda \cap \GL(V)\) consists of the automorphisms that fix all generators except \(x\) and that send \(x\) to \(x\) plus an element of \(W\), and the maximal reductive quotient is as described in the statement of the proposition.

To determine \(\Gamma(X)\), recall that \(X\) is the homotopy fiber of the \(6\)th integral cohomology class
\[
    K(\Z, 3)^3 \times K(\Z, 6) \longrightarrow K(\Z, 6)
\]
given by the cup product of the fundamental classes of two of the \(K(\Z, 3)\) factors.

The integral cohomology ring of \(X\) agrees up to dimension 6 with
\[
    \Lambda_\Z(\bar x, \bar y, \bar z, \bar w)/(\bar y \bar z, 2\bar x^2, 2\bar y^2, 2\bar z^2)
\]
where the generators \(\bar x, \bar y\) and \(\bar z\) live in degree 3, \(\bar w\) lives in degree 6, and each generator maps to the corresponding element of \(H^*(\Lambda)\cong H^*(X; \Q)\) under the change-of-coefficient homomorphism.
Explicitly, \(H^3(X) \cong \Z^3\) is spanned by \(\bar x, \bar y, \bar z\), and \(H^6(X) \cong \Z^3 \oplus (\Z/2)^3\) is spanned by \(\bar x \bar y, \bar x \bar z, \bar w\) and \(\bar x^2, \bar y^2, \bar z^2\).

Thus the homotopy class of an endomorphism \(f \colon X \to X\) is uniquely determined by the quadruple \((f^*\bar x, f^*\bar y, f^*\bar z, f^*\bar w)\in H^3(X)^3\times H^6(X)\), and a quadruple \((\bar x', \bar y', \bar z', \bar w')\) is induced by an endomorphism if \(\bar y'\smile \bar z'=0\).

Thus an element \((\lambda, A, \mu) \in \GL_1 \times \GL_2 \times \GL_1 \cong R(X)\) can be realised by a homotopy automorphism of \(X\) if and only if \(\lambda, \mu \in \Z\), the matrix \(A\) has integer entries, and each row of \(A\) contains at least one even number (since \((a \bar y + b \bar z)(c \bar y + d \bar z) = ac \bar y^2 + bd \bar z^2 \in H^6(X)\) is non-zero if \(ac\) or \(bd\) are odd).
\end{proof}

\begin{proof}[Proof of Proposition~\ref{prop: non-formal g}]
We begin by describing in some detail the structure of the dg Lie algebra \(\lie g(X)\).
In positive degrees, \(\lie g(X)\) is just the full dg Lie algebra of derivations of the free cdga \(\Lambda V\), and it is easy to list the basis in each degree.

In degree \(0\), \(\lie g(X)\) is the Lie algebra of the unipotent radical of \(\Aut \Lambda\).
This radical sits in a (split) short exact sequence
\[
    1 \longrightarrow \Hom(V^6, \Lambda^2 V^3) \longrightarrow \Aut_u \Lambda \longrightarrow \Hom(\langle x \rangle, W) \longrightarrow 1
\]
where the left-hand term is the unipotent radical of \(\Aut \Lambda V\) and the right-hand term is the unipotent radical of \(\Aut \Lambda \cap \GL(V)\).
From this, a basis for \(\lie g_0(X)\) can be read off.
We list bases in all degrees in the table below.
\begin{center}
\begin{tabular}{c|l}
    degree & basis\\
    \hline
    0 & \(y \frac{\partial}{\partial x},\; z \frac{\partial}{\partial x},\; xy \frac{\partial}{\partial w},\; xz\frac{\partial}{\partial w},\; yz\frac{\partial}{\partial w}\)\\[1ex]
    1 & \(u\frac{\partial}{\partial w}\)\\[1ex]
    2 & \(x\frac{\partial}{\partial u},\; y\frac{\partial}{\partial u},\; z\frac{\partial}{\partial u}\)\\[1ex]
    3 & \(x\frac{\partial}{\partial w},\; y\frac{\partial}{\partial w},\; z\frac{\partial}{\partial w},\; \frac{\partial}{\partial x},\; \frac{\partial}{\partial y},\; \frac{\partial}{\partial z}\)\\[1ex]
    5 & \(\frac{\partial}{\partial u}\)\\[1ex]
    6 & \(\frac{\partial}{\partial w}\)
\end{tabular}
\end{center}
The differential \(\delta\) is given by
\[
    u\frac{\partial}{\partial w} \longmapsto yz\frac{\partial}{\partial w},\qquad
    \frac{\partial}{\partial y} \longmapsto -z\frac{\partial}{\partial u},\qquad
    \frac{\partial}{\partial z} \longmapsto y\frac{\partial}{\partial u};
\]
it vanishes on all the other generators.

Note that
\[
    \left[ y\frac{\partial}{\partial x}, x\frac{\partial}{\partial u} \right] = y\frac{\partial}{\partial u} = \delta \left(\frac{\partial}{\partial z}\right)
\]
while the other two pairs bracket to zero.
Thus the Massey product contains
\[
    \left[ \frac{\partial}{\partial z}, z \frac{\partial}{\partial w} \right] = \frac{\partial}{\partial w}.
\]
The indeterminacy is equal to the subset
\[
    \left[ y\der{x}, H_6(\lie g(X)) \right] + \left[ x\der{u}, H_4(\lie g(X)) \right] + \left[ z\der{w}, H_3(\lie g(X)) \right]
\]
of \(H_6(\lie g(X))\).
This is easily seen to only contain \(0\).
Thus the Massey product is indeed non-trivial as \(\der{w}\) is a non-trivial cycle, obstructing formality of \(\lie g(X)\).
\end{proof}

\printbibliography
\end{document}